\documentclass[11pt]{article}
\usepackage[margin=1.1in]{geometry}
\usepackage{amssymb}
\usepackage{amsmath}
\usepackage{amsthm}
\ifdefined\kanjiskip
  \usepackage[dvipdfmx]{xcolor}
  \usepackage[dvipdfmx]{graphicx}
\else
  \usepackage{xcolor}
  \usepackage{graphicx}
\fi
\usepackage{tikz}

\renewcommand{\proofname}{\textbf{Proof}}
\makeatletter
\renewenvironment{proof}[1][\proofname]{\par
  \pushQED{\qed}%
  \normalfont \topsep6\p@\@plus6\p@\relax
  \trivlist
  \item[\hskip\labelsep\bfseries #1\@addpunct{.}]\ignorespaces
}{\popQED\endtrivlist\@endpefalse}
\makeatother
\newtheorem{theorem}{Theorem}[section]
\newtheorem{proposition}[theorem]{Proposition}
\newtheorem{lemma}{Lemma}[section]
\newtheorem{corollary}{Corollary}[section]
\theoremstyle{definition}
\newtheorem{example}{Example}
\newtheorem{remark}{Remark}
\newcommand{\Na}{\mathcal{N}}
\newcommand{\esup}{\operatorname*{ess\,sup}}
\newcommand{\N}{\mathbb{N}}
\newcommand{\F}{\mathcal{F}}
\newcommand{\T}{\mathcal{T}}

\begin{document}
\title{Monotone Multiple Stopping and Last-Success Problems}

\author{Katsunori Ano\\Bunkyo Gakuin University}
\date{October 8, 2026}

\maketitle

\begin{abstract}
An optimal stopping problem is monotone when its stopping region is preserved under subsequent
	transitions, so that the one-step look-ahead stopping time is optimal.
We develop a finite-horizon discrete-time theory of multiple stopping that separates
	general recursive dynamic programming from this stronger monotone structure.
Extending the classical martingale-system approach, we construct a recursively defined
	optimal vector of stopping times without assuming monotonicity.
If monotonicity holds at every exercise level, this vector coincides with the
	recursive vector of one-step look-ahead stopping times.
A bounded counterexample shows that monotonicity of the original single-stopping problem
	alone is insufficient, whereas monotonicity propagates to every level when the level-one
	one-step look-ahead function is deterministic.
We apply the theory to multiple-selection last-success problems.
For the classical independent odds setting, a continuous-time Poisson embedding
	gives an alternative proof of the multiple-stopping lower bound and clarifies
	the associated threshold constants.
We also examine random horizons, Markov-dependent trials, and an unknown success probability.
\end{abstract}

\medskip
\noindent {\itshape Keywords:} optimal stopping; multiple stopping; monotone; 
odds problem; last-success problem.
    
\smallskip
\noindent {\itshape 2020 Mathematics Subject Classification:} 60G40; 62L15.

\section{Introduction}
Optimal stopping problems are often solved by identifying a stopping region
	and proving that the first entrance time into that region is optimal.  
In a monotone stopping problem in the sense of Chow and Robbins~\cite{chowrobbins}, 
	the stopping region is preserved by subsequent transitions, and this stability reduces
	the optimal rule to the one-step look-ahead rule.  
The classical theory was developed through the martingale-system approach of 
	Snell~\cite{snell1952} and the monotone stopping
	theory of Chow and Robbins~\cite{chowrobbins}; standard accounts include
	Chow, Robbins and Siegmund~\cite{chowrobbinssiegmund}, Ferguson~\cite{fergusonbook},
	and Neveu~\cite{neveu}.

The corresponding question for multiple stopping is more delicate.
An early discrete-time contribution is Haggstrom~\cite{haggstrom1967}, who studied
	optimal sequential procedures when more than one stop is required.
In continuous-time Markovian models this issue is usually expressed through a system of variational 
	inequalities or free-boundary problems; see, for example, Carmona and
	Dayanik~\cite{carmonadayanik}.  
Even when only two exercise rights are available, proving the geometry and regularity of the 	
	stopping regions may be difficult.

This article is based on the multiple-stopping and last-success parts of
	the author's Japanese preprint~\cite{ano2026}, reorganized here as a self-contained research article.  
The best-choice applications from the longer manuscript
	are treated separately in the companion paper~\cite{ano2026best}.
We study the finite-horizon discrete-time multiple-stopping problem
	\begin{align*}
		\sup_{\vec\tau^{[m]}\in\bar{\mathcal T}_N^{[m]}}
			E[X_{\vec\tau^{[m]}}],
			\qquad
			X_{\vec\tau^{[m]}}:=\sum_{j=1}^m X_{\tau^{[j]}},
	\end{align*}
	where $X=(X_n)_{n\in\mathbb N}$ is a nonnegative adapted process and
	$\bar{\mathcal T}_N^{[m]}$ is the admissible class of vectors of at most $m$
	stopping times introduced in Section~\ref{subsec:multiple-stopping}.  
Our first objective is to extend the martingale-system theory from one stopping time to the 
	multiple-stopping setting.  
Our second objective is to determine when the resulting optimal stopping times admit a monotone, 
	one-step look-ahead representation.

The main structural results are as follows.  
Under the integrability condition \eqref{A1}, the recursively augmented reward processes inherit the required
	integrability at every exercise level.  
This yields a recursive optimal vector of stopping times without a monotonicity assumption
	(Theorem~\ref{th22}).  
If the stopping problem is monotone at every level, the optimal vector is the corresponding 
	recursive vector of one-step look-ahead stopping times (Theorem~\ref{monotonetheorem2}).  
We show, however, that monotonicity of the original single-stopping problem does not 
	by itself imply levelwise monotonicity: a bounded finite-horizon counterexample is provided,
	in which the one-step look-ahead vector is strictly suboptimal.
We also derive a recursive formula for the one-step look-ahead functions
	(Corollary~\ref{corollary21}) and prove that levelwise monotonicity propagates
	when the single-level one-step look-ahead function is deterministic
	(Proposition~\ref{propdet}).
A statewise threshold structure does not by itself imply the pathwise
	monotonicity required in Theorem~\ref{monotonetheorem2};
	see Section~\ref{sec:markov}.

Section~\ref{sec:applications} applies the general theory to last-success problems, beginning
	with Bruss' classical odds problem~\cite{bruss2000} and its random-horizon extension, and then
	turning to Markov-dependent and Bayesian models.  
For independent Bernoulli trials we give an alternative Poisson--Bellman proof of the known 
	multiple-stopping lower bound.  
For a uniform random horizon, logarithmic time scaling leads to a Volterra free-boundary system, 
	threshold ratios converging to $e^{-\widehat C_k}$, and limiting winning probability
	$\sum_{k=1}^m \widehat C_k e^{-\widehat C_k}$.
For Markov-dependent trials, persistent success clusters rule out a universal analogue of the
	independent lower bound, whereas rare-entry and vanishing-persistence assumptions recover the same
	Poisson limits.  
Finally, with an unknown Bernoulli success probability, Bayesian updating destroys
	the required pathwise monotonicity in general; the optimal multiple-stopping rule is therefore
	obtained by backward induction rather than by a one-step look-ahead argument.

The paper thus has two complementary aims: to isolate the precise structural assumptions under
	which monotone multiple stopping reduces to recursive one-step look-ahead rules, and to show how this
	structure interacts with classical and nonclassical last-success models.

\section{Optimal Multiple Stopping}\label{sec:optimal-multiple-stopping}
\subsection{Single stopping}\label{subsec:single-stopping}
This subsection briefly reviews the standard theory of optimal stopping;
	it is included for completeness and to fix the notation.
Throughout, $\N=\{0,1,2,\ldots\}$.
Let $(\Omega, \F, P)$ be a complete probability space 
	with a filtration $\mathbb{F}=\{\F_n\}_{n\in \mathbb{N}}$, 
	where $\F_0$ is assumed to be $P$-trivial.
Let $X=(X_n)_{n\in\mathbb{N}}$ be a sequence of nonnegative 
	$\mathbb{F}$-adapted random variables satisfying, for some $p>1$, 
		\begin{align}\tag{A1}\label{A1}
		E\left[\left(\sup_{n\in\N}X_n\right)^p\right]<\infty.
	\end{align}
Throughout this section we write $\bar{X}=\sup_{n\in\N}X_n$.
The assumption \eqref{A1} implies in particular that $E[\bar{X}]<\infty$, 
	so that the family $(X_n)_{n\in\N}$ is uniformly integrable.
For the single-stopping problem of this subsection, the integrability
	of $\bar{X}$ alone would be sufficient.  
On the finite horizon considered in Sections~\ref{subsec:multiple-stopping}--\ref{sec:markov},
	integrability of $\bar X$ is also inherited by every levelwise reward
	sequence, but only with constants that grow with the horizon $N$.  
We retain the stronger $\mathcal{L}^p$ condition in \eqref{A1} because it yields
	levelwise bounds that are uniform in $N$; see Lemma~\ref{lem21} and
	\eqref{lem21uniform}.
The random variable $X_n$ is regarded as the reward obtained by stopping at time $n$.
Let $N\in\N$ be a fixed horizon; in the infinite-horizon case we set $N=+\infty$. 
Set $\F_{\infty}=\sigma(\cup_{n\in\N}\F_n)$ and 
	$X_\infty=\limsup_{n\to\infty}X_n$.
We treat the finite- and infinite-horizon stopping problems 
	simultaneously by assuming that $X_{\cdot}\equiv 0$ on $(N,\infty)$ when $N<\infty$. 
Denote by $\T$ the set of all $\mathbb{F}$-stopping times with
	values in $\mathbb N\cup\{\infty\}$.  If $N<\infty$, let
	$\T^N$ be the stopping times valued in $\{0,\ldots,N\}$; if
	$N=\infty$, set $\T^N:=\T$.  
For a given $\sigma\in\T^N$, let $\T_{\sigma,N}$, or $\T_\sigma$ for
	short, be the stopping times $\tau\in\T^N$ satisfying
	$\tau\ge\sigma$ a.s.  We write $\T_{n,N}:=\T_{\sigma,N}$
	when $\sigma\equiv n$, and $\T_{0,N}=\T^N$.
The optimal stopping problem is to evaluate the maximal expected reward
	$\sup_{\tau\in\T^N}E[X_\tau]$
	and to find, or to characterize, an optimal stopping time which attains the supremum.
The supremum of an uncountable family of random variables may fail 
	to be a random variable. 
We avoid this difficulty by introducing the family of
	random variables defined, for $0\le n\le N$ when $N<\infty$ and for
	$n\ge0$ when $N=\infty$, by
	\begin{align*}
		V_n=\esup_{\tau\in\T_{n,N}}E[X_\tau\mid\F_n].
	\end{align*}
	When $N<\infty$, set $V_n:=0$ for $n>N$.  
The essential supremum always exists; it is called the Snell envelope of $(X_n)_{n\in\N}$
	and is sometimes referred to as the value function. 
From the integrability condition \eqref{A1}, the random variable $\bar{X}$ 
	is integrable and $X_\tau\leq \bar{X}$ a.s.\ for every $\tau \in\T_{0,N}$, 
	where we use the convention $X_\infty=\limsup_{n\to\infty}X_n\le \bar{X}$. 
With the convention $V_\infty:=X_\infty$, it follows that 
	$V_\tau\leq E[\bar{X}\mid\F_\tau]$ a.s.\ for every 
	$\tau \in\T_{0,N}$, 
	which proves that the Snell envelope $(V_n)_{n\in\N}$ 
	is bounded in $\mathcal{L}^1$ and, being dominated by the uniformly 
	integrable family $(E[\bar{X}\mid\F_n])_{n\in\N}$, is itself uniformly integrable. 
By the same domination, the family $(V_\tau)_{\tau\in\T_{0,N}}$ is also uniformly integrable. 	
Under assumption \eqref{A1}, the theory of optimal stopping 
	(see, e.g., Chapter~VI of Neveu~\cite{neveu}) yields the following properties. 
\begin{enumerate}
	\item[(a)] $V_n$ satisfies the dynamic programming equation (optimality equation, Bellman equation), 
		\begin{align*}
			V_n=\max\{X_n,E[V_{n+1}\mid\F_n]\}\quad\text{a.s.},
		\end{align*}
		for $0\le n<N$ when $N<\infty$, and for every $n\ge0$ when
		$N=\infty$, with terminal condition $V_N=X_N$ in the finite-horizon case.
	\item[(b)] For every $n$, $E[V_n]=\sup_{\tau\in\T_{n,N}}E[X_\tau]$.
	\item[(c)] $(V_n)_{n\in\N}$ is the smallest supermartingale 
		dominating 
		$(X_n)_{n\in\N}$; moreover, it is uniformly integrable. Indeed, by \eqref{A1}, 
		$0\le V_n\leq E[\bar{X}\mid\F_n]$ for every $n\in\N$, 
		and the family $(E[\bar{X}\mid\F_n])_{n\in\N}$ is uniformly integrable. 
	\item[(d)] Let 
		\begin{align*}
			\tau^*=\min\{n\in\N:V_n=X_n\}, 
		\end{align*} 
		with the conventions $\min\emptyset=\infty$ and $V_\infty=X_\infty$.
The stopped sequence $(V_{\tau^*\land n})_{n\in\N}$ is a uniformly 
		integrable martingale. 
	\item[(e)] $\lim_{n\to\infty}V_n=\limsup_{n\to\infty}X_n$ a.s.
	\item[(f)] The stopping time $\tau^*$ is optimal for the single-stopping problem, 
		as follows from the next theorem.
\end{enumerate}
\begin{theorem}[Neveu~\cite{neveu}]
	\label{th21}
	The stopping time $\tau\in \T_{0,N}$ is optimal in the sense of 
	$E[X_{\tau}]=\sup_{\sigma\in\T_{0,N}}E[X_\sigma]$ if and only if 
	(i) $V_{\tau}=X_{\tau}$ a.s., and 
	(ii) the stopped sequence $(V_{\tau\land n})_{n\in\N}$ is a martingale.
\end{theorem}
For the proof, see Chapter~VI of Neveu~\cite{neveu}.
Although $\tau^*$ is characterized abstractly as the first time at which
	the Snell envelope meets the reward process, this characterization does not
	by itself provide an explicit stopping rule.
To obtain a practically useful rule, one must analyze the structure of the
	value process through the dynamic programming equation.
\subsection{Multiple stopping}\label{subsec:multiple-stopping}
Throughout Sections~\ref{subsec:multiple-stopping}--\ref{sec:markov}, 
	the horizon $N$ is finite and $1\le m\le N+1$.
The infinite-horizon multiple-stopping problem requires additional
	attainment conditions and is not treated here.
Let $\partial:=N+1$ be a cemetery time, set
$\F_{\partial}:=\F_N$, $X_{\partial}:=0$, and define the successor map
	\begin{align*}
		 s(n)=n+1\quad(0\le n\le N),\qquad s(\partial)=\partial.
	\end{align*}
On $\{0,1,\ldots,N,\partial\}$, write $a\prec_{\partial}b$ if either
	$a,b\le N$ and $a<b$, or $a\le N$ and $b=\partial$, or
	$a=b=\partial$.
Denote by $\bar{\T}^{[m]}_N$ the set of all stopping-time vectors
	$\vec\tau^{[m]}=(\tau^{[m]},\tau^{[m-1]},\ldots,\tau^{[1]})$ satisfying
	\begin{align*}
 		\tau^{[m]}\prec_{\partial}\tau^{[m-1]}\prec_{\partial}\cdots
 		\prec_{\partial}\tau^{[1]}.
	\end{align*}
Repeated occurrences of $\partial$ represent unused stopping rights.
The finite-horizon multiple-stopping problem is
	\begin{align*}
		\sup_{\vec\tau^{[m]}\in\bar{\T}^{[m]}_N}
		 E[X_{\vec\tau^{[m]}}],	\qquad
 			X_{\vec\tau^{[m]}}:=\sum_{i=1}^m X_{\tau^{[i]}}.
	\end{align*}
Throughout the article, the superscript $[i]$ denotes the number of
	remaining stopping rights.
Set $X_n^{[1]}=X_n$ and $X_{\partial}^{[i]}=V_{\partial}^{[i]}=0$ for every $i\ge1$.
The processes $X^{[i]}$ and $V^{[i]}$ are defined alternately in $i$:
	for $i=2,\ldots,m$ and $n=0,\ldots,N$, let
	\begin{align}\label{xnm}
     		X_n^{[i]}&=X_n+E[V_{s(n)}^{[i-1]}\mid\F_n],
	\end{align}
	the maximal expected total reward obtained by exercising one right at
	time $n$ with $i$ stopping rights remaining, and, for $i=1,\ldots,m$,
	\begin{align*}
     		V_n^{[i]}&=\esup_{\tau\in\T_{n,N}^{\partial}}
        		 E[X_\tau^{[i]}\mid\F_n],\qquad n=0,\ldots,N,
	\end{align*}
	where $\T_{n,N}^{\partial}$ is the set of stopping times taking values in
	$\{n,n+1,\ldots,N,\partial\}$.
We also use $X^{[0]}_\cdot\equiv V^{[0]}_\cdot\equiv0$.
Since the rewards are nonnegative, induction on $i$ gives
	$X_n^{[i]}\ge X_n^{[i-1]}$ and $V_n^{[i]}\ge V_n^{[i-1]}$ for all $n$.

\begin{lemma}\label{lem21}
If the reward process $(X_n)_{n\in\N}$ satisfies \eqref{A1} for some $p>1$,
	then, for each $i=1,\ldots,m$,
	\begin{align*}
 		E\left[\left(\max_{0\le n\le N}X_n^{[i]}\right)^p\right]<\infty.
	\end{align*}
Thus every reward sequence $(X_n^{[i]})_{0\le n\le N}$, $i=1,\ldots,m$, satisfies the same
	$\mathcal L^p$-integrability condition on the finite horizon.
\end{lemma}
\begin{proof}
Let $\bar X_N=\max_{0\le n\le N}X_n$ and
	$\bar X_N^{[i]}=\max_{0\le n\le N}X_n^{[i]}$.
For $i=1$, $\bar X_N^{[1]}=\bar X_N\le\bar X$, so that
	$\|\bar X_N^{[1]}\|_p<\infty$ by \eqref{A1}, where $\|Y\|_p=(E[|Y|^p])^{1/p}$
	denotes the $\mathcal L^p$-norm.
Assume inductively that $\|\bar X_N^{[i-1]}\|_p<\infty$ and set
	$M_n=E[\bar X_N^{[i-1]}\mid\F_n]$, $n=0,\ldots,N$.
Since $X_\tau^{[i-1]}\le\bar X_N^{[i-1]}$ a.s.\ 
	for every $\tau\in\T_{n,N}^\partial$, we have $V_n^{[i-1]}\le M_n$ a.s.\ 
	for every $n\le N$; hence, by the martingale property of $(M_n)$, for $n<N$,
	\begin{align*}
		 E[V_{s(n)}^{[i-1]}\mid\F_n]\le E[M_{s(n)}\mid\F_n]=M_n,
	\end{align*}
	while for $n=N$ the left-hand side equals
	$E[V_{\partial}^{[i-1]}\mid\F_N]=0$.
Hence $X_n^{[i]}\le X_n+M_n$ for every $n\le N$
	(for $n=N$ the conditional expectation term vanishes and $M_N\ge0$),
	and taking the maximum over $n$ gives
	\begin{align*}
 		\bar X_N^{[i]}\le \bar X_N+\max_{0\le n\le N}M_n.
	\end{align*}
Minkowski's inequality then yields
	\begin{align*}
 		\|\bar X_N^{[i]}\|_p\le \|\bar X_N\|_p+\Bigl\|\max_{0\le n\le N}M_n\Bigr\|_p.
	\end{align*}
Since $\bar X_N^{[i-1]}$ is $\F_N$-measurable, $(M_n)_{n=0}^N$ is a
	nonnegative martingale with terminal value $M_N=\bar X_N^{[i-1]}$,
	and Doob's $\mathcal L^p$ maximal inequality gives
	\begin{align*}
 		\Bigl\|\max_{0\le n\le N}M_n\Bigr\|_p\le \frac{p}{p-1}\,\|M_N\|_p
 		=\frac{p}{p-1}\,\|\bar X_N^{[i-1]}\|_p.
	\end{align*}
Combining the two estimates, we obtain
	\begin{align*}
 		\|\bar X_N^{[i]}\|_p\le \|\bar X_N\|_p+\frac{p}{p-1}\|\bar X_N^{[i-1]}\|_p<\infty,
	\end{align*}
	where the right-hand side is finite by \eqref{A1} and the induction
	hypothesis.
The assertion follows by induction.
Since $\|\bar X_N\|_p\le\|\bar X\|_p$, iterating the last estimate
	gives, for $i=1,\ldots,m$,
	\begin{align}\label{lem21uniform}
		\|\bar X_N^{[i]}\|_p\le\sum_{k=0}^{i-1}\Bigl(\frac{p}{p-1}\Bigr)^k\|\bar X\|_p ,
	\end{align}
	a bound that does not depend on the horizon $N$.
\end{proof}
 
Lemma~\ref{lem21} allows the standard finite-horizon Snell-envelope
	results to be applied at every level $i=1,\ldots,m$.  
For a stopping time $\sigma$, let
	$\T_{\sigma}^{[i]}$ denote the set of vectors
	$\vec\tau^{[i]}=(\tau^{[i]},\ldots,\tau^{[1]})\in\bar\T_N^{[i]}$ such that
	$\tau^{[i]}\succeq_{\partial}\sigma$ a.s., where
	$\succeq_{\partial}$ is the reverse of $\prec_{\partial}$ with equality
	allowed.  
The levelwise conclusions are as follows.
\begin{enumerate}
	\item[(a')] For $n=0,\ldots,N$,
		\begin{align*}
			V_n^{[i]}&=\max\{X_n^{[i]},E[V_{s(n)}^{[i]}\mid\F_n]\}\quad\text{a.s.},
		\end{align*}
		with $V_{\partial}^{[i]}=0$.
	\item[(b')]
		$E[V_n^{[i]}]=\sup_{\vec\tau^{[i]}\in\T_n^{[i]}}E[X_{\vec\tau^{[i]}}]$; 
		the case $n=0$ is proved in Theorem~\ref{th22}, and the general case follows 
		by the same argument.
	\item[(c')] $(V_n^{[i]})_{0\le n\le N}$ is the smallest supermartingale
		dominating $(X_n^{[i]})_{0\le n\le N}$.
	\item[(d')] Define $\vec\tau_*^{[m]}=(\tau_*^{[m]},\ldots,\tau_*^{[1]})$
		recursively by
		\begin{align*}
			\tau_*^{[m]}&=\min\{0\le n\le N:V_n^{[m]}=X_n^{[m]}\},\\
		 	\tau_*^{[i]}&=\begin{cases}
 					\min\{s(\tau_*^{[i+1]})\le n\le N:
      			 		V_n^{[i]}=X_n^{[i]}\},&\tau_*^{[i+1]}<\partial,\\
 					\partial,&\tau_*^{[i+1]}=\partial,
 				\end{cases}
 				\qquad i=m-1,\ldots,1,
		\end{align*}
		where the minimum of an empty set is $\partial$.
		Put $s(\tau_*^{[m+1]}):=0$.
		For every $i=m,\ldots,1$ and every finite
		$j\in\{0,1,\ldots,N\}$, the process
		\begin{align*}
			 \left(V_{\tau_*^{[i]}\wedge n}^{[i]}\mathbf1_{\{s(\tau_*^{[i+1]})=j\}}
 			\right)_{n=j,\ldots,N}
		\end{align*}
		is a martingale. 
		On the event $\{s(\tau_*^{[i+1]})=\partial\}$, one has
		$\tau_*^{[i]}=\partial$ and all subsequent rewards and values are zero.
	\item[(e')] At the terminal date, $V_N^{[i]}=X_N^{[i]}=X_N$ a.s.
	\item[(f')] The vector $\vec\tau_*^{[m]}$ belongs to
		$\bar\T_N^{[m]}$ and is optimal, as shown next.
	\end{enumerate}

\begin{theorem}\label{th22}
If $(X_n)_{n\in\N}$ satisfies \eqref{A1}, then, for every $m=1,\ldots,N+1$,
	\begin{align*}
 		V_0^{[m]}&=\sup_{\vec\tau^{[m]}\in\bar\T_N^{[m]}}E[X_{\vec\tau^{[m]}}]
 		=E[X_{\vec\tau_*^{[m]}}]=E\left[\sum_{i=1}^mX_{\tau_*^{[i]}}\right].
	\end{align*}
The vector of stopping times $\vec\tau_*^{[m]}$ is optimal.
\end{theorem}
\begin{proof}
\textit{Step 1.}
Let $\vec\tau^{[m]}=(\tau^{[m]},\ldots,\tau^{[1]})\in\bar\T_N^{[m]}$.
Recall that $\F_{\partial}=\F_N$ and
	$X_{\partial}=X_{\partial}^{[i]}=V_{\partial}^{[i]}=0$.
We claim that, for $i=1,\ldots,m$,
	\begin{align}\label{step1claim}
 		E\left[\sum_{j=1}^{i}X_{\tau^{[j]}}\,\middle|\,\F_{\tau^{[i]}}\right]
 		\le V_{\tau^{[i]}}^{[i]}\quad\text{a.s.}
	\end{align}
For $i=1$, this is the Snell-envelope inequality
	$X_{\tau^{[1]}}\le V_{\tau^{[1]}}^{[1]}$, which also holds on
	$\{\tau^{[1]}=\partial\}$, where both sides vanish.
Suppose that \eqref{step1claim} holds for $i-1$.
The supermartingale property of $V^{[i-1]}$ extends to the cemetery
	step, since $E[V_{\partial}^{[i-1]}\mid\F_N]=0\le V_N^{[i-1]}$; hence
	optional sampling applies to stopping times with values in
	$\{0,\ldots,N,\partial\}$.
Because $\tau^{[i-1]}\succeq_{\partial}s(\tau^{[i]})$, the tower property, the
	induction hypothesis, and optional sampling give
	\begin{align*}
 		E\left[\sum_{j=1}^{i}X_{\tau^{[j]}}\,\middle|\,\F_{\tau^{[i]}}\right]
 		&\le X_{\tau^{[i]}}
   			+E[V_{\tau^{[i-1]}}^{[i-1]}\mid\F_{\tau^{[i]}}] \le X_{\tau^{[i]}}
   			+E[V_{s(\tau^{[i]})}^{[i-1]}\mid\F_{\tau^{[i]}}]\\
 		&=X_{\tau^{[i]}}^{[i]} \le V_{\tau^{[i]}}^{[i]},
	\end{align*}
	where every relation remains valid on $\{\tau^{[i]}=\partial\}$, on which
	all terms are zero because $\tau^{[j]}=\partial$ for every $j\le i$.
Taking $i=m$ in \eqref{step1claim}, taking expectations, and applying
	optional sampling to the supermartingale $V^{[m]}$ between $0$ and
	$\tau^{[m]}$, we obtain
	\begin{align*}
 		E[X_{\vec\tau^{[m]}}]\le E[V_{\tau^{[m]}}^{[m]}]\le V_0^{[m]}.
	\end{align*}

\textit{Step 2.}
For the first exercise time, the martingale property in (d') gives
	\begin{align*}
		 V_0^{[m]}=E[V_{\tau_*^{[m]}}^{[m]}]=E[X_{\tau_*^{[m]}}^{[m]}]
 			=E[X_{\tau_*^{[m]}}]+E[V_{s(\tau_*^{[m]})}^{[m-1]}].
	\end{align*}
For $i=m-1,\ldots,1$, sum the indicator-weighted martingale identities
	in (d') over the finite values of $s(\tau_*^{[i+1]})$.
The cemetery event contributes zero, and hence
	\begin{align*}
 		E[V_{s(\tau_*^{[i+1]})}^{[i]}]=E[V_{\tau_*^{[i]}}^{[i]}]
 			=E[X_{\tau_*^{[i]}}^{[i]}]
		 	=E[X_{\tau_*^{[i]}}]+E[V_{s(\tau_*^{[i]})}^{[i-1]}].
	\end{align*}
Iterating this identity and using $V^{[0]}\equiv0$ yields
	\begin{align*}
 		V_0^{[m]}=E\left[\sum_{i=1}^mX_{\tau_*^{[i]}}\right].
	\end{align*}
Together with Step~1, this proves the theorem.
\end{proof}
\subsection{Monotone multiple stopping}\label{subsec:monotone-multiple}
Let $A_n$ denote the event $\{X_n\ge E[X_{n+1}\mid\F_n]\}$. 
Following Chow, Robbins and Siegmund~\cite{chowrobbinssiegmund}, 
	we say that the single-stopping problem is \emph{monotone} if
	\begin{align*}
		A_0\subset A_1\subset A_2\subset \cdots \quad\text{a.s.}
	\end{align*}
Roughly speaking, monotonicity implies that once the reward process 
	$(X_n)_{n \in \N}$ acquires the one-step supermartingale inequality 
	at some time $n$, it retains it at all subsequent times.
Let $\tau_{\mathrm{ola}}$ denote the stopping time defined by 
	\begin{align*}
		\tau_{\mathrm{ola}}=\min\{n\in\N:X_n\ge E[X_{n+1}\mid\F_n]\},
	\end{align*} 
	with the convention $\min\emptyset=\infty$, 
	where $X_n$ is the stopping reward at time $n$ and
	$E[X_{n+1}\mid\F_n]$ is the expected reward if one were 
	to continue for one additional period and then stop at time $n+1$.
In this interpretation, $\tau_{\mathrm{ola}}$ is sometimes referred 
	to as the one-step look-ahead stopping time, or the monotone stopping time. 
The terms \emph{stopping rule} and \emph{strategy} are used when one
	emphasizes that the stopping time prescribes when to stop; for
	consistency, we use the term \emph{stopping time} throughout this article.

We emphasize that, even in the monotone case, the reward process 
	$(X_n)_{n\in\N}$ need not be a supermartingale globally in time; 
	monotonicity only guarantees the one-step supermartingale inequality 
	$X_n\ge E[X_{n+1}\mid\F_n]$ a.s.\ on the event $\{\tau_{\mathrm{ola}}\le n\}$, 
	for every $n\in\N$.
A cornerstone of the theory of monotone stopping, in both finite 
	and infinite horizons, is that if the single-stopping problem is monotone, then $\tau_{\mathrm{ola}}$ is optimal.

\begin{theorem}[Monotone case; cf.~\cite{chowrobbins,chowrobbinssiegmund}]\label{monotonetheorem}
	If the reward process $(X_n)_{n\in\N}$ satisfies \eqref{A1} and 
	the stopping problem is monotone, then $\tau_{\mathrm{ola}}$ is optimal.
\end{theorem} 
\begin{proof}
Let $\tau^{[1]}=\tau_{\mathrm{ola}}$ for short.
From Theorem \ref{th21}, if (i) $V_{\tau^{[1]}}=X_{\tau^{[1]}}$ a.s.\ and 
	(ii) $(V_{\tau^{[1]}\land n})_{n\in\N}$ is a martingale, 
	then $\tau^{[1]}$ is optimal.

We first show (i).
Fix $j\in\N$. 
By the definition of $\tau^{[1]}$ and the monotonicity assumption, 
	$\{\tau^{[1]}=j\}\subset A_j\subset A_n$ for every $n\ge j$, that is, 
	\begin{align*}
				X_n\ge E[X_{n+1}\mid\F_n] \quad\text{a.s.\ on }\{\tau^{[1]}=j\}
		\ \text{for every }n\ge j.
	\end{align*}
Since $\{\tau^{[1]}=j\}\in\F_j\subset\F_n$, multiplying the inequality by 
	$\mathbf{1}_{\{\tau^{[1]}=j\}}$ shows that the process 
	$(X_n\mathbf{1}_{\{\tau^{[1]}=j\}})_{n\ge j}$ is a nonnegative 
	supermartingale.
By the optional sampling theorem for nonnegative supermartingales, 
	which applies to arbitrary stopping times 
	with the terminal value $X_\infty=\limsup_{n\to\infty}X_n$, 
	we have, for every $\tau\in\T$ with $\tau\ge\tau^{[1]}$ a.s., 
	(apply the theorem to the stopping time $\tau\vee j\ge j$, which
	coincides with $\tau$ on $\{\tau^{[1]}=j\})$,
	\begin{align*}
		E[X_\tau\mathbf{1}_{\{\tau^{[1]}=j\}}\mid\F_j]
				\le X_j\mathbf{1}_{\{\tau^{[1]}=j\}} \quad\text{a.s.}
	\end{align*}
Patching over $j\in\N$ yields $E[X_\tau\mid\F_{\tau^{[1]}}]\le X_{\tau^{[1]}}$ a.s.\ on $\{\tau^{[1]}<\infty\}$; 
	on $\{\tau^{[1]}=\infty\}$ we have $\tau=\infty$ and 
	$X_\tau=X_\infty=V_\infty$, so the inequality is trivial.
Hence, taking the essential supremum over $\tau\ge\tau^{[1]}$, 
	\begin{align*}
		V_{\tau^{[1]}}=\esup_{\tau\ge \tau^{[1]}}E[X_\tau\mid\F_{\tau^{[1]}}]
								\le X_{\tau^{[1]}}\quad\text{a.s.}
	\end{align*}
On the other hand, 
	$V_{\tau^{[1]}}\ge X_{\tau^{[1]}}$ a.s.\ by the definition of $V$, 
	which proves $V_{\tau^{[1]}}=X_{\tau^{[1]}}$ a.s.

We next show (ii). We have
	\begin{align*}
		E[V_{\tau^{[1]}\land (n+1)}\mid\F_n]
			&=\sum_{j=0}^nV_j\mathbf{1}_{\{\tau^{[1]}=j\}}
				+E[V_{n+1}\mathbf{1}_{\{\tau^{[1]}\ge n+1\}}\mid\F_n]\\
			&=\sum_{j=0}^nV_j\mathbf{1}_{\{\tau^{[1]}=j\}}
				+\mathbf{1}_{\{\tau^{[1]}\ge n+1\}}E[V_{n+1}\mid\F_n],
	\end{align*}
	where we use $\{\tau^{[1]}\ge n+1\}\in\F_n$. 
By the definition of $\tau^{[1]}$, we have 
	\begin{align*}
		X_n < E[X_{n+1}\mid\F_n] \quad\text{on the event }\{\tau^{[1]}\ge n+1\},
	\end{align*}
	and $E[X_{n+1}\mid\F_n]\le E[V_{n+1}\mid\F_n]$ because $X_{n+1}\le V_{n+1}$. 
Therefore, by the dynamic programming equation $V_n=\max\{X_n, E[V_{n+1}\mid\F_n]\}$, 
	we have $V_n=E[V_{n+1}\mid\F_n]$ on $\{\tau^{[1]}\ge n+1\}$. 
Hence, 
	\begin{align*}
		E[V_{\tau^{[1]}\land (n+1)}\mid\F_n]
			&=\sum_{j=0}^nV_j\mathbf{1}_{\{\tau^{[1]}=j\}}
				+\mathbf{1}_{\{\tau^{[1]}\ge n+1\}}V_n=V_{\tau^{[1]}\land n},
	\end{align*}
	which proves that the stopped sequence $(V_{\tau^{[1]}\land n})_{n\in\N}$ 
	is a martingale; it is uniformly integrable since 
	$0\le V_{\tau^{[1]}\land n}\le E[\bar{X}\mid\F_{\tau^{[1]}\land n}]$. 
\end{proof}

The argument used for (i) and (ii) above, applied at an arbitrary time 
	$n$ instead of at $\tau_{\mathrm{ola}}$, yields the following structure of 
	the value function in the monotone case: 
	\begin{align}\label{V1}
		V_{n}=
			\begin{cases}
						X_{n} & \text{a.s.\ on } \{\tau_{\mathrm{ola}}\le n\}, \\[10pt]
			E[V_{n+1}\mid\F_n] & \text{a.s.\ on } \{\tau_{\mathrm{ola}}> n\}. 
			\end{cases}
	\end{align}
Indeed, on $\{\tau_{\mathrm{ola}}\le n\}$ the one-step supermartingale inequality 
	holds from time $n$ onwards, so the argument of (i) gives $V_n=X_n$ 
	there, while on $\{\tau_{\mathrm{ola}}>n\}$ the argument of (ii) applies.

We now turn to the finite-horizon multiple-stopping problem of
	Section~\ref{subsec:multiple-stopping}.  
Set $G_{\partial}^{[i]}:=0$ and, for $n=0,\ldots,N$, define
	\begin{align}\label{Gnm}
	 	G_n^{[i]}:=X_n^{[i]}-E[X_{s(n)}^{[i]}\mid\F_n],
 		\qquad i=1,\ldots,m.
	\end{align}
Here $X_n^{[1]}=X_n$ and $G^{[0]}_\cdot\equiv0$.
The quantity $G_n^{[i]}$ is the difference between exercising at time
	$n$ and exercising at the next date $s(n)$.
Let
	\begin{align*}
 		A_n^{[i]}:=\{G_n^{[i]}\ge0\},\qquad
 		\tau_{\mathrm{ola}}^{[i]}:=\min\{0\le n\le N:G_n^{[i]}\ge0\}.
	\end{align*}
Since $G_N^{[i]}=X_N\ge0$, the defining set is nonempty.
We say that the problem is \emph{monotone at level $i$} if
	\begin{align*}
 		A_0^{[i]}\subset A_1^{[i]}\subset\cdots\subset A_N^{[i]}
 		\quad\text{a.s.},
	\end{align*}
	and \emph{monotone at every level} if this holds for every
	$i=1,\ldots,m$.
If it is monotone at level $i$, then
	$G_n^{[i]}\ge0$ on $\{\tau_{\mathrm{ola}}^{[i]}\le n\}$.
To transfer the single-stopping results, extend the level-$i$ reward by
	zero after $N$; stopping at any time greater than $N$ then has the same
	(zero) reward as stopping at $\partial$, so the extended problem has
	the same value process $(V_n^{[i]})_{0\le n\le N}$, and its one-step
	look-ahead events coincide with $A_n^{[i]}$ for $n\le N$ and equal
	$\Omega$ afterwards.
The structure \eqref{V1}, applied to this extended problem, yields
	\begin{align}\label{Vlevel}
	 	V_n^{[i]}=\begin{cases}
 			X_n^{[i]},&\text{a.s.\ on }\{\tau_{\mathrm{ola}}^{[i]}\le n\},\\[6pt]
 			E[V_{s(n)}^{[i]}\mid\F_n],&\text{a.s.\ on }\{\tau_{\mathrm{ola}}^{[i]}>n\},
 			\end{cases}
 			\qquad n=0,\ldots,N.
	\end{align}
The following fundamental theorem for the monotone multiple-stopping
	problem is simple but useful for determining an optimal vector of stopping times.		
\begin{theorem}\label{monotonetheorem2}
Suppose that $(X_n)_{n\in\N}$ satisfies \eqref{A1} and that the finite-horizon
	multiple-stopping problem is monotone at every level $i=1,\ldots,m$.
Define $\vec{\tau}_{\mathrm{ola}}^{[m]}=(\sigma_m,\ldots,\sigma_1)$ recursively by
	\begin{align*}
 		\sigma_m&=\tau_{\mathrm{ola}}^{[m]},\\
 		\sigma_i&=\min\{s(\sigma_{i+1})\le n\le N:G_n^{[i]}\ge0\},
 			\qquad i=m-1,\ldots,1,
	\end{align*}
	where the minimum of an empty set is $\partial$.
When $s(\sigma_{i+1})<\partial$, monotonicity gives
	$\sigma_i=\max\{s(\sigma_{i+1}),\tau_{\mathrm{ola}}^{[i]}\}$;
	when $s(\sigma_{i+1})=\partial$, one has $\sigma_i=\partial$.
Then $\vec{\tau}_{\mathrm{ola}}^{[m]}$ is optimal.
\end{theorem}
\begin{proof}
We verify the two properties used in Step~2 of the proof of Theorem~\ref{th22}.
First, if $\sigma_i<\partial$, then $\sigma_i\ge\tau_{\mathrm{ola}}^{[i]}$, and
	\eqref{Vlevel} gives $V_{\sigma_i}^{[i]}=X_{\sigma_i}^{[i]}$.
The same equality holds at $\partial$, since both sides are zero.
Second, fix a finite $j$ and consider the event $\{s(\sigma_{i+1})=j\}$.
For $j\le n<\sigma_i$, the definition of $\sigma_i$ gives
	$G_n^{[i]}<0$, whence $n<\tau_{\mathrm{ola}}^{[i]}$ by the monotonicity
	at level $i$, and hence $V_n^{[i]}=E[V_{n+1}^{[i]}\mid\F_n]$ by \eqref{Vlevel}.
Therefore
	\begin{align*}
		 \left(V_{\sigma_i\wedge n}^{[i]}
 		\mathbf1_{\{s(\sigma_{i+1})=j\}}\right)_{n=j,\ldots,N}
	\end{align*}
	is a martingale.
On $\{s(\sigma_{i+1})=\partial\}$, all subsequent stopping times and
	rewards are $\partial$ and zero.
The telescoping argument of Step~2 in Theorem~\ref{th22}, together with
	the upper bound in Step~1, proves optimality.
\end{proof}

The following example shows that, in contrast with what one might expect, 
	the monotonicity of the single-stopping problem alone does \emph{not}  
	imply the monotonicity of the multiple-stopping problem at higher 
	levels, and that the one-step look-ahead vector may then be
	strictly suboptimal; this is why the levelwise assumption is imposed in 
	Theorem \ref{monotonetheorem2}.
\begin{example}\label{counterexample}
Let $N=4$ and let $\xi$ be a random variable with
	$P(\xi=H)=P(\xi=T)=1/2$, revealed at time $2$;
	that is, $\F_0=\F_1=\{\emptyset,\Omega\}$ (up to null sets) and
	$\F_n=\sigma(\xi)$ for $n\ge 2$.
Define the nonnegative bounded rewards by $X_0=0$, $X_1=1$, and
	\begin{align*}
	(X_2,X_3,X_4)=
		\begin{cases}
		(3,\ 1,\ 0) & \text{on } \{\xi=H\}, \\
		(0,\ 4,\ 5) & \text{on } \{\xi=T\}.
		\end{cases}
	\end{align*}
Then $G^{[1]}_0=-1$, $G^{[1]}_1=-1/2$,
	$(G^{[1]}_2,G^{[1]}_3,G^{[1]}_4)=(2,1,0)$ on $\{\xi=H\}$ and
	$(-4,-1,5)$ on $\{\xi=T\}$,
	so the single-stopping problem is monotone, with
	$\tau_{\mathrm{ola}}^{[1]}=2$ on $\{\xi=H\}$ and $\tau_{\mathrm{ola}}^{[1]}=4$ on
	$\{\xi=T\}$.
The level-one Snell envelope is
	$(V^{[1]}_2,V^{[1]}_3,V^{[1]}_4)=(3,1,0)$ on $\{\xi=H\}$,
	$(5,5,5)$ on $\{\xi=T\}$, and $V^{[1]}_1=4$.
Hence $X^{[2]}_0=4$, $X^{[2]}_1=5$, and
	$(X^{[2]}_2,X^{[2]}_3,X^{[2]}_4)=(4,1,0)$ on $\{\xi=H\}$ and
	$(5,9,5)$ on $\{\xi=T\}$.
Therefore
	$G^{[2]}_1=5-(4+5)/2=1/2\ge0$, whereas
	$G^{[2]}_2=5-9=-4<0$ on $\{\xi=T\}$.
Hence $A^{[2]}_1=\Omega\not\subset A^{[2]}_2$, and the double-stopping problem fails to be monotone.
The failure is not merely formal.  Since $G^{[2]}_0=-1<0$, the recursively
	defined vector of Theorem~\ref{monotonetheorem2} selects at time $1$ and
	then at $\tau_{\mathrm{ola}}^{[1]}$, and its expected reward is
	$1+(3+5)/2=5$.  By contrast, $V^{[2]}_0=13/2$: it is optimal to pass
	time $1$ and to stop at times $2$ and $3$ on $\{\xi=H\}$ and at times $3$
	and $4$ on $\{\xi=T\}$, which gives $(4+9)/2=13/2$.  Thus the
	one-step look-ahead vector is strictly suboptimal.
The mechanism is that $G^{[2]}_1\ge 0$ may be produced by the conditional
	expectation of a large future value of $G^{[1]}$,
	whereas, on the realized path, the anticipated gain fails to
	materialize.
\end{example}

The next corollary provides a recursive equation for the one-step 
	look-ahead functions, which is the practical tool for verifying the 
	levelwise monotonicity required in Theorem \ref{monotonetheorem2}; 
	this is how the theorem is applied in Section~\ref{sec:applications}.
\begin{corollary}\label{corollary21}
Let $m\ge2$ and suppose that the problem is monotone at every level $i=1,\ldots,m-1$.
Then, for $n=0,\ldots,N-1$,
	\begin{align}\label{Gnm1}
 		G_n^{[m]}=G_n^{[1]}+E\left[G_{n+1}^{[m-1]}
 			\mathbf1_{\{\tau_{\mathrm{ola}}^{[m-1]}\le n+1\}}\mid\F_n\right],
	\end{align}
	and $G_N^{[m]}=G_N^{[1]}=X_N$.
\end{corollary}
\begin{proof}
For $n<N$, equations \eqref{Gnm} and \eqref{xnm} give
	\begin{align*}
 		G_n^{[m]}&=G_n^{[1]}+E\left[V_{n+1}^{[m-1]}
 			-E[V_{n+2}^{[m-1]}\mid\F_{n+1}]\mid\F_n\right].
	\end{align*}
On $\{\tau_{\mathrm{ola}}^{[m-1]}>n+1\}$, \eqref{Vlevel} gives
	$V_{n+1}^{[m-1]}=E[V_{n+2}^{[m-1]}\mid\F_{n+1}]$, so the difference
	inside the conditional expectation vanishes there.
On $B:=\{\tau_{\mathrm{ola}}^{[m-1]}\le n+1\}\in\F_{n+1}$, we have
	$V_{n+1}^{[m-1]}=X_{n+1}^{[m-1]}$ by \eqref{Vlevel}; moreover,
	$B\subset\{\tau_{\mathrm{ola}}^{[m-1]}\le n+2\}$, so \eqref{Vlevel}
	also gives $V_{n+2}^{[m-1]}=X_{n+2}^{[m-1]}$ a.s.\ on $B$
	(for $n=N-1$, read $V_{N+1}^{[m-1]}=X_{N+1}^{[m-1]}=0$ at the
	cemetery), and hence
	$E[V_{n+2}^{[m-1]}\mid\F_{n+1}]
		=E[X_{n+2}^{[m-1]}\mid\F_{n+1}]$ a.s.\ on $B$.
The difference therefore equals $G_{n+1}^{[m-1]}$ on $B$, which proves
	\eqref{Gnm1}.  At $n=N$, all continuation values vanish, so
	$G_N^{[m]}=X_N$.
\end{proof}

\begin{remark}\label{remarkindicator}
Under the monotonicity at level $m-1$, one has
	$\{\tau_{\mathrm{ola}}^{[m-1]}\le n+1\}=A_{n+1}^{[m-1]}$, so that
	\eqref{Gnm1} may equivalently be written as
	\begin{align*}
	 	G_n^{[m]}=G_n^{[1]}+E\left[G_{n+1}^{[m-1]}
		 \mathbf1_{A_{n+1}^{[m-1]}}\mid\F_n\right].
	\end{align*}
This is the form in which the equation is applied to the embedded
	Markov chains (see the end of Section~\ref{sec:markov}) of 
	Section~\ref{sec:applications}, where the event $A_{n+1}^{[m-1]}$ is
	determined by the state reached at the next transition.
\end{remark}

This useful recursive equation for $G_n^{[m]}$, which determines 
	$A_n^{[m]}=\{G_n^{[m]}\ge 0 \}$, 
	was, to the best of our knowledge, first obtained by 
	Ano and Tamaki~\cite{anotamaki} 
	for solving variations of the secretary problem.
	
\begin{corollary}\label{corollary22}
Suppose that the multiple-stopping problem is monotone at every level 
	$i=1,\ldots,m$.
Then, for each $n=0,\ldots,N$, 
	$G^{[m]}_n\ge G^{[m-1]}_n\ge\cdots\ge G^{[1]}_n$ a.s., 
	and consequently 
  	\begin{align*}
   		\tau_{\mathrm{ola}}^{[m]}\le \tau_{\mathrm{ola}}^{[m-1]}\le \cdots 
				\le \tau_{\mathrm{ola}}^{[1]} \quad\text{a.s.}
  	\end{align*}
Moreover, $\tau_{\mathrm{ola}}^{[1]}\le N$ a.s., since $G^{[1]}_N=X_N\ge 0$.
\end{corollary}
\begin{proof}
We show $G^{[i]}_n\ge G^{[i-1]}_n$ a.s.\ for each $n$ by induction on 
	$i=2,\ldots,m$.
For $i=2$, from (\ref{Gnm1}), 
	$G_n^{[2]}-G_n^{[1]}
		=E[G_{n+1}^{[1]}\mathbf{1}_{\{\tau_{\mathrm{ola}}^{[1]}\le n+1\}}\mid\F_n]\ge 0$, 
	since $G^{[1]}_{n+1}\ge 0$ a.s.\ on $\{\tau_{\mathrm{ola}}^{[1]}\le n+1\}$ 
	by the monotonicity at level 1.
Assume that $G_n^{[i-1]}\ge G_n^{[i-2]}$ a.s.\ for each $n=0,\ldots,N$; 
	this implies $A^{[i-2]}_n\subset A^{[i-1]}_n$ and hence 
	$\tau_{\mathrm{ola}}^{[i-1]}\le \tau_{\mathrm{ola}}^{[i-2]}$ a.s. 
Then, using (\ref{Gnm1}) for levels $i$ and $i-1$, 
	\begin{align*}
		G_n^{[i]}-G_n^{[i-1]}
		&=E[G_{n+1}^{[i-1]}\mathbf{1}_{\{\tau_{\mathrm{ola}}^{[i-1]}\le n+1\}}\mid\F_n]
			-E[G_{n+1}^{[i-2]}\mathbf{1}_{\{\tau_{\mathrm{ola}}^{[i-2]}\le n+1\}}\mid\F_n]\\
		&=E[(G_{n+1}^{[i-1]}
			-G_{n+1}^{[i-2]})\mathbf{1}_{\{\tau_{\mathrm{ola}}^{[i-2]}\le n+1\}}\mid\F_n]
			+E[G_{n+1}^{[i-1]}
			\mathbf{1}_{\{\tau_{\mathrm{ola}}^{[i-1]}\le n+1<\tau_{\mathrm{ola}}^{[i-2]}\}}\mid\F_n]\\
		&\ge 0, 
	\end{align*}
	where the first term is nonnegative by the induction hypothesis and 
	the second is nonnegative since $G^{[i-1]}_{n+1}\ge 0$ a.s.\ on 
	$\{\tau_{\mathrm{ola}}^{[i-1]}\le n+1\}$ by the monotonicity at level $i-1$.
The a.s.\ ordering of the one-step look-ahead stopping times follows 
	immediately from the inclusions 
	$A^{[1]}_n\subset A^{[2]}_n\subset\cdots\subset A^{[m]}_n$.
\end{proof}

In some applications, the one-step look-ahead function $G^{[1]}_n$ is 
	deterministic as a function of the calendar-time index $n$ alone, 
	rather than merely a deterministic function of a Markov state.
In this case, the monotonicity of the single-stopping problem does 
	propagate to every level, as the following proposition shows, 
	so that the levelwise assumption of Theorem \ref{monotonetheorem2} 
	is automatically satisfied.
For a deterministic finite sequence $(G^{[i]}_n)_{0\le n\le N}$, 
	monotonicity at level $i$ means that 
	$G_n^{[i]}\ge 0$ implies $G_{n+1}^{[i]}\ge 0$.
In this case, $\tau_{\mathrm{ola}}^{[i]}$ reduces to the deterministic threshold 
	\begin{align*}
		i_*^{[i]}=\min\{0\le n\le N: G_n^{[i]}\ge 0\}. 
	\end{align*}
Each threshold value is recursively calculated using the 
	deterministic form of the equation (\ref{Gnm1}), 
	\begin{align}\label{Gnm1det}
		G_n^{[m]}=G_n^{[1]}+G_{n+1}^{[m-1]}\mathbf{1}_{\{n+1\ge i_*^{[m-1]}\}}.
	\end{align}
This provides further insight into the optimal multiple stopping time, 
	which is not offered by the form of 
	$\tau_*^{[m]} = \min\{0\le n\le N:V_n^{[m]}=X_n^{[m]}\}$ in (d').
\begin{proposition}\label{propdet}
Suppose that the reward process $(X_n)_{n\in\N}$ satisfies \eqref{A1}, 
	that the single-stopping problem is monotone, and that 
	$G^{[1]}_n$ is deterministic for every $n=0,\ldots,N$.
Then, for every $i=1,\ldots,m$, $G^{[i]}_n$ is deterministic for every 
	$n=0,\ldots,N$ and the multiple-stopping problem is monotone at level $i$.
Consequently, Theorem \ref{monotonetheorem2} applies.
\end{proposition}
\begin{proof}
We show, by induction on $i=2,\ldots,m$, the following three assertions: 
	($\alpha$) $G^{[i]}_n$ is deterministic for every $n$; 
	($\beta$) $G^{[i]}_n\ge G^{[i-1]}_n$ for every $n$; 
	($\gamma$) the problem is monotone at level $i$.
Note that ($\beta$) implies $i_*^{[i]}\le i_*^{[i-1]}$, so that, once 
	the induction is closed up to level $i-1$, we have 
	$i_*^{[i-1]}\le i_*^{[i-2]}\le\cdots\le i_*^{[1]}$.
Assume the assertions hold up to level $i-1$; for $i=2$, level $1$ 
	satisfies ($\alpha$) and ($\gamma$) by hypothesis, and assertion 
	($\beta$) at level $1$ is never used in the argument below.
Since the problem is monotone at levels $1,\ldots,i-1$, 
	Corollary \ref{corollary21} applies, and since 
	$G^{[i-1]}_{n+1}$ and $i_*^{[i-1]}$ are deterministic, 
	the recursive equation (\ref{Gnm1}) reduces to the deterministic 
	identity (\ref{Gnm1det}) with $m$ replaced by $i$, which proves 
	($\alpha$).
The assertion ($\beta$) follows from (\ref{Gnm1det}) exactly as in 
	the proof of Corollary \ref{corollary22}, using the induction 
	hypotheses ($\beta$) and ($\gamma$) at levels $\le i-1$.
We prove ($\gamma$). 
Suppose $G^{[i]}_n\ge 0$; we show $G^{[i]}_{n+1}\ge 0$.
If $n+1\ge i_*^{[i-1]}$, then, by ($\beta$) and the monotonicity at 
	level $i-1$, 
	$G^{[i]}_{n+1}\ge G^{[i-1]}_{n+1}\ge 0$.
If $n+1<i_*^{[i-1]}$, then the indicator in (\ref{Gnm1det}) vanishes 
	and $G^{[i]}_n=G^{[1]}_n\ge 0$, whence $n\ge i_*^{[1]}$; 
	but $i_*^{[1]}\ge i_*^{[i-1]}$ by the ordering of the thresholds, 
	so $n+1>i_*^{[i-1]}$, a contradiction.
Hence the second case is vacuous, and the induction is complete.
\end{proof}
Under the hypotheses of Proposition~\ref{propdet}, the levelwise
	monotonicity yields the threshold structure
	$G^{[i]}_n\ge 0$ for all $n\ge i_*^{[i]}$, and
	Corollary~\ref{corollary22} gives
	$i_*^{[m]}\le i_*^{[m-1]}\le\cdots\le i_*^{[1]}\le N$.

Example \ref{counterexample} shows that the deterministic assumption of 
	Proposition \ref{propdet} cannot simply be dropped: there, 
	$G^{[1]}_2$ is random, and the propagation of the monotonicity 
	from level $1$ to level $2$ fails.
\subsection{Markovian case}\label{sec:markov}
Let $(Z_n)_{n=0}^N$ be a Markov chain with respect to
	$\F_n=\sigma(Z_0,\ldots,Z_n)$, started from a deterministic state $Z_0$
	(so that $\F_0$ is trivial), and let the basic reward be
	$X_n=g_n(Z_n)$, where each $g_n$ is nonnegative and measurable.
Assume \eqref{A1}.
In the Markovian case the Snell envelopes and the augmented rewards are
	functions of the current state, and we use the same symbol for the state
	function and for the random variable obtained by evaluating it at $Z_n$:
	\begin{align*}
 		V_n^{[i]}=V_n^{[i]}(Z_n),\qquad X_n^{[i]}=g_n^{[i]}(Z_n).
	\end{align*}
This convention is used throughout the applications in Section~\ref{sec:applications}, 
	where capital letters denote finite-horizon state functions and lower-case
	letters are reserved for their scaling limits.
Set $V_{\partial}^{[i]}\equiv0$, $g_n^{[1]}=g_n$, and, for $i=2,\ldots,m$,
	\begin{align*}
 		g_n^{[i]}(z)&=g_n(z)+E\! \left[V_{s(n)}^{[i-1]}(Z_{s(n)})\mid Z_n=z\right].
	\end{align*}
For $n=N$ one has $s(N)=\partial$, and every conditional expectation
	involving $Z_{s(N)}$ is read as zero; equivalently,
	$g_\partial^{[i]}\equiv V_\partial^{[i]}\equiv0$.
The Markov property and the dynamic programming principle imply
	\begin{align*}
 		V_n^{[i]}(z)=\max\left\{g_n^{[i]}(z),
 		E\left[V_{s(n)}^{[i]}(Z_{s(n)})\mid Z_n=z\right]\right\},\qquad n=0,\ldots,N.
	\end{align*}
The optimal stopping vector is therefore defined recursively by
	\begin{align*}
 		\tau_*^{[m]}&=\min\{0\le n\le N:
 			V_n^{[m]}(Z_n)=g_n^{[m]}(Z_n)\},\\
 		\tau_*^{[i]}&=\min\{s(\tau_*^{[i+1]})\le n\le N:
 			V_n^{[i]}(Z_n)=g_n^{[i]}(Z_n)\},\qquad i=m-1,\ldots,1,
	\end{align*}
	where an empty minimum is $\partial$, in accordance with the cemetery convention 
	of Section~\ref{subsec:multiple-stopping}.
Theorem~\ref{th22} applies to this vector.

For the monotone analysis, define the deterministic state function
	\begin{align*}
	 	\gamma_n^{[i]}(z):=g_n^{[i]}(z) -E[g_{s(n)}^{[i]}(Z_{s(n)})\mid Z_n=z],
	\end{align*}
	so that $G_n^{[i]}=\gamma_n^{[i]}(Z_n)$.
Theorem~\ref{monotonetheorem2} applies once the events
	$A_n^{[i]}=\{\gamma_n^{[i]}(Z_n)\ge0\}$ are a.s.\ nondecreasing in $n$
	at every level.
Corollary~\ref{corollary21} is the main recursive tool for checking this condition.
Proposition~\ref{propdet} applies only in the special case in which
	$G_n^{[1]}$ is deterministic as a function of the calendar-time index
	$n$ itself; the fact that $\gamma_n^{[1]}$ is a deterministic function
	of the state does not by itself imply that hypothesis.

\paragraph{Threshold structure versus pathwise monotonicity.}
Put
	\begin{align*}
   		 \mathcal A_n^{[i]}:=\{z:\gamma_n^{[i]}(z)\ge0\},
  		  \qquad\text{so that}\qquad
   		 A_n^{[i]}=\{Z_n\in\mathcal A_n^{[i]}\}.
	\end{align*}
If the state space is partially ordered, the problem has a
	\emph{statewise threshold structure} at level $i$ when every
	$\mathcal A_n^{[i]}$ is an upper set.  
For a real-valued state this holds,
	for instance, when $z\mapsto\gamma_n^{[i]}(z)$ crosses zero at most once,
	from negative to nonnegative, in which case $\mathcal A_n^{[i]}$ is a
	half-line determined by a scalar boundary.

By contrast, \emph{pathwise monotonicity} at level $i$ is the hypothesis of
	Theorem~\ref{monotonetheorem2}, namely
	\begin{align*}
    		P\bigl(Z_n\in\mathcal A_n^{[i]},\
    		Z_{n+1}\notin\mathcal A_{n+1}^{[i]}\bigr)=0,
    		\qquad n=0,\ldots,N-1.
	\end{align*}
This is a property of the transition kernel relative to the sets
	$\mathcal A_n^{[i]}$, not of the sets alone.
A simple sufficient condition, used for the success and record chains of
	Section~\ref{sec:applications}, is that $Z_n$ is almost surely nondecreasing
	and that
	\begin{align*}
    		\mathcal A_n^{[i]}\subset\mathcal A_{n+1}^{[i]}.
	\end{align*}
Then $Z_n\in\mathcal A_n^{[i]}$ implies
	$Z_{n+1}\in\mathcal A_n^{[i]}\subset\mathcal A_{n+1}^{[i]}$.
Without such a condition, nested or even time-independent thresholds do not
	prevent the state from moving back across the boundary.
Thus a statewise threshold structure does not by itself imply the pathwise
	monotonicity required in Theorem~\ref{monotonetheorem2}.

\paragraph{Embedded chains.}
In the applications, Sections~\ref{subsec:multiple-stopping}--\ref{sec:markov}
	are applied along an embedded chain, typically the sequence of success or
	record epochs, rather than in calendar time.
Let $(Z_k)_{k\ge0}$ be a Markov chain on $E\cup\{\partial\}$, where
	$\partial$ is absorbing, with reward $g(\partial)=0$.
If the chain is absorbed after at most $N$ transitions, the preceding
	finite-horizon results apply verbatim with filtration
	$\sigma(Z_0,\ldots,Z_k)$ and horizon $N$.
When the number of transitions is unbounded, only the single-stopping results
	of Section~\ref{subsec:single-stopping} are available.
Whether stopping can be restricted to the embedded epochs without loss of
	optimality is model specific and is verified in each application.

\section{Applications: Last-Success Problems}
\label{sec:applications}
This section applies the general multiple-stopping framework
	developed in Section~\ref{sec:optimal-multiple-stopping} to last-success selection problems.
	Throughout this section, ``odds problem'' refers to the classical independent-Bernoulli setting
	with known success probabilities and its direct multiple-selection and random-horizon extensions,
	whereas ``last-success problem'' refers to the broader class of models whose objective is to select
	the final success, including the Markov-dependent and Bayesian models considered below.  
We begin with Bruss' classical odds problem and then consider random horizons, Markov dependence, and
	Bayesian learning.
In each application, $\partial$ denotes a cemetery state placed after all
	admissible observation indices; when the index set is unbounded, we take $\partial=\infty$.  
A selection at $\partial$ has reward zero.
\subsection{The classical odds problem}
\label{subsec:odds-basic}
The classical odds problem introduced by Bruss~\cite{bruss2000} is a last-success problem for
	independent Bernoulli trials with known success probabilities.  
It contains the classical secretary problem as a special case and has generated many extensions.  
These include the last-$\ell$-success problems of Bruss and
	Paindaveine~\cite{brusspaindaveine} and Tamaki~\cite{tamaki2010},
the Markov-dependent trials studied by Hsiau and Yang~\cite{hsiauyang}, and the
	multiple-stopping odds problem of Ano, Kakinuma and Miyoshi~\cite{anokakinumamiyoshi}.
Matsui and Ano~\cite{matsuiano} established lower bounds for the
	multiple-stopping odds problem, while Matsui and Ano~\cite{matsuiano2017}
	studied related threshold rules through ratios of symmetric polynomials of the odds.

\subsubsection{Single stopping}
Let $I_1,\ldots,I_N$ be independent Bernoulli random variables, where
	$I_i=1$ denotes a success.  
Assume that $0<p_i:=P(I_i=1)<1$, write $q_i:=1-p_i$ and $r_i:=p_i/q_i$, 
	and put $\Na=\{1,\ldots,N\}$.
The objective is to stop at the last success.
If one stops at time $i$, the actual win indicator is
	\begin{align*}
 		Y_i=I_i\prod_{j=i+1}^N(1-I_j).
	\end{align*}
Its conditional expectation given the information at time $i$ is the
	adapted reward
	\begin{align*}
 		X_i:=E[Y_i\mid\F_i]=I_i\prod_{j=i+1}^Nq_j.
	\end{align*}
For every stopping time $\tau$, the tower property gives $E[Y_\tau]=E[X_\tau]$.
Thus maximizing the winning probability is equivalent to the optimal stopping
	problem with reward $(X_i)$.

On a failure, $X_i=0$, whereas continuation may have positive value.
Stopping on a failure can therefore be excluded without loss of optimality.  
The same holds with several selection opportunities: selecting a
	failure yields reward zero and uses up one opportunity, and since
	$V^{[\ell-1]}\le V^{[\ell]}$ at every state, rejecting a failure weakly
	dominates selecting it.

We thus pass to the embedded Markov chain whose
	state is the calendar time of the current success.  
At an embedded success state $i$, write
 	$x_i^{[1]}:=\prod_{j=i+1}^Nq_j$. 
The probability that the next success after $i$ occurs at $j>i$ is
	\begin{align}\label{oddstransition}
 		P_{ij}&=\prod_{k=i+1}^{j-1}q_k\,p_j=\prod_{k=i+1}^{j}q_k\,r_j.
	\end{align}
The kernel is substochastic; its defect is the probability that no
	further success occurs, in which case the continuation payoff is zero.
Let $V_i^{[1]}$ be the value at the embedded success state $i$.  
Then $V_N^{[1]}=x_N^{[1]}=1$ and
	\begin{align*}
 		V_i^{[1]}=\max\left\{x_i^{[1]},\sum_{j=i+1}^NP_{ij}V_j^{[1]}\right\},
 		\qquad i=1,\ldots,N-1.
	\end{align*}
The one-step look-ahead difference is
	\begin{align*}
		 G_i^{[1]}&:=x_i^{[1]}-\sum_{j=i+1}^NP_{ij}x_j^{[1]}=\prod_{j=i+1}^Nq_j
 			\left(1-\sum_{j=i+1}^Nr_j\right).
	\end{align*}
Put
	\begin{align*}
		 H_i^{[1]}:=\sum_{j=i+1}^Nr_j,\qquad
		 i_*^{[1]}:=\min\{i\in\Na:H_i^{[1]}\le1\}.
	\end{align*}
Since $H_i^{[1]}$ is nonincreasing in $i$, the embedded problem is monotone.  
By Theorem~\ref{monotonetheorem}, an optimal stopping time is
	 $\tau_{\mathrm{ola}}^{[1]}=\min\{i\ge i_*^{[1]}:I_i=1\}$,
	with the convention that an empty defining set gives $\partial$.
When $G_i^{[1]}=0$, stopping at $i$ is still optimal; we adopt the
	convention of stopping.

For the classical secretary problem, $p_i=1/i$ for $i\ge2$.
The deterministic first record, for which $p_1=1$, is treated separately
	(or equivalently omitted from the embedded odds calculation), so it does
	not conflict with the standing assumption $0<p_i<1$.

\subsubsection{Multiple stopping}
Suppose that $m$ selections are available.  
Since the selected success epochs are distinct, at most one selected epoch can be the last success.
Consequently, the corresponding win events are mutually exclusive, and
	the expected sum of the conditional rewards equals the probability that
	one of the selections identifies the last success.

At an embedded success state $i$, define recursively
	\begin{align*}
 		x_i^{[m]}&=x_i^{[1]}+\sum_{j=i+1}^NP_{ij}V_j^{[m-1]},\\
 		V_i^{[m]}&=\max\left\{x_i^{[m]},\sum_{j=i+1}^NP_{ij}V_j^{[m]}\right\},
 		\qquad V_N^{[m]}=1,
	\end{align*}
	and define
	\begin{align*}
 		G_i^{[m]}:=x_i^{[m]}-\sum_{j=i+1}^NP_{ij}x_j^{[m]}.
	\end{align*}
For integers $a,b$, write $a\lor b:=\max\{a,b\}$.

\begin{theorem}[Ano, Kakinuma and Miyoshi~\cite{anokakinumamiyoshi}]
\label{th31}
Set
	\begin{align*}
 		H_i^{[1]}=\sum_{j=i+1}^Nr_j,\qquad
		 i_*^{[1]}=\min\{i\in\Na:H_i^{[1]}\le1\}.
	\end{align*}
For $\ell=2,\ldots,m$, define recursively
	\begin{align}\label{Hrecursion}
	 	H_i^{[\ell]}&=H_i^{[1]}+\sum_{j=(i+1)\lor i_*^{[\ell-1]}}^N
 			r_j\bigl(H_j^{[\ell-1]}-1\bigr),\\
 		i_*^{[\ell]}&=\min\{i\in\Na:H_i^{[\ell]}\le1\}. \notag
	\end{align}
Then the vector $\vec{\tau}_{\mathrm{ola}}^{[m]}=(\tau_{\mathrm{ola}}^{[m]},\ldots,
	\tau_{\mathrm{ola}}^{[1]})$ defined by
	\begin{align*}
	 	\tau_{\mathrm{ola}}^{[m]}&=\min\{i\ge i_*^{[m]}:I_i=1\},\\
 		\tau_{\mathrm{ola}}^{[\ell]}
 			&=\min\{i\ge\max\{s(\tau_{\mathrm{ola}}^{[\ell+1]}),i_*^{[\ell]}\}: I_i=1\},
 		\qquad \ell=m-1,\ldots,1,
	\end{align*}
	is optimal.  
The minimum of an empty set is $\partial$.
\end{theorem}
\begin{proof}
For level one, $G_i^{[1]}=\prod_{j=i+1}^Nq_j(1-H_i^{[1]})$, and
	$H_i^{[1]}$ is nonnegative and nonincreasing; in particular, the
	embedded problem is monotone at level one.
Assume inductively that $H_i^{[m-1]}$ is nonnegative and nonincreasing in $i$, 
	so that the embedded problem is monotone at level
	$m-1$, with threshold $i_*^{[m-1]}$; in particular,
	$0\le H_j^{[m-1]}\le1$ for $j\ge i_*^{[m-1]}$.
Corollary~\ref{corollary21}, in the form of
	Remark~\ref{remarkindicator} and applied to the embedded Markov chain,
	gives
	\begin{align}\label{oddsgim}
	  	G_i^{[m]}=G_i^{[1]} +\sum_{j=(i+1)\lor i_*^{[m-1]}}^NP_{ij}G_j^{[m-1]}.
	\end{align}
Using \eqref{oddstransition} and writing
	$G_i^{[m]}=\prod_{j=i+1}^Nq_j(1-H_i^{[m]})$, equation
	\eqref{oddsgim} is exactly the recursion \eqref{Hrecursion}.
Equivalently,
	\begin{align*}
		 H_i^{[m]}=\begin{cases}
	 		\displaystyle\sum_{j=i+1}^Nr_jH_j^{[m-1]},
   				&i+1\ge i_*^{[m-1]},\\[8pt]
 			\displaystyle\sum_{j=i+1}^Nr_j+\sum_{j=i_*^{[m-1]}}^Nr_j(H_j^{[m-1]}-1),
   				&i+1<i_*^{[m-1]}.
 			\end{cases}
	\end{align*}
This representation shows that $H_i^{[m]}\ge0$: the first line is a
	sum of nonnegative terms, while in the second line, since
	$0\le H_j^{[m-1]}\le1$ for $j\ge i_*^{[m-1]}$ and $i_*^{[m-1]}\ge i+1$,
	\begin{align*}
 		\sum_{j=i_*^{[m-1]}}^Nr_j\bigl(1-H_j^{[m-1]}\bigr)
 		\le\sum_{j=i_*^{[m-1]}}^Nr_j \le\sum_{j=i+1}^Nr_j.
	\end{align*}
Moreover,
	\begin{align*}
 		H_i^{[m]}-H_{i+1}^{[m]}=
 		\begin{cases}
 			r_{i+1}H_{i+1}^{[m-1]},&i+1\ge i_*^{[m-1]},\\
 			r_{i+1},&i+1<i_*^{[m-1]},
 		\end{cases}
	\end{align*}
	(the case $i+1=i_*^{[m-1]}$ falls under the first line, and the
	case $i+2=i_*^{[m-1]}$ under the second).
Thus $H_i^{[m]}$ is nonnegative and nonincreasing, which closes the
	induction; the problem is monotone at every level $\ell=1,\ldots,m$.
Theorem~\ref{monotonetheorem2} then gives the stated recursively defined
	optimal vector.
\end{proof}

\begin{corollary}\label{corollary31}
The optimal thresholds satisfy
	\begin{align*}
		 1\le i_*^{[m]}\le i_*^{[m-1]}\le\cdots\le i_*^{[1]}\le N.
	\end{align*}
\end{corollary}
\begin{proof}
Equation~\eqref{oddsgim} and induction give $G_i^{[m]}\ge G_i^{[m-1]}$ for all $i$, 
	because the additional terms are
	nonnegative on the level-$(m-1)$ stopping region.
Hence the level-$m$ stopping region contains the level-$(m-1)$ stopping
	region, which yields $i_*^{[m]}\le i_*^{[m-1]}$.
\end{proof}

Starting from $H_i^{[1]}$, recursion \eqref{Hrecursion} determines all
	higher-level functions and thresholds.  
Let $P_N^{[m]}(\mathrm{win})$ denote the optimal winning probability with $m$
	selection opportunities.
Bruss~\cite{bruss2003} proved the lower bound for single stopping: 
 	$P_N^{[1]}(\mathrm{win})\ge e^{-1}$ whenever $\sum_{j=1}^Nr_j\ge1$. 
For $m=2$, Ano, Kakinuma and Miyoshi~\cite{anokakinumamiyoshi} obtained,
	under suitable conditions,
 	$\lim_{N\to\infty}P_N^{[2]}(\mathrm{win})=e^{-1}+e^{-3/2}$.

For completeness, we state explicitly the equality system used by Matsui and Ano.  
For $k\ge1$, let $\Xi_k$ be the set of vectors
	$b=(b_k,\ldots,b_1)\in\mathbb Z_+^k$ for which there exists
	$k^*\in\{1,\ldots,k\}$ satisfying
	\begin{align*}
 		\sum_{j=\ell}^kb_j&<1+k-\ell,\qquad \ell=k,k-1,\ldots,k^*+1,\\
 		\sum_{j=k^*}^kb_j&=1+k-k^*, \qquad b_{k^*-1}=\cdots=b_1=0.
	\end{align*}

\begin{theorem}[Matsui and Ano~\cite{matsuiano}]\label{th32}
There is a unique positive sequence $(\lambda_k)_{k\ge1}$ satisfying,
for every $k\ge1$,
	\begin{align}\label{equalitysystem}
 		\sum_{b\in\Xi_k}\prod_{\ell=1}^k\frac{\lambda_\ell^{b_\ell}}{b_\ell!}=1.
	\end{align}
In particular,
	$\lambda_1=1$,
	$\lambda_1+\lambda_2=3/2$, and
	$\lambda_1+\lambda_2+\lambda_3=47/24$.
If
	$\prod_{i=1}^Nq_i <\exp\{-\sum_{j=1}^m\lambda_j\}$, 
	then
	\begin{align}\label{boundsofwin}
		 P_N^{[m]}(\mathrm{win})\ge\sum_{k=1}^m
	 	\exp\left\{-\sum_{j=1}^k\lambda_j\right\}.
	\end{align}
For the classical secretary problem,
	$i_*^{[k]}/N\to\exp\{-\sum_{j=1}^k\lambda_j\}$, and the bound is
	asymptotically attained.
\end{theorem}
\subsubsection{A Poisson--Bellman proof of Theorem~\ref{th32}}\label{sec:poisson-proof}
We give an alternative proof of Theorem~\ref{th32}. 
The argument is probabilistic and is based on a continuous-time last-arrival problem, its
	dynamic programming equation, and a Poisson-block comparison. 
The embedding is one-dimensional in time and uses only arrival epochs.

The idea of the reduction is as follows.  
Since the objective is to select the \emph{last} success, the only feature of the trial sequence
	that matters is the pattern of success epochs.  
Such a pattern can be generated by a rate-one Poisson process: a block of length
	$\kappa_i:=-\log q_i$ is empty with probability $e^{-\kappa_i}=q_i$, so the block
	occupancy indicators of consecutive blocks of lengths $\kappa_1,\ldots,\kappa_N$
	are independent and have exactly the law of $(I_i)$.  
The Bernoulli problem thus becomes a coarsened version of a continuous-time problem in
	which at most $m$ arrivals are selected and the aim is to select the
	last arrival on an interval of total length
	\begin{align*}
 		R=\sum_{i=1}^N \kappa_i=-\log\prod_{i=1}^N q_i .
	\end{align*}
Three facts then yield Theorem~\ref{th32}.  
First, the continuous-time problem is Markov in the remaining time alone, 
	so its dynamic programming equation
	is a scalar differential equation whose successive free boundaries
	$C_1<C_2<\cdots$ are determined uniquely
	(Lemmas~\ref{lemma:poisson-last-arrival}--\ref{lemma:poisson-boundaries}).
Second, iterated integration of the boundary condition
	$\Delta^{[k]}(C_k)=1$ reproduces the equality system \eqref{equalitysystem},
	which identifies $C_k=\lambda_1+\cdots+\lambda_k$
	(Lemma~\ref{lemma:poisson-equality-system}).  
Third, any Poisson strategy can be imitated by a randomized Bernoulli strategy that wins
	whenever the Poisson strategy wins, so the Bernoulli value dominates the
	Poisson value $v^{[m]}(R)$ (Lemma~\ref{lemma:poisson-block-comparison}).
Under the hypothesis $\prod_iq_i<e^{-C_m}$, one has $R>C_m$, where
	$v^{[m]}$ has already reached its plateau $\sum_{k\le m}e^{-C_k}$, and the
	lower bound follows.  Note that the embedding uses $-\log q_i$ rather
	than Bruss' odds $p_i/q_i$; the two agree asymptotically when
	$\max_ip_i\to0$, which is the regime of the secretary limit.

\begin{lemma}\label{lemma:poisson-last-arrival}
Let a rate-one Poisson process be observed on an interval of length $t$.
The decision maker may select at most $m$ arrival epochs and wins when the
	last arrival, if any, is among the selected epochs.  
Let $v^{[m]}(t)$ be the optimal winning probability before the observation begins, with the convention
	$v^{[0]}(t)=0$.  
Then $v^{[m]}(0)=0$ and
	\begin{align}
		 v^{[m]\prime}(t)=\left[e^{-t}+v^{[m-1]}(t)-v^{[m]}(t)\right]^+,\qquad t>0.                                   
		   \label{poisson-bellman-ode}
	\end{align}
At an arrival epoch with $t$ units of time remaining, it is optimal to
	select the arrival precisely when
	\begin{align}
		 e^{-t}+v^{[m-1]}(t)\ge v^{[m]}(t).                  
		   \label{poisson-stop-test}
	\end{align}
Equivalently, let the observation interval have length $T$, let
	$0<T_1<T_2<\cdots$ be the arrival epochs of the Poisson process, and let
	$\mathcal S^{[m]}:=\{t\ge0:\ e^{-t}+v^{[m-1]}(t)\ge v^{[m]}(t)\}$ be the level-$m$
	selection set in terms of the remaining time.  
Then the optimal $m$-selection rule is given recursively by
\begin{align*}
 	\tau^{[m]}&:=\inf\{T_j:\ T-T_j\in\mathcal S^{[m]}\},\\
 	\tau^{[k]}&:=\inf\{T_j>\tau^{[k+1]}:\ T-T_j\in\mathcal S^{[k]}\},\quad k=m-1,\ldots,1,                          
\end{align*}
	with the convention $\inf\emptyset=\infty$, where $\tau^{[k]}=\infty$ means
	that the corresponding right is never used.
\end{lemma}
\begin{proof}
Let $(M_u)_{u\ge 0}$ denote a rate-one Poisson counting process, so that
	the number of arrivals in any interval of length $u$ is distributed as
	$M_u$ and $P(M_u=0)=e^{-u}$.
Observed at the arrival epochs, the decision problem is a Markov decision
	process whose state consists of the remaining time and the number of
	remaining rights; the rewards are bounded and the number of decision
	epochs is almost surely finite, so that the principle of optimality applies.
Suppose that the first arrival occurs when $s$ units of time remain.
If this arrival is selected, the decision maker wins in exactly one of
	two mutually exclusive ways: either no further arrival occurs in the
	remaining time, an event of probability $P(M_s=0)=e^{-s}$, or at least
	one further arrival occurs and the last of these later arrivals is selected
	using one of the remaining $m-1$ rights.
By the independence of the increments of the Poisson process, the optimal
	probability of the second alternative is $v^{[m-1]}(s)$, which already accounts for the
	probability $P(M_s\ge 1)=1-e^{-s}$ that a later arrival exists.
The value of selecting is therefore $e^{-s}+v^{[m-1]}(s)$.
The value of rejecting is $v^{[m]}(s)$. 

Let $U$ denote the waiting time until the first arrival.  
Since the arrival process is Poisson with unit rate, $U$ is exponentially
	distributed with $P(U\in du)=e^{-u}\,du$ and $P(U>t)=P(M_t=0)=e^{-t}$.
If $U>t$, no arrival occurs within the remaining time and the decision
	maker loses, so this event contributes nothing.  
If $U=u\le t$, the first arrival occurs when $s=t-u$ units of time remain, and the decision
	maker then acts optimally, obtaining
	$\max\{e^{-s}+v^{[m-1]}(s),\,v^{[m]}(s)\}$.
Conditioning on $U$ therefore gives
	\begin{align*}
 		v^{[m]}(t)&= E\bigl[\max\{e^{-(t-U)}+v^{[m-1]}(t-U),\,v^{[m]}(t-U)\}
    			 \,\mathbf 1_{\{U\le t\}}\bigr]\\
 			&= \int_0^t e^{-u}\max\{e^{-(t-u)}+v^{[m-1]}(t-u),\,v^{[m]}(t-u)\}\,du .
	\end{align*}
Substituting $s=t-u$, the remaining time at the first arrival, and noting
	that $e^{-u}=e^{-t}e^{s}$, we obtain
	\begin{align}
	 	v^{[m]}(t)=e^{-t}\int_0^t e^{s}\max\{e^{-s}+v^{[m-1]}(s),\,v^{[m]}(s)\}\,ds .
 		\label{poisson-first-arrival}
	\end{align}
Here $e^{-(t-s)}=e^{-t}e^{s}$ is the probability that no arrival occurs
	while the remaining time decreases from $t$ to $s$, so that the integrand
	is the probability density of the first arrival at remaining time $s$,
	weighted by the optimal value at that instant.

The right-hand side is absolutely continuous.  
Differentiating \eqref{poisson-first-arrival} yields
	\begin{align*}
	 	v^{[m]\prime}(t)&=-v^{[m]}(t)+\max\{e^{-t}+v^{[m-1]}(t),v^{[m]}(t)\}\\	 	
			&=\left[e^{-t}+v^{[m-1]}(t)-v^{[m]}(t)\right]^+,
	\end{align*}
	which proves both assertions.
For a given continuous $v^{[m-1]}$, the right-hand side of
	\eqref{poisson-bellman-ode} is Lipschitz in $v^{[m]}$ with constant one,
	uniformly in $t$.  Hence, as in Step~3 of the proof of
	Theorem~\ref{th:uniform-random-horizon-odds}, the Gronwall inequality
	applied to the difference of two solutions shows that
	\eqref{poisson-bellman-ode} with $v^{[m]}(0)=0$ has a unique solution on
	every compact interval.  
In particular, any function constructed to satisfy \eqref{poisson-bellman-ode}, 
	as in Lemma~\ref{lemma:poisson-boundaries}, coincides with the value function $v^{[m]}$.
\end{proof}

For the analysis of the free boundaries, set
	\begin{align*}
	 	w^{[m]}(t):=e^t v^{[m]}(t),\qquad
 		\Delta^{[m]}(t):=w^{[m]}(t)-w^{[m-1]}(t),\qquad w^{[0]}\equiv0.
	\end{align*}
The stopping advantage in \eqref{poisson-stop-test} can then be written as
	\begin{align*}
 		e^{-t}+v^{[m-1]}(t)-v^{[m]}(t)=e^{-t}\{1-\Delta^{[m]}(t)\}.                         
	\end{align*}
Since $e^{-t}>0$, the selection rule \eqref{poisson-stop-test} is
	equivalent to $\Delta^{[m]}(t)\le 1$. 
As shown in the next lemma, $\Delta^{[m]}$
	starts from $\Delta^{[m]}(0)=0$ and crosses the level one exactly once, so
	that the free boundary at level $m$ is the unique point $C_m$ at which
	$\Delta^{[m]}(C_m)=1$.

\begin{lemma}\label{lemma:poisson-boundaries}
There are unique constants $0=C_0<C_1<C_2<\cdots$
	such that the level-$m$ Poisson rule selects an arrival if and only if the
	remaining time is at most $C_m$.  
Moreover, if
 	$\lambda_m:=C_m-C_{m-1},$ 
	then $\lambda_m>0$ and 
	\begin{align}
 		\Delta^{[m]}(t)=t+\int_0^{t\wedge C_{m-1}}
      		 \{\Delta^{[m-1]}(s)-1\}\,ds,
		 \qquad 0\le t\le C_m,     \label{delta-recursion}
	\end{align}
	where $\Delta^{[0]}\equiv0$ is used only as a notational convention.  
For $t\ge C_m$, instead, $\Delta^{[m]}(t)=e^{t-C_m}$, 
	because both $v^{[m]}$ and $v^{[m-1]}$ are constant there. 
The boundary is characterized by
	\begin{align}
 		\Delta^{[m]}(C_m)=1.            \label{delta-boundary}
	\end{align}
Finally, the limiting Poisson winning probability
	$p^{[m]}:=\lim_{t\to\infty}v^{[m]}(t)$ satisfies
	\begin{align}
 		p^{[m]}=p^{[m-1]}+e^{-C_m}=\sum_{k=1}^m e^{-C_k}.                      
		 \label{poisson-value-sum}
	\end{align}
\end{lemma}
\begin{proof}
For $m=1$, equation \eqref{poisson-bellman-ode} gives
	$v^{[1]\prime}(t)=e^{-t}-v^{[1]}(t)$ as long as stopping is preferred.  
Hence $v^{[1]}(t)=te^{-t}$ and $\Delta^{[1]}(t)=t$ up to the first contact with one.
Consequently, $C_1=1$.

We now proceed inductively.  
Before the level-$m$ boundary is reached,
	\eqref{poisson-bellman-ode} and the definition of $w^{[m]}$ give
	\begin{align*}
	 	w^{[m]\prime}(t)=1+w^{[m-1]}(t).                              
	\end{align*}
After the level-$(m-1)$ boundary has been crossed, $v^{[m-1]}$ is constant,
	and hence $w^{[m-1]\prime}=w^{[m-1]}$.  Therefore,
	\begin{align*}
 	\Delta^{[m]\prime}(t)=
		\begin{cases}
   			\Delta^{[m-1]}(t),&0<t<C_{m-1},\\
   			1,&C_{m-1}<t<C_m.
  		\end{cases}
	\end{align*}
Integration gives \eqref{delta-recursion}; by the uniqueness
	statement at the end of the proof of
	Lemma~\ref{lemma:poisson-last-arrival}, the function constructed in
	this way is the value function.
It remains to verify that the new boundary lies strictly to the right of
	the preceding one.  We prove simultaneously that
	\begin{align}
 		0\le \Delta^{[m]}(t)<\Delta^{[m-1]}(t)\le1,
 		\qquad 0<t\le C_{m-1},                            \label{delta-order}
	\end{align}
	and that $\Delta^{[m]\prime}(t)\ge\Delta^{[m]}(t)$ for
	$0<t<C_m$.  For $m=1$, the latter assertion follows from
	$\Delta^{[1]}(t)=t$ on $[0,1]$.  
Suppose it holds at level $m-1$.  
On $[0,C_{m-1}]$,
	\begin{align*}
	 	\frac{d}{dt}\{\Delta^{[m-1]}(t)-\Delta^{[m]}(t)\}
 		=\Delta^{[m-1]\prime}(t)-\Delta^{[m-1]}(t)\ge0.
	\end{align*}
Both functions vanish at zero.  Moreover, the derivative
	$\Delta^{[m-1]\prime}-\Delta^{[m-1]}$ is strictly positive almost
	everywhere on $(0,C_{m-1})$: on $(0,C_{m-2})$ it equals
	$\Delta^{[m-2]}-\Delta^{[m-1]}>0$ by \eqref{delta-order} at level $m-1$,
	and on $(C_{m-2},C_{m-1})$ it equals $1-\Delta^{[m-1]}>0$, because
	$\Delta^{[m-1]}$ increases with slope one and reaches the level one only
	at $C_{m-1}$ (for $m=2$ this is $1-t>0$ on $(0,1)$).  
Hence $\Delta^{[m]}<\Delta^{[m-1]}$ on $(0,C_{m-1}]$.  
Moreover, $\Delta^{[m]}(0)=0$ and
	$\Delta^{[m]\prime}=\Delta^{[m-1]}\ge0$ on $(0,C_{m-1})$, so that
	$\Delta^{[m]}\ge0$ there.
This proves \eqref{delta-order}.  
In particular, $\Delta^{[m]}(C_{m-1})<1$. 
Since
	$\Delta^{[m]\prime}(t)=1$ thereafter until the level one is reached,
	$\Delta^{[m]}$ reaches the level one for the first and only time at
	\begin{align*}
		 C_m=C_{m-1}+1-\Delta^{[m]}(C_{m-1}),
	\end{align*}
	so that
	\begin{align}
 		\lambda_m=1-\Delta^{[m]}(C_{m-1})>0.            \label{lambda-delta}
	\end{align}
On $(0,C_{m-1})$, $\Delta^{[m]\prime}=\Delta^{[m-1]}>\Delta^{[m]}$ by
	\eqref{delta-order}, and on $[C_{m-1},C_m]$,
	$\Delta^{[m]\prime}=1\ge\Delta^{[m]}$, which completes the induction.

At $t=C_m$, stopping and continuation are equal.  
Hence $v^{[m]}(C_m)=e^{-C_m}+v^{[m-1]}(C_m)$.
Since $C_m>C_{m-1}$, the preceding value has already reached its
	plateau: $v^{[m-1]}(C_m)=p^{[m-1]}$.  For $t\ge C_m$, the bracket in
	\eqref{poisson-bellman-ode} equals $e^{-t}-e^{-C_m}\le0$, so the
	level-$m$ value remains constant, and therefore
	\begin{align*}
	 	p^{[m]}=e^{-C_m}+p^{[m-1]}.
	\end{align*}
Iteration proves \eqref{poisson-value-sum}.
\end{proof}

In particular, the optimal rule of Lemma~\ref{lemma:poisson-last-arrival}
	has the threshold form
	$\tau^{[m]}=\inf\{T_j:\ T_j\ge T-C_m\}$ and
	$\tau^{[k]}=\inf\{T_j>\tau^{[k+1]}:\ T_j\ge T-C_k\}$, $k=m-1,\ldots,1$.

The next lemma connects the analytic free-boundary construction to the
finite equality system in Theorem~\ref{th32}.

\begin{lemma}
\label{lemma:poisson-equality-system}
Let $C_k=\lambda_1+\cdots+\lambda_k$ be the boundaries constructed in
	Lemma~\ref{lemma:poisson-boundaries}.  
Then, for every $k\ge1$,
	\begin{align}
 	\Delta^{[k]}(C_k)=\sum_{b\in\Xi_k}
  	 \prod_{\ell=1}^k\frac{\lambda_\ell^{b_\ell}}{b_\ell!}.
                                                               \label{delta-Xi}
	\end{align}
Consequently, $(\lambda_k)_{k\ge1}$ is the unique positive solution of
	\eqref{equalitysystem}.
\end{lemma}
\begin{proof}
Partition the time axis into the successive boundary intervals
	\begin{align*}
				\Lambda_j:=[C_{j-1},C_j],\qquad |\Lambda_j|=\lambda_j.
	\end{align*}
On the interior of $\Lambda_j$, levels below $j$ have already entered their
	continuation plateaus, level $j$ is in the final part of its stopping
	region, and all levels above $j$ are still in their stopping regions.
Thus
	\begin{align}
		 \Delta^{[r]\prime}(t)=
		 \begin{cases}
 			 \Delta^{[r]}(t),&r<j,\\
  			1,&r=j,\\
  			\Delta^{[r-1]}(t),&r>j.
 		\end{cases}                                             \label{triangular-system}
	\end{align}
For $r\ge j$, the system \eqref{triangular-system} on $\Lambda_j$ is
	triangular: the level-$j$ equation has the constant source $1$, and each
	higher level is driven by the level immediately below it.  
Write $t=C_{j-1}+x$ with $0\le x\le\lambda_j$.  
Integrating the level-$j$ equation from $C_{j-1}$ gives
	\begin{align*}
 		\Delta^{[j]}(C_{j-1}+x)=\Delta^{[j]}(C_{j-1})+\int_0^x 1\,dy
  		 =\Delta^{[j]}(C_{j-1})+x .
	\end{align*}
Substituting this into the level-$(j+1)$ equation and integrating again,
	\begin{align*}
 		\Delta^{[j+1]}(C_{j-1}+x)&=\Delta^{[j+1]}(C_{j-1})+\int_0^x \Delta^{[j]}(C_{j-1}+y)\,dy\\
   			&=\Delta^{[j+1]}(C_{j-1})+\int_0^x\bigl\{\Delta^{[j]}(C_{j-1})+y\bigr\}\,dy\\
  			 &=\Delta^{[j+1]}(C_{j-1})+x\,\Delta^{[j]}(C_{j-1})+\frac{x^2}{2!}.
	\end{align*}
In general, by induction on $r$,
	\begin{align}
 		\Delta^{[r]}(C_{j-1}+x)=\sum_{a=0}^{r-j}\frac{x^a}{a!}\,\Delta^{[r-a]}(C_{j-1})
    			+\frac{x^{r-j+1}}{(r-j+1)!},\qquad 0\le x\le\lambda_j .                       
		\label{interval-recursion-x}
	\end{align}
Indeed, if \eqref{interval-recursion-x} holds at level $r-1$, then
	\begin{align*}
 	\Delta^{[r]}(C_{j-1}+x)&=\Delta^{[r]}(C_{j-1})+\int_0^x \Delta^{[r-1]}(C_{j-1}+y)\,dy\\
   		&=\Delta^{[r]}(C_{j-1})+\sum_{a=0}^{r-1-j}\Delta^{[r-1-a]}(C_{j-1})\int_0^x\frac{y^a}{a!}\,dy
     			+\int_0^x\frac{y^{r-j}}{(r-j)!}\,dy\\
   		&=\Delta^{[r]}(C_{j-1}) +\sum_{a=0}^{r-1-j}\frac{x^{a+1}}{(a+1)!}\,\Delta^{[r-1-a]}(C_{j-1})
     			+\frac{x^{r-j+1}}{(r-j+1)!},
	\end{align*}
	and re-indexing the sum by $a+1\mapsto a$ yields
	\eqref{interval-recursion-x} at level $r$.  
The sum collects the initial values at $C_{j-1}$ transported upward by repeated integration, and the
	last term is the contribution of the constant source at level $j$.
Setting $x=\lambda_j$ in \eqref{interval-recursion-x} gives
	\begin{align}
 		\Delta^{[r]}(C_j)=\sum_{a=0}^{r-j}\frac{\lambda_j^a}{a!}\Delta^{[r-a]}(C_{j-1})
  		+\frac{\lambda_j^{r-j+1}}{(r-j+1)!}.             
		 \label{interval-recursion}
	\end{align}

Starting from $\Delta^{[r]}(C_0)=0$, unroll \eqref{interval-recursion} downward from $(r,j)=(k,k)$.  
Suppose that the current level at step $j$ is $\nu_j$, with $\nu_k=k$.  
Equation \eqref{interval-recursion} offers two alternatives: either choose
	$b_j\in\{0,\ldots,\nu_j-j\}$, collect the factor $\lambda_j^{b_j}/b_j!$ and
	pass to $\Delta^{[\nu_{j-1}]}(C_{j-1})$ with $\nu_{j-1}:=\nu_j-b_j\ge j$; or take
	the source term, that is, put $b_j:=\nu_j-j+1$, collect
	$\lambda_j^{b_j}/b_j!$ and stop.  
Paths that reach $j=0$ without stopping contribute zero because $\Delta^{[r]}(C_0)=0$.  
Since $\nu_j=k-\sum_{i>j}b_i$, the nonterminal condition $b_j\le \nu_j-j$ is
	equivalent to $\sum_{i=j}^kb_i<1+k-j$, the terminal condition is
	equivalent to $\sum_{i=j}^kb_i=1+k-j$, and after termination the
	remaining coordinates are set to zero.  
Denoting the terminal index by $k^*$, the terminating paths are therefore
	in one-to-one correspondence with the vectors of $\Xi_k$, and the path
	associated with $b\in\Xi_k$ contributes
	$\prod_{\ell=1}^k\lambda_\ell^{b_\ell}/b_\ell!$.  
This proves \eqref{delta-Xi}.

By \eqref{delta-boundary}, the left-hand side of \eqref{delta-Xi} equals
	one, so the constructed sequence satisfies \eqref{equalitysystem}.  
In the $k$th equation the only term containing $\lambda_k$ is the vector 
	$(b_k,b_{k-1},\ldots,b_1)=(1,0,\ldots,0)$; all other terms depend only on
	$\lambda_1,\ldots,\lambda_{k-1}$.  
Hence each equation determines $\lambda_k$ uniquely from the preceding constants.  
Positivity follows from \eqref{lambda-delta}.  
Thus the positive solution is unique.
\end{proof}

For example, the first three boundary equations are
	\begin{align*}
	 	\lambda_1&=1,\\
 		\lambda_2+\frac{\lambda_1^2}{2!}&=1,\\
 		\lambda_3+\frac{\lambda_2^2}{2!}
 			+\frac{\lambda_2\lambda_1^2}{2!}+\frac{\lambda_1^3}{3!}&=1.
	\end{align*}
They give
 	$C_1=1$, $C_2=3/2$ and $C_3=47/24$.

\begin{lemma}
\label{lemma:poisson-block-comparison}
Let $I_1,\ldots,I_N$ be independent Bernoulli variables with
	$P(I_i=0)=q_i\in[0,1]$, put $\kappa_i:=-\log q_i\in[0,\infty]$ and
	$R:=\sum_{i=1}^N\kappa_i=-\log\prod_{i=1}^Nq_i\in[0,\infty]$, and set
	$v^{[m]}(\infty):=p^{[m]}=\lim_{t\to\infty}v^{[m]}(t)$.
Then the optimal $m$-stopping winning probability satisfies
	\begin{align*}
 		P_N^{[m]}(\mathrm{win})\ge v^{[m]}(R).                 
	\end{align*}
\end{lemma}

\begin{proof}
Assume first that $R<\infty$, that is, $q_i>0$ for every $i$; the case
	where some $q_i=0$ is treated at the end.  
Let $M=(M_t)_{t\ge0}$ be a rate-one Poisson counting process, and set
	\begin{align*}
		 a_0=0,\qquad a_i=\sum_{j=1}^i \kappa_j,\qquad
		 B_i=(a_{i-1},a_i],\qquad R=a_N.
	\end{align*}
Then the counts $M(B_i)=M_{a_i}-M_{a_{i-1}}\sim\operatorname{Poisson}(\kappa_i)$
	are independent, and
	\begin{align*}
		 P(M(B_i)=0)=e^{-\kappa_i}=q_i,\qquad P(M(B_i)\ge1)=1-e^{-\kappa_i}=p_i.
	\end{align*}
	(If $q_i=1$, then $B_i=\varnothing$ and $I_i=0$ almost surely.)

We realize this Poisson configuration \emph{sequentially} from the
	Bernoulli data, so that the imitation strategy constructed below is adapted.  
Let $\xi_1,\ldots,\xi_N$ be i.i.d.\ uniform random variables, independent of $(I_i)$.  
If $I_i=0$, declare $B_i$ empty; if $I_i=1$, use $\xi_i$ to generate a
	configuration in $B_i$ distributed as the Poisson configuration in $B_i$
	conditioned on being nonempty.  
The resulting point process has the law of $M$ on $(0,R]$, its restriction
	to $B_1\cup\cdots\cup B_i$ is a function of
	$(I_1,\ldots,I_i,\xi_1,\ldots,\xi_i)$, and it is independent of
	$(I_{i+1},\ldots,I_N,\xi_{i+1},\ldots,\xi_N)$.  
In particular, $I_i$ is the occupancy indicator of $B_i$.

Fix an arbitrary stopping strategy $\pi$ for the Poisson embedding problem,
	which we call a Poisson stopping strategy, and denote its finite
	selection times by $\sigma_1<\cdots<\sigma_L$, where $0\le L\le m$.
Define
	\begin{align*}
		 A_i=\mathbf 1_{\{\exists\,\ell\le L:\ \sigma_\ell\in B_i\}},
	\end{align*}
	so that $A_i$ indicates whether the strategy selects at least one point
	in $B_i$.  
Since all selections of $\pi$ inside $B_i$ are determined by the
	configuration on $(0,a_i]$, $A_i$ is
	$\sigma(I_1,\ldots,I_i,\xi_1,\ldots,\xi_i)$-measurable.  
In the odds problem, select time $i$ once if $A_i=1$.
Even if the Poisson strategy selects two or more points in the same
	interval, these selections correspond to a single selection in the
	discrete problem, whereas the number of selections remaining for the
	Poisson strategy is reduced by the actual number of points it selects.
Consequently,
	\begin{align*}
		 A_i\le I_i,\qquad \sum_{i=1}^N A_i\le L\le m.
	\end{align*}

If the Poisson strategy selects the last arrival, which lies in the last
	nonempty interval $B_J$, then $A_J=1$ and $J$ is the last success, so
	the imitating strategy selects the last success.  
The converse need not hold.  
Therefore, writing $W_{\mathrm{Poisson}}$ and $W_{\mathrm{Odds}}$ for the
	respective winning events, we have
	\begin{align}\label{eq:win-inclusion}
 		W_{\mathrm{Poisson}}\subseteq W_{\mathrm{Odds}}.
	\end{align}
If there are no points, neither strategy is regarded as winning.
For each fixed realization of $\xi_1,\ldots,\xi_N$, the rule $(A_i)$ is a
	deterministic adapted rule for $(I_i)$ with at most $m$ selections, and
	hence its winning probability is at most $P_N^{[m]}(\mathrm{win})$.
Averaging over the independent variables $\xi_i$ and using
	\eqref{eq:win-inclusion}, we obtain
	$P(W_{\mathrm{Poisson}})\le P(W_{\mathrm{Odds}})\le P_N^{[m]}(\mathrm{win})$.  
Taking the supremum over Poisson strategies gives
	\begin{align*}
		 P_N^{[m]}(\mathrm{win})\ge v^{[m]}(R).
	\end{align*}

For example, suppose that the last nonempty interval $B_J$ contains three
	points $T_{J,1}<T_{J,2}<T_{J,3}$.  
If the Poisson strategy selects the last point $T_{J,3}$, the discrete strategy selects $J$.  
If the Poisson strategy selects only $T_{J,1}$ in this interval, the discrete strategy still selects
	$J$ and wins.  
Thus the winning events need not coincide, and the inclusion
	has the direction stated in \eqref{eq:win-inclusion}.
Figure~\ref{fig:poisson-blocks} shows one realization for $N=5$ and $m=2$.
Each vertical band represents $B_i=(a_{i-1},a_i]$.  
From top to bottom, the figure shows the point configuration, the number of points, 
	the success indicators, and the selections in the odds problem.
In this example, the Poisson strategy selects one point in $B_2$ and the
	last point in $B_4$, so the corresponding discrete selection epochs are
	$2$ and $4$.  Since $B_5$ is empty, the last success time is $J=4$, and
	\begin{align*}
		\sigma_2=T_*\in B_4\quad\Longrightarrow\quad J=4\text{ is selected}.
	\end{align*}

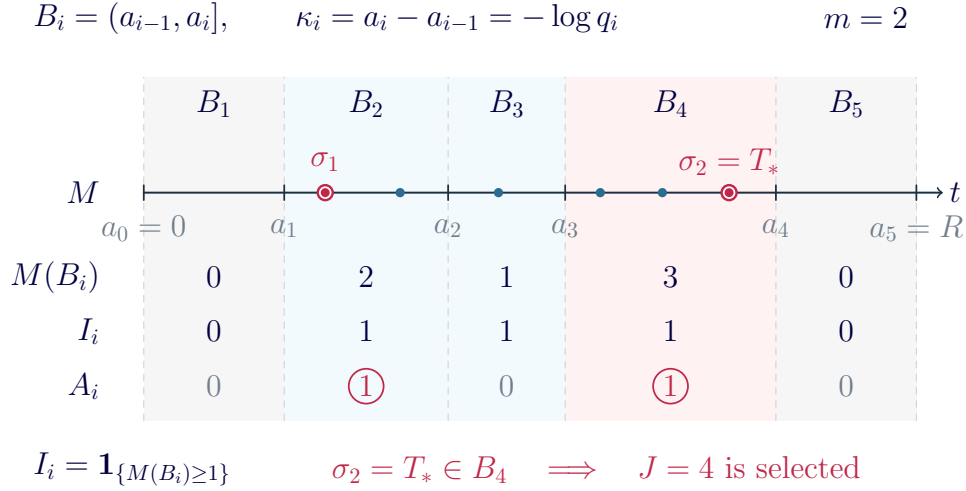
\begin{figure}[htbp]
 \centering
 \resizebox{0.8\linewidth}{!}{%
 \begin{tikzpicture}[x=1cm,y=1cm,
   every node/.style={font=\Large,text=blue!25!black},
   line cap=round]
  \definecolor{poissonnavy}{RGB}{27,47,66}
  \definecolor{poissonblue}{RGB}{41,109,145}
  \definecolor{poissonred}{RGB}{192,40,73}
  \definecolor{poissongray}{RGB}{112,126,139}
  \path[use as bounding box] (-0.3,-1.15) rectangle (15.2,6.9);
  \fill[gray!7] (1.6,0) rectangle (3.85,5.5);
  \fill[cyan!4!white] (3.85,0) rectangle (6.475,5.5);
  \fill[cyan!4!white] (6.475,0) rectangle (8.35,5.5);
  \fill[red!5!white] (8.35,0) rectangle (11.725,5.5);
  \fill[gray!7] (11.725,0) rectangle (13.975,5.5);
  \foreach \x in {1.6,3.85,6.475,8.35,11.725,13.975}{
   \draw[gray!35,dashed,line width=0.5pt] (\x,0)--(\x,5.5);
  }
  \node[anchor=west] at (-0.3,6.45)
   {$B_i=(a_{i-1},a_i],\qquad \kappa_i=a_i-a_{i-1}=-\log q_i$};
  \node[anchor=east] at (13.975,6.45) {$m=2$};
  \foreach \x/\i in {2.725/1,5.1625/2,7.4125/3,10.0375/4,12.85/5}{
   \node at (\x,5.08) {$B_{\i}$};
  }
  \draw[->,poissonnavy,line width=0.9pt] (1.6,3.65)--(14.4,3.65);
  \node at (14.6,3.65) {$t$};
  \node[anchor=east] at (1.03,3.65) {$M$};
  \foreach \x in {1.6,3.85,6.475,8.35,11.725,13.975}{
   \draw[poissonnavy,line width=0.8pt] (\x,3.53)--(\x,3.77);
  }
  \foreach \x in {5.70625,7.28125,8.9125,9.90625}{
   \fill[poissonblue] (\x,3.65) circle (0.074);
  }
  \foreach \x in {4.50625,10.975}{
   \filldraw[fill=white,draw=poissonred,line width=1.1pt]
    (\x,3.65) circle (0.12);
   \fill[poissonred] (\x,3.65) circle (0.072);
  }
  \node[text=poissonred] at (4.50625,4.14) {$\sigma_1$};
  \node[text=poissonred] at (10.975,4.14) {$\sigma_2=T_*$};
  \node[text=poissongray] at (1.6,3.08) {$a_0=0$};
  \foreach \x/\i in {3.85/1,6.475/2,8.35/3,11.725/4}{
   \node[text=poissongray] at (\x,3.08) {$a_{\i}$};
  }
  \node[text=poissongray] at (13.975,3.08) {$a_5=R$};
  \node[anchor=east] at (1.03,2.30) {$M(B_i)$};
  \node[anchor=east] at (1.03,1.43) {$I_i$};
  \node[anchor=east] at (1.03,0.55) {$A_i$};
  \foreach \x/\n/\i in {2.725/0/0,5.1625/2/1,7.4125/1/1,10.0375/3/1,12.85/0/0}{
   \node at (\x,2.30) {$\n$};
   \node at (\x,1.43) {$\i$};
  }
  \foreach \x in {2.725,7.4125,12.85}{
   \node[text=poissongray] at (\x,0.55) {$0$};
  }
  \foreach \x in {5.1625,10.0375}{
   \node[circle,draw=poissonred,text=poissonred,line width=0.8pt,
    inner sep=1.5pt] at (\x,0.55) {$1$};
  }
  \node[anchor=west] at (-0.3,-0.78)
   {$I_i=\mathbf 1_{\{M(B_i)\ge1\}}$};
  \node[text=poissonred] at (8.85,-0.78)
   {$\sigma_2=T_*\in B_4\quad\Longrightarrow\quad J=4\text{ is selected}$};
 \end{tikzpicture}%
 }
 \caption{\footnotesize Correspondence between Poisson intervals and discrete
 	observations and selections ($N=5$, $m=2$).
 	The top row, labeled $M$, shows the arrival points of the constructed counting
 		process; red circles mark the points selected by the strategy.
	 Here $M(B_i)$ is the number of points in the interval,
 	$I_i=\mathbf 1_{\{M(B_i)\ge1\}}$ indicates the success or failure observed
		 in the discrete problem, and $A_i$ indicates whether time $i$ is selected.
 	The interval $B_4$, which contains the last point $T_*$, is the last nonempty 
		interval.  
	If the Poisson strategy selects $T_*$, the discrete strategy
 		selects the last success time $J=4$.}
 \label{fig:poisson-blocks}
\end{figure}

The case where some $q_i=0$ follows by approximation with finite values of $R$.  
Replace only the zero components $q_i$ by $\varepsilon\in(0,1)$,
	leaving all other components unchanged.  
Denote the modified failure
	probabilities by $q_i^{(\varepsilon)}$ and the corresponding optimal
	winning probability by $V^{(\varepsilon)}$.
The comparison established above gives
	\begin{align*}
 		V^{(\varepsilon)}\ge v^{[m]}(R_\varepsilon),\qquad
 		R_\varepsilon=-\sum_{i=1}^N\log q_i^{(\varepsilon)}.
	\end{align*}
Since at least one original component is zero, $R_\varepsilon\to\infty$
	as $\varepsilon\downarrow0$.  
Lemma~\ref{lemma:poisson-boundaries} therefore
	implies $v^{[m]}(R_\varepsilon)\to p^{[m]}$.

For a finite horizon $N$, there are only finitely many deterministic
	strategies: each specifies whether to select or reject at each of the
	finitely many observation histories.  
The winning probability of each strategy is a polynomial in $(p_1,\ldots,p_N)$, and the maximum of these
	finitely many polynomials is continuous.  
Consequently,
	\begin{align*}
		 V^{(\varepsilon)}\longrightarrow P_N^{[m]}(\mathrm{win}).
	\end{align*}
It follows that $P_N^{[m]}(\mathrm{win})\ge p^{[m]}=v^{[m]}(\infty)$, which proves
	the assertion in this case as well.
\end{proof}

\begin{proof}[Alternative proof of Theorem~\ref{th32}]
Lemma~\ref{lemma:poisson-equality-system} proves the existence, uniqueness,
	and positivity of the sequence $(\lambda_k)$ in
	\eqref{equalitysystem}, with $C_k=\sum_{j=1}^k\lambda_j$.
Suppose now that $\prod_{i=1}^Nq_i<e^{-C_m}$.
Then $R=-\log\prod_iq_i>C_m$.  
By Lemma~\ref{lemma:poisson-boundaries}, $v^{[m]}$ is constant on
	$[C_m,\infty)$, and by Lemma~\ref{lemma:poisson-block-comparison},
	\begin{align*}
	 	P_N^{[m]}(\mathrm{win})\ge v^{[m]}(R)=p^{[m]}=\sum_{k=1}^m e^{-C_k}.
	\end{align*}
This is exactly \eqref{boundsofwin}.
It remains to identify the secretary limit.  
For the classical secretary problem, the record indicators are independent and, 
	for $i\ge2$, $p_i=1/i$ and $q_i=1-1/i$. 
Introduce the remaining logarithmic time $t_{i,N}:=-\log(i/N)$. 
For every fixed $0\le a<b<\infty$, the number of record indicators whose
	logarithmic times lie in $(a,b]$ is a sum of independent Bernoulli
	variables with maximal success probability tending to zero and with mean
	converging to $b-a$.  
Counts on disjoint intervals are independent.
Therefore the record point process in logarithmic time converges to a rate-one Poisson process.

The same conclusion is visible directly in the threshold recursion.
Throughout, $\lfloor a\rfloor$ denotes the integer part of $a$, so that
	$\lfloor Ne^{-t}\rfloor$ is the observation index corresponding to the
	logarithmic time $t=-\log(i/N)$.
Using $r_i=1/(i-1)$ in \eqref{Hrecursion}, a Riemann-sum induction gives,
	for every fixed $k$ and locally uniformly in $t$,
	\begin{align}
 		H_{\lfloor Ne^{-t}\rfloor}^{[k]}
 			\longrightarrow \widetilde\Delta^{[k]}(t),      
		 \label{H-to-delta}
	\end{align}
	where
	\begin{align*}
		\widetilde\Delta^{[k]}(t):=t+\int_0^{t\wedge C_{k-1}}
		\{\Delta^{[k-1]}(s)-1\}\,ds,\qquad t\ge0.
	\end{align*}
By \eqref{delta-recursion}, $\widetilde\Delta^{[k]}=\Delta^{[k]}$ on
	$[0,C_k]$; for $t>C_k$, $\widetilde\Delta^{[k]}$ is the linear 
	continuation of $\Delta^{[k]}$ and differs from
	$\Delta^{[k]}(t)=e^{t-C_k}$.
For $k=1$, this is the harmonic-sum limit
	$\sum_{j>Ne^{-t}}1/(j-1)\to t$.  
If \eqref{H-to-delta} holds at level $k-1$, the additional sum in \eqref{Hrecursion} converges to
	\begin{align*}
	 	\int_0^{t\wedge C_{k-1}} \{\Delta^{[k-1]}(s)-1\}\,ds,
	\end{align*}
	which is precisely \eqref{delta-recursion}; only values of
	$H^{[k-1]}$ at indices $j\ge i_*^{[k-1]}$, that is, at logarithmic times
	not exceeding $C_{k-1}$ in the limit, enter this sum, and there
	$\widetilde\Delta^{[k-1]}=\Delta^{[k-1]}$.  
Since $\widetilde\Delta^{[k]}$ is strictly increasing and
	$\widetilde\Delta^{[k]}(t)=1$ has the unique solution $t=C_k$, convergence of the
	crossing points yields
	\begin{align*}
 		\frac{i_*^{[k]}}{N}\longrightarrow e^{-C_k}
			=\exp\left\{-\sum_{j=1}^k\lambda_j\right\}.
	\end{align*}
The Poisson-block comparison supplies only the lower bound.  Thus, to identify
	the secretary limit, it remains to show that there is no asymptotic gain over
	the Poisson value.  In the secretary problem $q_1=0$, so $R=\infty$ and
	Lemma~\ref{lemma:poisson-block-comparison} gives
	$P_N^{[m]}(\mathrm{win})\ge p^{[m]}$ for \emph{every} $N$.  
To obtain the matching upper bound, independent trials are the special case
	$\alpha_i=1-p_{i+1}$, $\beta_i=p_{i+1}$ of the Markov model of
	Section~\ref{subsec:odds-markov}.  
For $p_i=1/i$ the localized conditions \eqref{MP1}, \eqref{MP2} and
	\eqref{MP3} of Theorem~\ref{th:markov-poisson-odds} hold, because
	$\sum_{\lfloor Na\rfloor\le i<\lfloor Nb\rfloor}1/(i+1)\to\log(b/a)$
	locally uniformly and $\max_{i\ge\delta N}1/(i+1)\to0$, and the
	single-crossing condition \eqref{markovmonotonicitycondition} holds
	since $H_i^{[1]}$ is nonincreasing.  
Theorem~\ref{th:markov-poisson-odds} therefore gives
	\begin{align*}
 		P_N^{[m]}(\mathrm{win})\longrightarrow p^{[m]}
 		=\sum_{k=1}^m\exp\left\{-\sum_{j=1}^k\lambda_j\right\}.
	\end{align*}
The proof of Theorem~\ref{th:markov-poisson-odds} uses only
	Lemmas~\ref{lemma:poisson-last-arrival}--\ref{lemma:poisson-equality-system}
	and not Theorem~\ref{th32} itself, so the argument is not circular.
This completes the alternative proof.
\end{proof}

For example, $C_1=1$, $C_2=3/2$, $C_3=47/24$, $C_4=2761/1152$,
$C_5=4162637/1474560$, $C_6\approx3.2410$, and $C_7\approx3.6530$, and,
under the hypothesis of Theorem~\ref{th32} for the corresponding $m$,
	\begin{align*}
 		P_N^{[3]}(\mathrm{win})
		 	&\ge e^{-1}+e^{-3/2}+e^{-47/24}\approx0.7321,\\
		P_N^{[4]}(\mathrm{win})
 			&\ge e^{-1}+e^{-3/2}+e^{-47/24}+e^{-2761/1152}
  		 	\approx0.8231,\\
		P_N^{[5]}(\mathrm{win})
 			&\ge e^{-1}+e^{-3/2}+e^{-47/24}+e^{-2761/1152}
				+e^{-4162637/1474560}\approx0.8825.
	\end{align*}
For three selection opportunities, begin accepting successes after
	approximately $e^{-47/24}N$.
After the first accepted success, accept the first subsequent success
	occurring no earlier than $e^{-3/2}N$; if that threshold has already
	passed, accept the next success immediately.
After the second accepted success, apply the same rule with threshold $e^{-1}N$.
The resulting asymptotic winning probability is approximately $0.7321$.

The deterministic thresholds
	$i_*^{[m]}\le\cdots\le i_*^{[1]}$, calculated recursively from the
	one-step look-ahead functions, give a more explicit description of the
	optimal policy than the abstract contact-time representation by the Snell
	envelopes.

\subsection{Random-horizon odds problem}
\label{subsec:odds-random-horizon}
A natural extension of the classical secretary/last-success setting allows the
	number of observations to be random.  
Presman and Sonin~\cite{presmansonin} initiated the study of this model 
	and identified a single-crossing condition
	under which the optimal rule is monotone.  
Porosi\'nski~\cite{porosinski} studied the corresponding full-information problem.  
Related random-horizon models include the secretary problem with a random observation horizon 
	of Ano and Sakaguchi~\cite{anosakaguchi}, Poisson-arrival models of
	Bruss~\cite{bruss1987} and Ano and Ando~\cite{anoando}, Bayesian Poisson
	models of Kurushima and Ano~\cite{kurushimaano2003} and Gnedin and
	Derbazi~\cite{gnedinderbazi}, and related random-arrival selection models of
	Gnedin~\cite{gnedin2022}.  
We review here the multiple-stopping extension of Kurushima and Ano~\cite{kurushimaano2016} 
	and express it in the notation of Section~\ref{sec:optimal-multiple-stopping}.
\subsubsection{Single stopping}
Let $N_0$ be a nonnegative integer-valued random variable, independent of
	an infinite sequence $(I_i)_{i\ge1}$ of independent Bernoulli random
	variables.  
Write
	\begin{align*}
 		\delta_k&:=P(N_0=k), &&k\in\N,\\
		 \pi_i&:=P(N_0\ge i)=\sum_{k=i}^{\bar N}\delta_k,
   		&&i\ge0,
	\end{align*}
	where $\bar N:=\sup\{k\ge0:\delta_k>0\}\in\N\cup\{\infty\}$.
Only indices $i$ for which $\pi_i>0$ are relevant.  
The event $N_0=0$
	means that no observation is available and the decision maker loses.  For
	$i\ge1$, assume that $0<p_i:=P(I_i=1)<1$, and write $q_i:=1-p_i$ and $r_i:=p_i/q_i$.
Throughout, an empty product $\prod_{k=a}^{b}q_k$ with $a>b$ is
	interpreted as $1$.

Suppose that the $i$th observation exists and is a success.  
If that success is selected, the conditional probability that it is the last
	success before the random horizon is
	\begin{align}\label{randomhorizonreward}
 		x_i^{[1]}=P(I_{i+1}=\cdots=I_{N_0}=0\mid N_0\ge i)
 		=\frac{1}{\pi_i}\sum_{k=i}^{\bar N}\delta_k \prod_{n=i+1}^{k}q_n.
	\end{align}
The probability that, after a success at $i$, the next success occurs at $j>i$ is
	\begin{align*}
		P_{ij}:=\Bigl(\prod_{k=i+1}^{j-1}q_k\Bigr)p_j=\Bigl(\prod_{k=i+1}^{j}q_k\Bigr)r_j.
	\end{align*}
After conditioning also on the survival of the horizon to time $j$, the
	substochastic transition kernel of the embedded success process is
	$P_{ij}\pi_j/\pi_i$, $j>i$.
Its defect accounts for termination of the observation sequence before the
	next success, or for the absence of any later success.

Let $V_i^{[1]}$ be the maximal conditional winning probability at the
	embedded success state $i$, before deciding whether to select that success.
Then
	\begin{align*}
 		V_i^{[1]}=\max\left\{x_i^{[1]},
   		\sum_{j=i+1}^{\bar N}P_{ij}\frac{\pi_j}{\pi_i}V_j^{[1]}\right\}.
	\end{align*}
When $\bar N<\infty$, the terminal condition is
	$V_{\bar N}^{[1]}=x_{\bar N}^{[1]}=1$.
The one-step look-ahead difference is
	\begin{align}\label{psGi1}
 		G_i^{[1]} :=x_i^{[1]}-\sum_{j=i+1}^{\bar N}P_{ij}\frac{\pi_j}{\pi_i}x_j^{[1]}.
	\end{align}
For $j<\bar N$, define
	\begin{align*}
 		d_j:=\delta_j-r_{j+1}\sum_{k=j+1}^{\bar N}\delta_k\prod_{n=j+1}^{k}q_n;
	\end{align*}
	when $\bar N<\infty$, set $d_{\bar N}:=\delta_{\bar N}$.  
The same formula also defines $d_0$.  
A direct rearrangement of \eqref{randomhorizonreward}--\eqref{psGi1} gives
	\begin{align}\label{psHi1}
		 H_i^{[1]}:=\pi_iG_i^{[1]}=\sum_{j=i}^{\bar N}\Bigl(\prod_{n=i+1}^{j}q_n\Bigr)d_j.
	\end{align}
In particular,
	\begin{align}\label{randomhorizonHrelation}
 		H_i^{[1]}=d_i+q_{i+1}H_{i+1}^{[1]}.
	\end{align}
Since $\pi_i>0$, $G_i^{[1]}$ and $H_i^{[1]}$ have the same sign.

We impose the Presman--Sonin single-crossing condition: 
	\begin{align}\tag{PS}\label{PScondition}
 		d_i\ge0\quad\Longrightarrow\quad d_j\ge0
		 \quad\text{for every }j>i.
	\end{align}
The stopping convention at equality is the same as in Section~\ref{sec:optimal-multiple-stopping}.

\begin{theorem}[Kurushima and Ano~\cite{kurushimaano2016}]
\label{th33}
Assume \eqref{PScondition}.  Then the single-stopping problem is monotone,
	and an optimal stopping time is
	\begin{align}\label{randomhorizontau1}
 		\tau_{\mathrm{ola}}^{[1]}&=\inf\{i\ge i_*^{[1]}:N_0\ge i,\ I_i=1\},\\
 		i_*^{[1]}&:=\inf\{i\ge1:\pi_i>0,\ H_i^{[1]}\ge0\}. \notag
	\end{align}
The infimum of an empty set is the cemetery time $\partial$. 
 If $\bar N<\infty$, then $H_{\bar N}^{[1]}=\delta_{\bar N}>0$, so the threshold is finite.
\end{theorem}
\begin{proof}
Suppose that $H_i^{[1]}\ge0$.  If $d_i<0$, then
	\eqref{randomhorizonHrelation} gives $H_{i+1}^{[1]}>0$.  
If $d_i\ge0$, condition \eqref{PScondition} implies that every term in the series defining
	$H_{i+1}^{[1]}$ is nonnegative.  
Hence
	\begin{align*}
 		H_i^{[1]}\ge0\quad\Longrightarrow\quad H_{i+1}^{[1]}\ge0,
	\end{align*}
	so the stopping region has a threshold form.  
The one-step look-ahead rule is therefore optimal by the monotone stopping theorem.
\end{proof}

The following easily checked conditions imply \eqref{PScondition} when the
	support of $N_0$ is finite:
		\begin{align}
 		&\delta_i\text{ is nondecreasing on }\{0,1,\ldots,\bar N\},
 			\tag{A2}\label{A2}\\
 		&r_i\text{ is nonincreasing in }i.\tag{A3}\label{A3}
	\end{align}
To verify this, put
	\begin{align*}
 		A_i:=\sum_{k=i+1}^{\bar N}\delta_k\prod_{n=i+1}^{k}q_n =q_{i+1}(\delta_{i+1}+A_{i+1}).
	\end{align*}
If $d_i=\delta_i-r_{i+1}A_i\ge0$, then
	\begin{align*}
 		\delta_{i+1}\ge\delta_i \ge p_{i+1}(\delta_{i+1}+A_{i+1}),
	\end{align*}
	which implies
	$\delta_{i+1}\ge r_{i+1}A_{i+1}\ge r_{i+2}A_{i+1}$ and hence $d_{i+1}\ge0$.  
Iteration proves \eqref{PScondition}.  
Notice that \eqref{A2}--\eqref{A3} imply the single-crossing property; they do not in general imply
	that the numerical sequence $(d_i)$ itself is nondecreasing.

\subsubsection{Multiple stopping}
Suppose that $m$ selection opportunities are available.  
At an embedded success state $i$, define recursively
	\begin{align*}
 		x_i^{[m]} &=x_i^{[1]}+  \sum_{j=i+1}^{\bar N}P_{ij}\frac{\pi_j}{\pi_i}V_j^{[m-1]},\\
		 V_i^{[m]} &=\max\left\{x_i^{[m]}, \sum_{j=i+1}^{\bar N}P_{ij}\frac{\pi_j}{\pi_i}V_j^{[m]}\right\},
	\end{align*}
	with $V_{\bar N}^{[m]}=1$ when $\bar N<\infty$.  
Here $x_i^{[m]}$ is the winning probability obtained by selecting the current success 
	and then using the remaining $m-1$ opportunities optimally.  
Define
	\begin{align*}
 		G_i^{[m]}:=x_i^{[m]}-\sum_{j=i+1}^{\bar N}P_{ij}\frac{\pi_j}{\pi_i}x_j^{[m]},
		 \qquad H_i^{[m]}:=\pi_iG_i^{[m]}.
	\end{align*}
For integers $a,b$, recall that $a\lor b=\max\{a,b\}$.

\begin{theorem}[Kurushima and Ano~\cite{kurushimaano2016}]
\label{th34}
Assume $\bar N<\infty$ and \eqref{PScondition}.  
Set $H_i^{[1]}$ as in \eqref{psHi1} and $i_*^{[1]}$ as in \eqref{randomhorizontau1}.  
For $\ell=2,\ldots,m$, define recursively
	\begin{align}\label{psHiell}
		 H_i^{[\ell]}&=H_i^{[1]}+ \sum_{j=(i+1)\lor i_*^{[\ell-1]}}^{\bar N}
      			P_{ij}H_j^{[\ell-1]},\\
		 i_*^{[\ell]}&:=\inf\{i\ge1:\pi_i>0,\ H_i^{[\ell]}\ge0\}. \notag
	\end{align}
Then the problem is monotone at every level.  
Define
	\begin{align*}
		\vec\tau_{\mathrm{ola}}^{[m]}&:=(\tau_{\mathrm{ola}}^{[m]},\ldots,
			\tau_{\mathrm{ola}}^{[1]}).
		\end{align*}
This recursively defined vector is optimal, where
	\begin{align}\label{randomhorizonmultitau}
 		\tau_{\mathrm{ola}}^{[m]}&=\inf\{i\ge i_*^{[m]}:N_0\ge i,\ I_i=1\},\\
	 	\tau_{\mathrm{ola}}^{[\ell]}&=\inf\left\{i\ge
  			 \max\{s(\tau_{\mathrm{ola}}^{[\ell+1]}),i_*^{[\ell]}\}:
  			 N_0\ge i,\ I_i=1\right\},
 			\qquad \ell=m-1,\ldots,1. \notag
	\end{align}
An empty defining set gives $\partial$.
\end{theorem}
Here $s(\cdot)$ is the successor map of Section~\ref{subsec:multiple-stopping}, so that
	$s(\tau^{[\ell+1]}_{\mathrm{ola}})=\tau^{[\ell+1]}_{\mathrm{ola}}+1$
	when $\tau^{[\ell+1]}_{\mathrm{ola}}\le\bar N$: the $\ell$th search starts
	strictly after the previous selection.
\begin{proof}
We prove simultaneously, by induction on $m$, that the level-$m$ problem is
	monotone and that
	\begin{align}\label{randomhorizonorderingH}
 		H_i^{[m]}\ge H_i^{[m-1]}
 		\quad\text{for every relevant }i.
	\end{align}
The level-one monotonicity was proved in Theorem~\ref{th33}.  
Assume that the claims hold through level $m-1$.  Corollary~\ref{corollary21}, in
	the form of Remark~\ref{remarkindicator} and applied to the embedded
	success process, yields
	\begin{align*}
 		G_i^{[m]}=G_i^{[1]}+\sum_{j=(i+1)\lor i_*^{[m-1]}}^{\bar N}
     			 P_{ij}\frac{\pi_j}{\pi_i}G_j^{[m-1]}.
	\end{align*}
Multiplication by $\pi_i$ gives precisely
\eqref{psHiell}.

For $m=2$, \eqref{randomhorizonorderingH} follows immediately because the
	additional terms in \eqref{psHiell} are nonnegative.  
For $m\ge3$, the induction hypothesis gives
	$i_*^{[m-1]}\le i_*^{[m-2]}$, and therefore
	\begin{align*}
		H_i^{[m]}-H_i^{[m-1]}&=\sum_{j=(i+1)\lor i_*^{[m-1]}}^{\bar N}
     			P_{ij}H_j^{[m-1]}-\sum_{j=(i+1)\lor i_*^{[m-2]}}^{\bar N}
    			 P_{ij}H_j^{[m-2]}\\
 			&\ge0.
	\end{align*}
Indeed, on the additional range created by the lower threshold,
	$H_j^{[m-1]}\ge0$, and on the common range the difference
	$H_j^{[m-1]}-H_j^{[m-2]}$ is nonnegative.

It remains to prove monotonicity at level $m$.  
If $i+1\ge i_*^{[m-1]}$, then
 	$H_{i+1}^{[m]}\ge H_{i+1}^{[m-1]}\ge0$. 
If $i+1<i_*^{[m-1]}$, the lower limit in the additional sum in
	\eqref{psHiell} is unchanged when $i$ is replaced by $i+1$, and
	$P_{ij}=q_{i+1}P_{i+1,j}$ on that range.  
Consequently,
	\begin{align}\label{randomhorizonlevelrelation}
	 	H_i^{[m]}=d_i+q_{i+1}H_{i+1}^{[m]}.
	\end{align}
If $H_i^{[m]}\ge0$ and $d_i<0$, equation
	\eqref{randomhorizonlevelrelation} gives $H_{i+1}^{[m]}>0$.  If
	$d_i\ge0$, condition \eqref{PScondition} implies
	$H_{i+1}^{[1]}\ge0$, and the additional sum in \eqref{psHiell} is
	nonnegative; hence $H_{i+1}^{[m]}\ge0$.  
Thus
	\begin{align*}
 		H_i^{[m]}\ge0\quad\Longrightarrow\quad H_{i+1}^{[m]}\ge0.
	\end{align*}
The problem is monotone at level $m$, and
	Theorem~\ref{monotonetheorem2} gives the optimal vector
	\eqref{randomhorizonmultitau}.
\end{proof}

\begin{corollary}\label{corollary32}
If $\bar N<\infty$ and \eqref{PScondition} holds, the optimal thresholds satisfy
	\begin{align*}
 		1\le i_*^{[m]}\le i_*^{[m-1]}\le\cdots\le i_*^{[1]}\le\bar N.
	\end{align*}
\end{corollary}
\begin{proof}
Inequality \eqref{randomhorizonorderingH} implies that the level-$m$
	stopping region contains the level-$(m-1)$ stopping region.  
Hence $i_*^{[m]}\le i_*^{[m-1]}$, and iteration gives the result.
\end{proof}

For $\bar N=\infty$, the displayed recursions can be applied to finite truncations.  
Passing to an infinite-support optimal rule additionally
	requires convergence of the truncated value functions and attainment of
	the limiting value by the proposed stopping times.  
These properties do not follow from \eqref{PScondition} alone and are therefore not asserted
	in Theorem~\ref{th34}.

For distributions with $\delta_j>0$, it is often convenient to write
	\begin{align*}
 		S_j:=\sum_{k=j+1}^{\bar N}\frac{\delta_k}{\delta_j}\prod_{n=j+1}^{k}q_n,
 			\qquad d_j=\delta_j(1-r_{j+1}S_j).
	\end{align*}
Thus \eqref{PScondition} is satisfied whenever the sequence
	$r_{j+1}S_j$ crosses the level $1$ at most once, from above to below.  
In the special case $r_{j+1}S_j\le1$ for every $j$, all $d_j$ are
	nonnegative and the single-stopping threshold equals $1$.  
For multiple selections, Theorem~\ref{th34} requires a finite support.  
If $N_0$ is truncated to a finite support $\{0,\ldots,n\}$, that is,
	conditioned on $N_0\le n$, then the ratios $\delta_k/\delta_j$ with
	$k\le n$ are unchanged and $S_j$ only loses terms, so the condition
	$r_{j+1}S_j\le1$ is preserved.  
For the truncated distribution, all thresholds therefore equal $1$, 
	and the optimal policy is to select
	successive successes from the beginning until the available opportunities
	are exhausted or the random horizon ends.  
The stopping times themselves
	remain strictly ordered; it is only their deterministic thresholds that coincide.

\paragraph{Uniform distribution.} 
Suppose that $N_0$ is uniform on $\{1,\ldots,N\}$.  
Then $\delta_0=0$, $\delta_k=1/N$ for $1\le k\le N$, and
	\begin{align*}
 		d_j=\frac1N\left(1-r_{j+1}\sum_{k=j+1}^N\prod_{n=j+1}^{k}q_n\right),
 			\qquad j=1,\ldots,N-1,
	\end{align*}
	with $d_N=1/N$.  Conditions \eqref{A2}--\eqref{A3}, and hence
	\eqref{PScondition}, hold whenever the odds are nonincreasing.  
Equivalently, if $d_j\ge0$, then $\sum_{k=j+1}^N\prod_{n=j+1}^{k}q_n\le 1/r_{j+1}$, 
	and the recurrence
	$\sum_{k=j+1}^N\prod_{n=j+1}^{k}q_n=q_{j+1}
		(1+\sum_{k=j+2}^N\prod_{n=j+2}^{k}q_n)$, together with
	$r_{j+2}\le r_{j+1}$, implies $d_{j+1}\ge0$.

In the secretary problem, $p_i=1/i$ for $i\ge2$ and
	$r_i=1/(i-1)$; the deterministic first record is treated separately, 
	as in Section~\ref{subsec:odds-basic}.  
Although $p_1=1$ violates the standing assumption $0<p_i<1$, this is
	harmless: the first observation is always a record, and the functions
	$H_i^{[\ell]}$, $i\ge1$, and the kernel $P_{ij}\pi_j/\pi_i$, $j>i\ge1$,
	involve only $p_j$, $q_j$ and $r_j$ with $j\ge2$, so that
	Theorems~\ref{th33} and~\ref{th34} apply without change.

Presman and Sonin~\cite{presmansonin} showed for one
	selection that
	\begin{align*}
 		\frac{i_*^{[1]}}{N}\longrightarrow e^{-2}, \qquad
 		P_N^{[1]}(\mathrm{win})\longrightarrow2e^{-2}.
	\end{align*}
For $m\ge2$, Kurushima and Ano~\cite{kurushimaano2016} formulated the
	corresponding asymptotic statement as a conjecture.  
The following theorem proves the secretary specialization of that conjecture for every fixed number
	of selections.

\begin{theorem}
\label{th:uniform-random-horizon-odds}
Let $N_0^{(N)}$ be uniformly distributed on $\{1,\ldots,N\}$ and
	independent of the secretary record indicators
	 $I_1=1$, $P(I_i=1)=1/i$, $i=2,\ldots,N$. 
Suppose that at most $m$ record epochs may be selected and that the decision
	maker wins when one of the selected records is the last record among the
	first $N_0^{(N)}$ observations.
Put $\widehat\rho_0(t)\equiv0$ and $\widehat C_0:=0$.  
Recursively, for $k\ge1$, define
	\begin{align}
 		\widehat g_k(t)&:=t-\frac{t^2}{2}+\int_0^t\widehat\rho_{k-1}(s)\,ds,
       		\label{uniform_odds_g}\\
 		\widehat C_k&:=\textnormal{the unique positive zero of }\widehat g_k,\notag\\
 		\widehat\rho_k(t)&:=\bigl(\widehat g_k(t)\bigr)^+.\notag
	\end{align}
Then
	\begin{align}
 		2=\widehat C_1<\widehat C_2<\cdots,
 		\label{uniform_odds_C_order}
	\end{align}
	and, for every fixed $k=1,\ldots,m$, the optimal threshold
	$i_{*,N}^{[k]}$ satisfies
	\begin{align}
 		\frac{i_{*,N}^{[k]}}{N}\longrightarrow e^{-\widehat C_k}.
 		\label{uniform_odds_threshold_limit}
	\end{align}
Moreover,
	\begin{align}
 		P_N^{[m]}(\mathrm{win}) \longrightarrow  \widehat p_m
 			:=\sum_{k=1}^m\widehat C_ke^{-\widehat C_k}.
			 \label{uniform_odds_probability_limit}
	\end{align}
The boundary constants are equivalently determined by the scalar recursion
	\begin{align}
 		\frac{\widehat C_k^2}{2}-\widehat C_k
 		=\int_0^{\widehat C_{k-1}} \widehat\rho_{k-1}(s)\,ds,\qquad k\ge1,
	 		\label{uniform_odds_C_recursion}
	\end{align}
	where the right-hand side is zero for $k=1$.
\end{theorem}
In particular,
	\begin{align*}
		 		 \widehat C_1&=2, \qquad \widehat C_2 =1+\sqrt{\frac73}
		 	=2.527525231651946\ldots,\\
		 \widehat C_3 &=1+\sqrt{\frac{-\widehat C_2^3+3\widehat C_2^2
          		+4\widehat C_2-1}{3}}=3.010675763770513\ldots,
	\end{align*}
	and hence
	\begin{align*}
		 \widehat p_1 &=2e^{-2} =0.270670566473225\ldots,\\
 		\widehat p_2&=2e^{-2}+\widehat C_2e^{-\widehat C_2}
  			 =0.472509637549045\ldots,\\
 		\widehat p_3 &=2e^{-2}+\widehat C_2e^{-\widehat C_2}
    			 +\widehat C_3e^{-\widehat C_3}
   			=0.620810649828549\ldots.
	\end{align*}
\begin{proof}
We divide the proof into five steps.

\medskip
\noindent
\textit{Step 1.  Logarithmic-time limit of the record process and the random horizon.} 
Write the calendar location as $x=i/N$ and introduce the logarithmic time $t=-\log x$. 
The secretary record indicators are independent and satisfy $P(I_i=1)=1/i$.  
Let $0\le a<b<\infty$.  
The expected number of record indicators in the logarithmic interval $(a,b]$ is
	\begin{align*}
 		\sum_{Ne^{-b}<i\le Ne^{-a}}\frac1i
 		\longrightarrow b-a,
	\end{align*}
	whereas the largest summand tends to zero.  
Counts in disjoint index intervals are independent.  
The law of small numbers therefore implies
	that the record points, expressed in logarithmic time, converge on compact
	sets to a Poisson point process of rate one.

Independently,
	\begin{align*}
   		\frac{N_0^{(N)}}{N}\Longrightarrow U,
 		\qquad U\sim U(0,1).
	\end{align*}
Consequently, $Y:=-\log U$ 
	has the exponential distribution with parameter one and is independent of
	the limiting record process.  A record point at logarithmic time $t$ is
	observable precisely on the event $Y\le t$.

Condition on the observation of a record at time $t$.  
If this record is selected, it is the last observable record when, for some $0\le y\le t$,
	the horizon satisfies $Y\in dy$ and the Poisson process has no point in
	$(y,t)$.  
Hence the immediate winning probability is
	\begin{align*}
 		\widehat a(t)&=\frac{1}{1-e^{-t}}\int_0^t e^{-y}e^{-(t-y)}\,dy
 		=\frac{t}{e^t-1}.
	\end{align*}
Similarly, the substochastic density of the next observable record at $s<t$ is
	\begin{align*}
 		\widehat K(t,ds)&=e^{-(t-s)}\frac{1-e^{-s}}{1-e^{-t}}\,ds
 			=\frac{e^s-1}{e^t-1}\,ds.
	\end{align*}
The defect of this kernel is the probability that the horizon occurs before
	the next record. 
For a bounded measurable function $f$, put
	\begin{align*}
 		(\widehat{\mathcal K}f)(t):=\frac{1}{e^t-1}\int_0^t(e^s-1)f(s)\,ds.
	\end{align*}

\medskip
\noindent
\textit{Step 2.  Limiting dynamic programming equation.} 
Let $\widehat v^{[k]}(t)$ denote the maximal conditional winning probability
	at an observable record point $t$, before deciding whether to select that
	record, when $k$ selections remain. 
Set $\widehat v^{[0]}\equiv0$.  Selecting
	the current record gives
	 $\widehat a(t)+(\widehat{\mathcal K}\widehat v^{[k-1]})(t)$, 
	whereas rejecting it gives
	$(\widehat{\mathcal K}\widehat v^{[k]})(t)$.  Thus
	\begin{align}
 		\widehat v^{[k]}(t)=\max\left\{ \frac{t}{e^t-1}+(\widehat{\mathcal K}\widehat v^{[k-1]})(t),
  		 (\widehat{\mathcal K}\widehat v^{[k]})(t)\right\}.
 		\label{uniform_odds_Bellman}
	\end{align}
Introduce $\widehat Z_k(t):=(e^t-1)\widehat v^{[k]}(t)$, with $\widehat Z_0\equiv0$. 
Multiplication of \eqref{uniform_odds_Bellman} by $e^t-1$ yields
	\begin{align}
		 \widehat Z_k(t)=\max\left\{t+\int_0^t\widehat Z_{k-1}(s)\,ds,
   		\int_0^t\widehat Z_k(s)\,ds \right\}.
		 \label{uniform_odds_Z_Bellman}
	\end{align}
Define the nonnegative Bellman residual
	\begin{align}
 		\widehat R_k(t):=\widehat Z_k(t)-\int_0^t\widehat Z_k(s)\,ds.
 		\label{uniform_odds_residual}
	\end{align}
It is nonnegative in the stopping region and equal to zero in the
	continuation region.

Suppose first that $t$ belongs to the stopping region and that every
	$s\le t$ also belongs to that region.  
Then the first branch in \eqref{uniform_odds_Z_Bellman} is used on $[0,t]$, 
	and differentiation of \eqref{uniform_odds_residual} gives
	\begin{align*}
 		\widehat R_k'(t)&=1+\widehat Z_{k-1}(t)-\widehat Z_k(t)\\
 			&=1-t+\widehat Z_{k-1}(t)-\int_0^t\widehat Z_{k-1}(s)\,ds\\
			&=1-t+\widehat R_{k-1}(t).
	\end{align*}
Since $\widehat R_k(0)=0$, integration gives
	\begin{align*}
 		\widehat R_k(t)=t-\frac{t^2}{2}+\int_0^t\widehat R_{k-1}(s)\,ds.
	\end{align*}
Thus, as long as stopping remains optimal, the residual is
	$\widehat g_k$ in \eqref{uniform_odds_g}; after the free boundary it is zero.  
Therefore, the candidate residual is precisely
	\begin{align*}
 		\widehat R_k(t)=\widehat\rho_k(t) =\bigl(\widehat g_k(t)\bigr)^+.
	\end{align*}

\medskip
\noindent
\textit{Step 3.  Existence, uniqueness and ordering of the free boundaries.} 
For $k=1$, $\widehat g_1(t)=t-t^2/2$, 
	so its unique positive zero is $\widehat C_1=2$, and
	$\widehat\rho_1$ is positive on $(0,2)$ and zero on $[2,\infty)$.
Assume inductively that $\widehat\rho_{k-1}$ is positive on
	$(0,\widehat C_{k-1})$ and zero thereafter, and that
	$\widehat\rho_{k-1}\ge\widehat\rho_{k-2}$.  
From \eqref{uniform_odds_g},
	\begin{align*}
	 	\widehat g_k(t)-\widehat g_{k-1}(t)=\int_0^t
   		\{\widehat\rho_{k-1}(s)-\widehat\rho_{k-2}(s)\}\,ds\ge0.
	\end{align*}
The inequality is strict at $t=\widehat C_{k-1}$.  
Hence $\widehat g_k(t)>0$ for $0<t\le\widehat C_{k-1}$. 
For $t\ge\widehat C_{k-1}$, the integral in
	\eqref{uniform_odds_g} is constant and
	 $\widehat g_k'(t)=1-t<0$. 
Moreover, $\widehat g_k(t)\to-\infty$ as $t\to\infty$.  
Therefore $\widehat g_k$ has exactly one positive zero $\widehat C_k>\widehat C_{k-1}$.  
Taking positive parts in the preceding comparison gives
	$\widehat\rho_k\ge\widehat\rho_{k-1}$, completing the induction and
	proving \eqref{uniform_odds_C_order}.

Because $\widehat\rho_{k-1}$ vanishes after
	$\widehat C_{k-1}$, the equation $\widehat g_k(\widehat C_k)=0$ is equivalent to
	\eqref{uniform_odds_C_recursion}.

We also verify that the constructed boundary solves the dynamic programming equation.
For $0\le t\le\widehat C_k$, the residual is nonnegative, so the first
	branch in \eqref{uniform_odds_Z_Bellman} is optimal.  
For $t\ge\widehat C_k$, define $\widehat Z_k$ by the continuation equation
	$\widehat Z_k(t)=\int_0^t\widehat Z_k(s)\,ds$.
It follows that $\widehat Z_k'(t)=\widehat Z_k(t)$ and hence
	\begin{align}
 		\widehat Z_k(t)=\widehat p_ke^t,\qquad t\ge\widehat C_k,
		 \label{uniform_odds_continuation_Z}
	\end{align}
	for a constant $\widehat p_k$.  
At the boundary the two branches agree.
For $t>\widehat C_k$, the difference between the stopping branch and the
	continuation branch has derivative
	\begin{align*}
 		1+\widehat Z_{k-1}(t)-\widehat Z_k(t)=1-\widehat C_ke^{t-\widehat C_k}<0,
	\end{align*}
	where the identity established in the next step is used. 
Since the difference is zero at $\widehat C_k$, it is negative thereafter.  
Thus continuation is indeed optimal for $t>\widehat C_k$.  
Finally, the dynamic programming equation \eqref{uniform_odds_Z_Bellman} determines its solution
	uniquely on every compact interval, by the Gronwall inequality applied
	to the difference of two solutions; hence the constructed function is
	the value function $\widehat Z_k$, and $[0,\widehat C_k]$ is its exact
	stopping region.

\medskip
\noindent
\textit{Step 4.  Limiting winning probability.} 
From \eqref{uniform_odds_continuation_Z},
	\begin{align*}
 		\widehat v^{[k]}(t)=\frac{\widehat p_ke^t}{e^t-1}=\frac{\widehat p_k}{1-e^{-t}},
		 \qquad t\ge\widehat C_k,
	\end{align*}
	so $\widehat p_k=\lim_{t\to\infty}\widehat v^{[k]}(t)$.
At $t=\widehat C_k$, value matching and the stopping branch in
	\eqref{uniform_odds_Z_Bellman} give
	\begin{align*}
 		\widehat p_ke^{\widehat C_k}&=\widehat C_k+\int_0^{\widehat C_k}\widehat Z_{k-1}(s)\,ds.
	\end{align*}
Since $\widehat C_k>\widehat C_{k-1}$, level $k-1$ is already in its
	continuation region at $\widehat C_k$.  Hence
	\begin{align*}
 		\int_0^{\widehat C_k}\widehat Z_{k-1}(s)\,ds
 		=\widehat Z_{k-1}(\widehat C_k)=\widehat p_{k-1}e^{\widehat C_k}.
	\end{align*}
Consequently,
	\begin{align}
 		\widehat p_k=\widehat p_{k-1}+\widehat C_ke^{-\widehat C_k},
 		\qquad \widehat p_0:=0.
 		\label{uniform_odds_p_recursion}
	\end{align}
Iteration proves the expression in
	\eqref{uniform_odds_probability_limit}.  
It also justifies the derivative calculation at the end of Step~3. 
For the first constants,
	\begin{align*}
 		\widehat\rho_1(t)=\begin{cases}
   			t-t^2/2,&0\le t\le2,\\
   			0,&t>2.
  			\end{cases}
	\end{align*}
Therefore $\int_0^2\widehat\rho_1(t)\,dt=2/3$, 
	and \eqref{uniform_odds_C_recursion} gives $\widehat C_2=1+\sqrt{7/3}$.  
Moreover,
	\begin{align*}
 		\widehat\rho_2(t)=\begin{cases}
  				t-t^3/6,&0\le t\le2,\\[1mm]
  				t-t^2/2+2/3,&2\le t\le\widehat C_2,\\
  				0,&t>\widehat C_2.
 				\end{cases}
	\end{align*}
Thus
	\begin{align*}
 		\int_0^{\widehat C_2}\widehat\rho_2(t)\,dt
			=-\frac{\widehat C_2^3}{6}+\frac{\widehat C_2^2}{2}
   			+\frac{2\widehat C_2}{3}-\frac23.
	\end{align*}
Substitution in \eqref{uniform_odds_C_recursion} yields the displayed
	formula for $\widehat C_3$, and \eqref{uniform_odds_p_recursion} yields the stated values of
	$\widehat p_1,\widehat p_2$ and $\widehat p_3$.

\medskip
\noindent
\textit{Step 5.  Convergence of the finite-horizon dynamic programming equations.}
For the uniform horizon and the secretary record indicators, the immediate
	stopping reward at a record epoch $i$ is
	\begin{align*}
		 x_{i,N}^{[1]}=\frac{i}{N-i+1}\sum_{n=i}^N\frac1n.
	\end{align*}
The probability that the next observable record after $i$ occurs at $j>i$ is
	\begin{align*}
		 P_{ij}^{(N)}=\frac{i(N-j+1)}{(N-i+1)j(j-1)}.
	\end{align*}
These quantities satisfy the exact identity
	\begin{align}
		 x_{i,N}^{[1]}+\sum_{j=i+1}^N P_{ij}^{(N)}=1.
		 \label{uniform_odds_kernel_defect}
	\end{align}
Let $V_{i,N}^{[k]}$ be the finite-$N$ value at record state $i$ with $k$
selections remaining.  Its dynamic programming equation is
	\begin{align}
 		V_{i,N}^{[k]} =\max\left\{x_{i,N}^{[1]}+\sum_{j=i+1}^NP_{ij}^{(N)}V_{j,N}^{[k-1]},
 		\ \sum_{j=i+1}^NP_{ij}^{(N)}V_{j,N}^{[k]}\right\}.
		 \label{uniform_odds_finite_Bellman}
	\end{align}

For $t\ge0$, put $i_N(t)=\lfloor Ne^{-t}\rfloor$ and extend the function
	$\widehat a(t)=t/(e^t-1)$ of Step~1
	continuously to $t=0$ by $\widehat a(0):=1$.  
For every fixed $T<\infty$,
	\begin{align}
 		\sup_{0\le t\le T}\left|x_{i_N(t),N}^{[1]}-\widehat a(t)\right|&\longrightarrow0,
		 	\label{uniform_odds_reward_uniform}\\
 		\sup_{0\le t\le T}\left|\sum_{j=i_N(t)+1}^N P_{i_N(t),j}^{(N)}f\left(-\log\frac jN\right)
 			-\frac1{e^t-1}\int_0^t(e^s-1)f(s)\,ds\right|&\longrightarrow0
 			\label{uniform_odds_kernel_uniform}
	\end{align}
	for every continuous $f$ on $[0,T]$, where the integral operator at
	$t=0$ is defined to be zero.  
On every interval $[\delta,T]$ these are ordinary uniform Riemann-sum limits.  
To include the endpoint $t=0$, note that
	\begin{align*}
		 \frac{i}{N}\le x_{i,N}^{[1]}\le1,
	\end{align*}
	because $x_{i,N}^{[1]}$ is the average of $i/n$, $n=i,\ldots,N$.
Hence, when $0\le t\le\delta$ and $i=i_N(t)$,
	$1-x_{i,N}^{[1]}\le1-i/N=O(\delta)+O(N^{-1})$.
By \eqref{uniform_odds_kernel_defect}, the same bound controls the total
	mass of the discrete kernel, while the total mass of the limiting kernel is
	$1-\widehat a(t)=O(\delta)$.  
Letting first $N\to\infty$ and then $\delta\downarrow0$
	proves \eqref{uniform_odds_reward_uniform}--\eqref{uniform_odds_kernel_uniform}
	on all of $[0,T]$.

We now prove, by induction on $k$, that
	\begin{align}
 		V_{i_N(t),N}^{[k]}\longrightarrow\widehat v^{[k]}(t)
 		\label{uniform_odds_value_convergence}
	\end{align}
	uniformly for $0\le t\le T$.  
Write
	\begin{align*}
 		e_N^{[k]}(t):=\left|V_{i_N(t),N}^{[k]}-\widehat v^{[k]}(t)\right|,
 			\qquad
 		E_N^{[k]}(T):=\sup_{0\le t\le T}e_N^{[k]}(t).
	\end{align*}
For $k\ge1$ we extend $\widehat v^{[k]}$ continuously to $t=0$ by
	$\widehat v^{[k]}(0):=1$: by Step~2, $\widehat Z_k(t)=t+O(t^2)$ as
	$t\downarrow0$, so that
	$\widehat v^{[k]}(t)=\widehat Z_k(t)/(e^t-1)\to1$.  
With this convention, $\widehat v^{[k]}$ and $\widehat v^{[k-1]}$ are continuous on $[0,T]$, 
	and \eqref{uniform_odds_kernel_uniform} applies with $f=\widehat v^{[k]}$ and
	$f=\widehat v^{[k-1]}$.
Using $|\max\{a,b\}-\max\{a',b'\}|\le\max\{|a-a'|,|b-b'|\}$ in
	\eqref{uniform_odds_finite_Bellman} and
	\eqref{uniform_odds_Bellman}, together with
	\eqref{uniform_odds_reward_uniform}--\eqref{uniform_odds_kernel_uniform},
	we obtain
	\begin{align}
 		e_N^{[k]}(t) \le \varepsilon_{N,k}(T)+\sum_{j=i_N(t)+1}^N P_{i_N(t),j}^{(N)}
 			e_N^{[k]}\!\left(-\log\frac jN\right),
 		\label{uniform_odds_error_inequality}
	\end{align}
	where $\varepsilon_{N,k}(T)\to0$ by the induction hypothesis at level
	$k-1$ (with the level-$0$ error equal to zero).

The key point is that the kernel is uniformly contractive on every fixed
	logarithmic interval.  
Since $\widehat a$ is positive on $[0,T]$, put
	\begin{align*}
 		c_T:=\frac12\min_{0\le t\le T}\widehat a(t)>0.
	\end{align*}
By \eqref{uniform_odds_reward_uniform}, for all sufficiently large $N$,
	$x_{i_N(t),N}^{[1]}\ge c_T$ uniformly in $0\le t\le T$.
Therefore \eqref{uniform_odds_kernel_defect} gives
	\begin{align*}
 		\sum_{j=i_N(t)+1}^N P_{i_N(t),j}^{(N)}\le1-c_T
	\end{align*}
	uniformly on $[0,T]$.  
Taking the supremum in \eqref{uniform_odds_error_inequality} yields
	\begin{align*}
		 E_N^{[k]}(T) \le \varepsilon_{N,k}(T)+(1-c_T)E_N^{[k]}(T),
	\end{align*}
	and hence
	\begin{align*}
 		E_N^{[k]}(T)\le c_T^{-1}\varepsilon_{N,k}(T)\longrightarrow0.
	\end{align*}
This proves \eqref{uniform_odds_value_convergence} uniformly on every
	compact logarithmic interval, including the endpoint $t=0$.

By Theorem~\ref{th34}, the finite problem is monotone at every level, so that,
	by \eqref{Vlevel}, the first branch in \eqref{uniform_odds_finite_Bellman}
	is at least the second for $i\ge i_{*,N}^{[k]}$; for $1\le i<i_{*,N}^{[k]}$ it is
	strictly smaller, because there the selection value is strictly smaller than
	its one-step look-ahead continuation, which is at most the second branch.
By \eqref{uniform_odds_reward_uniform}, \eqref{uniform_odds_kernel_uniform} and
	\eqref{uniform_odds_value_convergence}, the difference of the two branches at
	$i=i_N(t)$ converges, uniformly on $[0,T]$, to
	\begin{align*}
		\frac{1}{e^t-1}\left\{t+\int_0^t\widehat Z_{k-1}(s)\,ds
			-\int_0^t\widehat Z_k(s)\,ds\right\},
	\end{align*}
	which, by Step~3, is positive on $(0,\widehat C_k)$ and negative on
	$(\widehat C_k,\infty)$; the crossing is transversal because
	$\widehat g_k'(\widehat C_k)=1-\widehat C_k<0$.
Therefore
	\begin{align*}
 		-\log\frac{i_{*,N}^{[k]}}{N}\longrightarrow\widehat C_k,
	\end{align*}
	which is equivalent to \eqref{uniform_odds_threshold_limit}.

Finally, we identify the unconditional value.  
Fix $T>\widehat C_m$ and put $j_N:=\lfloor Ne^{-T}\rfloor$.  
By \eqref{uniform_odds_threshold_limit}, $i_{*,N}^{[m]}>j_N$ for all large
	$N$, so the optimal rule makes no selection at or before $j_N$; in particular
	the deterministic first record, at logarithmic time $\log N\to\infty$, is
	not selected.  
Let $\sigma_N$ be the first record epoch in $\{j_N,\ldots,N\}$, if any.
Then
	\begin{align*}
		 P(\sigma_N=j_N)=\frac1{j_N},\qquad
 		P(\sigma_N=i)=\frac{j_N-1}{(i-1)i},\quad j_N<i\le N,
	\end{align*}
	and the event that $\sigma_N$ does not exist contributes nothing.  
Since $V_{i,N}^{[m]}$ is the value conditional on $N_0^{(N)}\ge i$, and
	$P(N_0^{(N)}\ge i)=(N-i+1)/N$ is the probability that the record at
	$i$ is observable, the strong Markov property gives the exact identity
	\begin{align*}
 		P_N^{[m]}(\mathrm{win})=\sum_{i=j_N}^N P(\sigma_N=i)\,\frac{N-i+1}{N}\,V_{i,N}^{[m]}.
	\end{align*}
In logarithmic time, $\Theta_N:=-\log(\sigma_N/N)$ converges in law on
	$[0,T]$ to the sub-probability law $e^{-(T-t)}\,dt$ of the first point
	of a rate-one Poisson process below level $T$, while
	$(N-i_N(t)+1)/N\to1-e^{-t}$ and
	$V_{i_N(t),N}^{[m]}\to\widehat v^{[m]}(t)$ uniformly on $[0,T]$ by
	\eqref{uniform_odds_value_convergence}.  
Hence
	\begin{align*}
		 \lim_{N\to\infty}P_N^{[m]}(\mathrm{win})
		 	 =\int_0^Te^{-(T-t)}(1-e^{-t})\,\widehat v^{[m]}(t)\,dt
			 =e^{-T}\int_0^T\widehat Z_m(t)\,dt .
	\end{align*}
Since $T>\widehat C_m$ lies in the continuation region,
	\eqref{uniform_odds_Z_Bellman} and \eqref{uniform_odds_continuation_Z}
	give $\int_0^T\widehat Z_m(t)\,dt=\widehat Z_m(T)=\widehat p_me^T$.  
Therefore $P_N^{[m]}(\mathrm{win})\to\widehat p_m$, which proves
	\eqref{uniform_odds_probability_limit}.  
The observability factor $(N-i+1)/N$ is essential: it turns the conditional
	value $\widehat v^{[m]}$ into $e^{-t}\widehat Z_m$, and the limit is exact
	for every fixed $T>\widehat C_m$.
\end{proof}

\paragraph{Geometric distribution.} 
Let
	$\delta_j=\widehat p\,\widehat q^{\,j}$,
 		$j=0,1,2,\ldots$, where $0<\widehat p<1$ and $\widehat q=1-\widehat p$. 
Then
	\begin{align*}
 			S_j=\sum_{k=j+1}^{\infty}\widehat q^{\,k-j}\prod_{n=j+1}^{k}q_n 
	 			\quad\text{and}\quad 
 			d_j=\widehat p\,\widehat q^{\,j} (1-r_{j+1}S_j).
	\end{align*}
The extra factor $\widehat p$ does not appear inside $S_j$, because
	$\delta_k/\delta_j=\widehat q^{k-j}$.  
The single-crossing property of $r_{j+1}S_j$ is sufficient for the threshold rule.  
In particular, if $r_{j+1}S_j\le1$ for all $j$, the single-stopping
	threshold is $1$; for the distribution truncated to a finite support as
	above, every threshold is $1$ and the optimal policy
	selects the first $m$ successive successes, subject to the random horizon.
\paragraph{Poisson distribution.} 
Let $N_0$ have the Poisson distribution with mean $\lambda>0$:
	 $\delta_j=e^{-\lambda}\lambda^j/j!$, $j=0,1,2,\ldots$. 
Then
	\begin{align*}
 		S_j=\sum_{k=j+1}^{\infty}\frac{\lambda^{k-j}j!}{k!}\prod_{n=j+1}^{k}q_n \quad\text{and}\quad 
 		d_j=e^{-\lambda}\frac{\lambda^j}{j!} (1-r_{j+1}S_j).
	\end{align*}
If both $(r_{j+1})$ and $(S_j)$ are nonincreasing, then
	$r_{j+1}S_j$ is nonincreasing and \eqref{PScondition} follows.  
This is a sufficient condition, not an automatic property of the Poisson model; the
	sequence $(S_j)$ must be checked for the specified Bernoulli probabilities.
The stronger condition $r_{j+1}S_j\le1$ for every $j$ again makes all
	thresholds equal to $1$, for multiple selections after truncation to a finite support.

\paragraph{Negative binomial distribution.} 
Let $\widehat r$ be a positive integer and let
	\begin{align*}
		 \delta_j &=\binom{\widehat r+j-1}{j}  \widehat p^{\,\widehat r}\widehat q^{\,j},
  		 \qquad j=0,1,2,\ldots.
	\end{align*}
For $\widehat r=1$ this reduces to the geometric case.  Here
	\begin{align*}
 		S_j&=\sum_{k=j+1}^{\infty}\frac{(\widehat r+k-1)!\,j!}{(\widehat r+j-1)!\,k!}
 	  		\widehat q^{\,k-j}\prod_{n=j+1}^{k}q_n,\\
 		d_j&=\delta_j(1-r_{j+1}S_j).
	\end{align*}
For fixed $k>j$,
	\begin{align*}
 		\frac{(\delta_k/\delta_{j+1})}
      			{(\delta_k/\delta_j)}=\frac{j+1}{(\widehat r+j)\widehat q}.
	\end{align*}
This ratio may be either smaller or larger than one, so no general
	monotonicity of $(S_j)$ follows without additional assumptions.  
The	criterion to verify is therefore the single-crossing property of
	$r_{j+1}S_j$.  If $r_{j+1}S_j\le1$ for all $j$, then, for the
	distribution truncated to a finite support, all thresholds are
	$1$ and the optimal selections are the successive successes from the
	beginning, rather than coincident stopping times.

\subsection{Last-success problem in Markov-dependent trials}
\label{subsec:odds-markov}
We next consider the last-success problem when the observations form a
	nonhomogeneous two-state Markov chain.  
Hsiau and Yang~\cite{hsiauyang} studied the one-choice problem by backward induction.  
Ano, Kakie and Miyoshi~\cite{anokakiemiyoshi} identified a sufficient single-crossing
	condition under which the optimal rule has a sum-the-odds form and
	extended the construction to multiple selection opportunities.  
Unlike the independent case, monotonicity of the relevant future-odds sum is not
	automatic and must be assumed or verified.

Let $I_1,\ldots,I_N$ take values in $\{0,1\}$, where $I_i=1$ denotes a
	success, and let $\F_i=\sigma(I_1,\ldots,I_i)$.  
It is convenient to introduce a virtual initial success $I_0=1$.  
For $i=0,\ldots,N-1$, put
	\begin{align*}
 		\alpha_i&:=P(I_{i+1}=0\mid I_i=1), &P(I_{i+1}=1\mid I_i=1)&=1-\alpha_i,
	\end{align*}
	and, for $i=1,\ldots,N-1$, put
	\begin{align*}
		 \beta_i&:=P(I_{i+1}=1\mid I_i=0), &P(I_{i+1}=0\mid I_i=0)&=1-\beta_i.
	\end{align*}
Thus $\alpha_0=P(I_1=0)$ specifies the initial distribution. 
We assume
	\begin{align}\label{markovparameterassumption}
 		0<\alpha_i<1\quad(0\le i\le N-1),\qquad 0<\beta_i<1\quad(1\le i\le N-1).
	\end{align}
The parameter $\alpha_0$ only determines the law of $I_1$, and the
	virtual state $0$ is never a stopping state.  Its positivity is used only
	to define the conditional odds $r_{0j}$ at the virtual state; the
	optimality statements below, which involve only the states $i\ge1$,
	remain valid for $\alpha_0\in[0,1]$.
For $i=1,\ldots,N-1$, the transition matrix, with rows and columns ordered
	as $0,1$, is
	\begin{align*}
 		\mathbf P_i=\begin{pmatrix}
  			1-\beta_i&\beta_i\\
  			\alpha_i&1-\alpha_i
 			\end{pmatrix}.
	\end{align*}
The objective is to select the last success, using at most $m$ distinct selection opportunities.

\subsubsection{The embedded success chain and single stopping}
Suppose that the process is currently at a success state $i$, so that
	$I_i=1$; for $i=0$ this refers to the virtual initial state.  Define
	\begin{align*}
 		x_i&:=P(I_{i+1}=\cdots=I_N=0\mid I_i=1)
  	 		=\alpha_i\prod_{k=i+1}^{N-1}(1-\beta_k),
 			&&0\le i<N,\\
	 	x_N&:=1.
	\end{align*}
Thus $x_i$ is the conditional winning probability when the current success is selected.  
For $0\le i<j\le N$, let
	\begin{align*}
 		P_{ij}:=P(I_{i+1}=\cdots=I_{j-1}=0,I_j=1\mid I_i=1),
	\end{align*}
	the transition probability that the next success after $i$ occurs at $j$.  
Explicitly,
	\begin{align*}
 		P_{i,i+1}=1-\alpha_i \quad\text{and}\quad  
 			P_{ij}=\alpha_i\prod_{k=i+1}^{j-2}(1-\beta_k)\,\beta_{j-1},\quad j\ge i+2.
	\end{align*}
The remaining probability is $x_i$, corresponding to no later success,
	so that $x_i+\sum_{j=i+1}^NP_{ij}=1$. 
The natural conditional odds attached to a possible next success at $j$ are
	\begin{align}\label{markovconditionalodds}
		 r_{ij}:=\frac{P_{ij}x_j}{x_i}, \qquad 0\le i<j\le N.
	\end{align}
Equivalently, $r_{ij}=p_{ij}/(1-p_{ij})$, where $p_{ij}$ is the
	conditional probability that $I_j=1$ given $I_i=1$ and given that all
	other observations between $i+1$ and $N$ are failures.  
With the terminal conventions $\alpha_N=1$ and $\beta_N=0$, direct calculation gives
	\begin{align}\label{markovoddsexplicit}
		r_{ij}=\begin{cases}
 		\displaystyle
 		\frac{(1-\alpha_i)\alpha_{i+1}}{\alpha_i(1-\beta_{i+1})},
   		&j=i+1,\\[12pt]
 		\displaystyle
 		\frac{\beta_{j-1}\alpha_j} {(1-\beta_{j-1})(1-\beta_j)},
   		&j\ge i+2.
 		\end{cases}
	\end{align}
This formula includes $j=N$ through the terminal conventions.

Let $V_i^{[1]}$ denote the optimal conditional winning probability at a
	success state $i$ with one selection opportunity remaining.  
The dynamic programming equation for the embedded chain is
	\begin{align}\label{markovsingleDP}
	 	V_i^{[1]}=\max\left\{x_i,\sum_{j=i+1}^NP_{ij}V_j^{[1]}\right\},
 		\qquad 1\le i<N,\qquad V_N^{[1]}=1.
	\end{align}
The one-step look-ahead difference is
	\begin{align*}
 		G_i^{[1]}:=x_i-\sum_{j=i+1}^NP_{ij}x_j=x_i\left(1-H_i^{[1]}\right),\qquad
		 H_i^{[1]}:=\sum_{j=i+1}^Nr_{ij},\qquad H_N^{[1]}:=0.
	\end{align*}
In contrast with independent Bernoulli trials, $H_i^{[1]}$ need not be nonincreasing in $i$.  
We therefore impose the following sufficient condition:
	\begin{align}\tag{M}\label{markovmonotonicitycondition}
 		H_i^{[1]}\ge H_{i+1}^{[1]}, \qquad i=1,\ldots,N-1.
	\end{align}
More generally, it is enough that $H_i^{[1]}-1$ have the single-crossing
	property from positive to nonpositive.  
We use the stronger form \eqref{markovmonotonicitycondition} 
	because it is transparent and easy to check from \eqref{markovoddsexplicit}.

\begin{theorem}[Ano, Kakie and Miyoshi~\cite{anokakiemiyoshi}]
\label{th_markov_single}
Assume \eqref{markovparameterassumption} and \eqref{markovmonotonicitycondition}.  
Define
	\begin{align*}
 		i_*^{[1]}:=\min\{i\in\{1,\ldots,N\}:H_i^{[1]}\le1\}.
	\end{align*}
Then an optimal stopping time is
	 $\tau_{\mathrm{ola}}^{[1]}=\min\{i\ge i_*^{[1]}:I_i=1\}$, 
	where an empty minimum is $\partial$.  
At equality
	$H_i^{[1]}=1$, stopping and continuation have the same value; throughout
	this article we adopt the convention of stopping.
\end{theorem}
\begin{proof}
Equation \eqref{markovconditionalodds} gives
	$G_i^{[1]}=x_i(1-H_i^{[1]})$. 
Under \eqref{markovmonotonicitycondition}, the stopping set
	$\{i:H_i^{[1]}\le1\}$ is an upper interval.  
Hence the embedded problem is monotone, and the one-step look-ahead rule is optimal by
	Theorem~\ref{monotonetheorem}.
\end{proof}

When $\alpha_i+\beta_i=1$ for every $i$, the distribution of
	$I_{i+1}$ is independent of $I_i$.  
In that case
	$r_{ij}$ reduces to the usual Bernoulli odds of trial $j$, and
	Theorem~\ref{th_markov_single} reduces to Bruss' sum-the-odds theorem.

Let $s=i_*^{[1]}-1$.  Conditional on a success at $s$ (with $s=0$
	interpreted as the virtual initial state), every subsequent success lies
	in the stopping region.  
Hence the continuation value is
	\begin{align*}
	 	C_s^{[1]}:=\sum_{j=s+1}^NP_{sj}x_j=x_sR_s,\qquad
	 	R_s:=\sum_{j=s+1}^Nr_{sj}.
	\end{align*}
Ano~\cite{ano2013} proves, under the additional hypotheses imposed in
	that note, the lower bound
	\begin{align*}
		 C_s^{[1]}>R_se^{-R_s}.
	\end{align*}
Consequently, writing $s=s_N$ to indicate the dependence of the
	threshold on the horizon, if $R_{s_N}\to1$ as $N\to\infty$, then the
	rigorous asymptotic conclusion is
	\begin{align*}
		 \liminf_{N\to\infty}C_{s_N}^{[1]}\ge e^{-1}.
	\end{align*}
This is a conditional value at the success state $s$; the unconditional
	winning probability from the beginning also depends on the initial
	distribution through $\alpha_0$ and is obtained from
	\eqref{markovsingleDP}.  
No asymptotic independence of the Markov trials is required or asserted.

\subsubsection{Multiple stopping}\label{sec:msot}
At an actual success state $i\in\{1,\ldots,N\}$, let
	$S_i^{[\ell]}$ be the value of selecting the current success when
	$\ell$ selection opportunities remain, and let $V_i^{[\ell]}$ be the
	optimal value.  
For $0\le i\le N$, let $W_i^{[\ell]}$ be the value of
	rejecting the current state and waiting for the next success; at $i=0$
	this is the value from the virtual initial state.  
Set $W_i^{[0]}:=0$ and $W_N^{[\ell]}:=0$.  Then
	\begin{align}\label{markovmultipleDP}
 		S_i^{[\ell]}&=x_i+W_i^{[\ell-1]},
 			&&1\le i\le N,\\
	 	W_i^{[\ell]}&=\sum_{j=i+1}^NP_{ij}V_j^{[\ell]},
		 	&&0\le i<N, \notag\\
 		V_i^{[\ell]}&=\max\{S_i^{[\ell]},W_i^{[\ell]}\},
 			&&1\le i\le N,\quad \ell\ge1. \notag
	\end{align}
The two terms in $S_i^{[\ell]}$ correspond to disjoint events: the
	current success is the last success, or a later selected success is the last success.  
Thus the value is a winning probability, not the expected number of wins.

Define the level-$\ell$ one-step look-ahead difference by
	\begin{align*}
 		G_i^{[\ell]}:=S_i^{[\ell]}-\sum_{j=i+1}^NP_{ij}S_j^{[\ell]},
	\end{align*}
	in agreement with the definition of $G_i^{[1]}$ above. 
Starting from $H_i^{[1]}=\sum_{j=i+1}^Nr_{ij}$, define recursively, for $\ell=2,\ldots,m$,
	\begin{align}\label{markovHrecursion}
	 	H_i^{[\ell]}=H_i^{[1]}+\sum_{j=(i+1)\lor i_*^{[\ell-1]}}^N
   		r_{ij}\bigl(H_j^{[\ell-1]}-1\bigr),
	\end{align}
	where
	\begin{align*}
 		i_*^{[\ell]}:=\min\{i\in\{1,\ldots,N\}:H_i^{[\ell]}\le1\}.
	\end{align*}
Here $a\lor b=\max\{a,b\}$.
The proof of the next theorem shows that, under \eqref{markovmonotonicitycondition},
	$G_i^{[\ell]}=x_i(1-H_i^{[\ell]})$ for every $\ell$.

\begin{theorem}[Ano, Kakie and Miyoshi~\cite{anokakiemiyoshi}]
\label{th_markov_multiple}
Assume \eqref{markovparameterassumption} and
	\eqref{markovmonotonicitycondition}.  
Then each recursively defined level-$\ell$ problem is monotone, and
	\begin{align*}
 		1\le i_*^{[m]}\le i_*^{[m-1]}\le\cdots\le i_*^{[1]}\le N.
	\end{align*}
The recursively defined vector
	\begin{align*}
 		\tau_{\mathrm{ola}}^{[m]}
		 &=\min\{i\ge i_*^{[m]}:I_i=1\},\\
 			\tau_{\mathrm{ola}}^{[\ell]}
		 &=\min\left\{i\ge
  			 \max\{s(\tau_{\mathrm{ola}}^{[\ell+1]}),i_*^{[\ell]}\}:
   			I_i=1\right\},
 			\qquad \ell=m-1,\ldots,1,
	\end{align*}
	is optimal.  
An empty minimum is the cemetery time $\partial$.
\end{theorem}
\begin{proof}
We prove by induction on $\ell$ that $G_i^{[\ell]}=x_i(1-H_i^{[\ell]})$ and that
	$H_i^{[\ell]}$ is nonnegative and
	nonincreasing in $i\in\{1,\ldots,N\}$, with $H_N^{[\ell]}=0$.  Since
	$x_i>0$, the level-$\ell$
	one-step look-ahead set $\{i:H_i^{[\ell]}\le1\}$ is then an upper
	interval, and because the embedded success chain moves forward in
	calendar time, the problem is monotone at level $\ell$.
For $\ell=1$, the identity for $G_i^{[1]}$ was noted above, $H_i^{[1]}=\sum_{j>i}r_{ij}\ge0$, and
	\eqref{markovmonotonicitycondition} is exactly the monotonicity in $i$.
Suppose inductively that $G_i^{[\ell-1]}=x_i(1-H_i^{[\ell-1]})$, that the level-$(\ell-1)$ problem is monotone with
	threshold $i_*^{[\ell-1]}$, and that $H^{[\ell-1]}$ is nonnegative
	and nonincreasing.  
Corollary~\ref{corollary21}, in the form of Remark~\ref{remarkindicator} 
	and applied to the embedded success chain, yields
	\begin{align*}
	 	G_i^{[\ell]}=G_i^{[1]}+\sum_{j=(i+1)\lor i_*^{[\ell-1]}}^N P_{ij}G_j^{[\ell-1]}.
	\end{align*}
Dividing by $x_i$ and using
	$P_{ij}x_j/x_i=r_{ij}$ shows that $G_i^{[\ell]}=x_i(1-H_i^{[\ell]})$ with
	$H_i^{[\ell]}$ given by \eqref{markovHrecursion}.  
We now show that \eqref{markovmonotonicitycondition} propagates
	through this recursion.  
By \eqref{markovoddsexplicit}, $r_{ij}$ does not depend on $i$ when $j\ge i+2$; put
	\begin{align*}
		u_i:=r_{i,i+1},\qquad \rho_j:=r_{ij}\quad(j\ge i+2),
	\end{align*}
	so that $H_i^{[1]}=u_i+\sum_{j\ge i+2}\rho_j$ and
	\begin{align*}
		H_i^{[1]}-H_{i+1}^{[1]}=u_i+\rho_{i+2}-u_{i+1}\ge0
	\end{align*}
	by \eqref{markovmonotonicitycondition}, with the conventions
	$H_N^{[\ell]}=0$ and that empty sums vanish.  Write
	$i_*:=i_*^{[\ell-1]}$ and $H_j:=H_j^{[\ell-1]}$, so that $H_j\ge0$ for
	all $j$, $0\le H_j\le1$ for $j\ge i_*$, and $H_j\ge H_{j+1}$.
	\begin{itemize}
	\item[(a)] If $i+2<i_*$, every $j$ in the sums defining $H_i^{[\ell]}$ and
		$H_{i+1}^{[\ell]}$ satisfies $j\ge i+3$, so $r_{ij}=r_{i+1,j}=\rho_j$ and
		$H_i^{[\ell]}-H_{i+1}^{[\ell]}=H_i^{[1]}-H_{i+1}^{[1]}\ge0$.
	\item[(b)] If $i+2=i_*$, then
		$H_i^{[\ell]}=u_i+\sum_{j\ge i+2}\rho_jH_j$ and
		$H_{i+1}^{[\ell]}=u_{i+1}H_{i+2}+\sum_{j\ge i+3}\rho_jH_j$, so that
		\begin{align*}
			H_i^{[\ell]}-H_{i+1}^{[\ell]}=u_i+(\rho_{i+2}-u_{i+1})H_{i+2}
			\ge\min\bigl\{u_i,\ H_i^{[1]}-H_{i+1}^{[1]}\bigr\}\ge0,
		\end{align*}
		because $0\le H_{i+2}\le1$.
	\item[(c)] If $i+1\ge i_*$, then $H_i^{[\ell]}=\sum_{j>i}r_{ij}H_j$.  For
		$i=N-1$ this gives $H_{N-1}^{[\ell]}=u_{N-1}H_N=0=H_N^{[\ell]}$; for
		$i\le N-2$,
		\begin{align*}
			H_i^{[\ell]}-H_{i+1}^{[\ell]}=u_iH_{i+1}+(\rho_{i+2}-u_{i+1})H_{i+2}.
		\end{align*}
		If $\rho_{i+2}\ge u_{i+1}$, this is nonnegative; otherwise
		$H_{i+1}\ge H_{i+2}\ge0$ gives the lower bound
		$(u_i+\rho_{i+2}-u_{i+1})H_{i+2}\ge0$.
	\end{itemize}
Finally, $H_i^{[\ell]}\ge0$: in case (c) it is a sum of nonnegative terms,
	and in cases (a) and (b)
	\begin{align*}
		H_i^{[\ell]}=H_i^{[1]}-\sum_{j\ge i_*}r_{ij}\bigl(1-H_j\bigr)
		\ge\sum_{j>i}r_{ij}-\sum_{j\ge i_*}r_{ij}\ge0,
	\end{align*}
	since $0\le 1-H_j\le1$ for $j\ge i_*\ge i+1$.  This closes the
	induction: every level satisfies \eqref{markovmonotonicitycondition} and
	has an upper stopping interval.
The optimality of the recursively defined vector then follows from
Theorem~\ref{monotonetheorem2}.  
Finally, since the problem is monotone at every level,
 	Corollary~\ref{corollary22} gives $G_i^{[\ell]}\ge G_i^{[\ell-1]}$ for
	every $i$ and $\ell$; hence the stopping regions expand with the number
	of remaining opportunities, and the thresholds are nested as stated.
\end{proof}

The unconditional optimal winning probability from the beginning is
	obtained directly from the dynamic program \eqref{markovmultipleDP},
	without any additional closed-form formula:
	\begin{align*}
 		P_N^{[m]}(\mathrm{win})=W_0^{[m]}=\sum_{j=1}^NP_{0j}V_j^{[m]}.
	\end{align*}
This expression correctly incorporates the initial distribution and all
	successes rejected before the first threshold.  In particular, it should
	not in general be replaced by a formula conditioned on a success at $i_*^{[m]}-1$.

Ano, Kakie and Miyoshi~\cite{anokakiemiyoshi} proposed that analogues of
	the independent-trial asymptotic constants may hold for Markov-dependent trials.  
Their paper also notes that the limiting optimal probabilities
	for multiple selections were not determined.  
We therefore record this only as a conjectural direction from that paper
	and do not identify the limits with the bounds in Theorem~\ref{th32}.

\subsubsection{A Poisson-limit counterpart of Theorem~\ref{th32}}
\label{sec:markov-poisson-limit}
The finite-sample lower bound of Theorem~\ref{th32} relies essentially on
	independence and cannot be transferred verbatim to arbitrary Markov-dependent trials.  
We first make this limitation explicit and then give a rigorous
	asymptotic counterpart under a rare-transition regime in which successes
	become isolated and their logarithmic-time point process converges to the
	rate-one Poisson process underlying Theorem~\ref{th32}.
\begin{proposition}
\label{prop:markov-no-universal-bound}
For every $m\ge1$ and every $\varepsilon\in(0,1)$, there exist $N$ and a
	two-state Markov chain $I_1,\ldots,I_N$ satisfying
	\eqref{markovparameterassumption} such that
	\begin{align*}
		P(I_1=\cdots=I_N=0)\le\varepsilon
		\qquad\text{and}\qquad
		P_N^{[m]}(\mathrm{win})\le(m+2)\varepsilon .
	\end{align*}
Consequently, no positive constant depending only on the number $m$ of selection
	opportunities is a lower bound for the optimal last-success probability
	throughout the class of two-state Markov chains satisfying
	\eqref{markovparameterassumption}, even if the probability of no success
	is required to be arbitrarily small.
In particular, the lower bound in Theorem~\ref{th32}, whose hypothesis
	$\prod_{i=1}^Nq_i<e^{-C_m}$ is a bound on the probability of no success,
	does not remain valid 
	without an additional condition that suppresses persistent runs of successes.
\end{proposition}
\begin{proof}
Fix $m\ge1$ and $\varepsilon\in(0,1)$, choose $N$ so large that
	$(1-\varepsilon)^{N-1}\le\varepsilon$, and put $\eta:=\varepsilon^2$.  
Consider the chain with
	\begin{align*}
		\alpha_0=\eta,\qquad \alpha_i=\varepsilon,\qquad \beta_i=\eta,
		\qquad 1\le i\le N-1,
	\end{align*}
	which satisfies \eqref{markovparameterassumption}.  
Then $P(I_1=\cdots=I_N=0)=\eta(1-\eta)^{N-1}\le\varepsilon$.
A selection at a failure never wins.
A success selected at a time $i<N$ can be the last success only if
	$I_{i+1}=0$, which has conditional probability $\alpha_i=\varepsilon$ given $\mathcal F_i$
	on $\{I_i=1\}$.  
A selection at time $N$ can win only on $\{I_N=1\}$.
Since $P(I_{i+1}=1)\le(1-\varepsilon)P(I_i=1)+\eta$ for $1\le i<N$, iteration gives
	\begin{align*}
		P(I_N=1)\le(1-\varepsilon)^{N-1}+\frac{\eta}{\varepsilon}\le2\varepsilon .
	\end{align*}
By optional sampling and a union bound over the at most $m$ selections,
	\begin{align*}
		 P_N^{[m]}(\mathrm{win})\le m\varepsilon+P(I_N=1)
		 \le(m+2)\varepsilon .
	\end{align*}
Since $\varepsilon$ is arbitrary, the remaining assertions follow.
\end{proof}

For a triangular array, write $I_{i,N}$ for the state at time $i$ and use the
	same transition convention as above:
	\begin{align*}
 		\alpha_{i,N}&:=P(I_{i+1,N}=0\mid I_{i,N}=1),\\
 		\beta_{i,N}&:=P(I_{i+1,N}=1\mid I_{i,N}=0).
	\end{align*}
Let $P_{N,\mathrm M}^{[m]}(\mathrm{win})$ denote the optimal probability of
	selecting the last success with at most $m$ selections.  
For each $N$, assume
	that the single-crossing condition \eqref{markovmonotonicitycondition} holds, so
	that the finite optimal rules are the nested threshold rules of 
	Theorem~\ref{th_markov_multiple}.
\begin{theorem}
\label{th:markov-poisson-odds}
Suppose that the triangular array of nonhomogeneous Markov chains satisfies
	the following conditions.
	\begin{align}
 		&\max_{\delta N\le i<N}\beta_{i,N}\longrightarrow0
 			\quad\text{for every }\delta\in(0,1),\tag{MP1}\label{MP1}\\
 		&\sum_{i=\lfloor Na\rfloor}^{\lfloor Nb\rfloor-1}\beta_{i,N}
 			\longrightarrow \log\frac ba  \quad\text{for every }0<a<b\le1,
	 		\tag{MP2}\label{MP2}\\
	 	&\max_{\delta N\le i<N}(1-\alpha_{i,N})\longrightarrow0
	 		\quad\text{for every }\delta\in(0,1).
		 	\tag{MP3}\label{MP3}
	\end{align}
Assume, in addition, that the convergence in \eqref{MP2} is
	uniform on compact subsets of $\{(a,b):0<a\le b\le1\}$ and that the
	finite single-crossing condition
	\eqref{markovmonotonicitycondition} holds for every $N$.

Let $\lambda_1,\lambda_2,\ldots$ be the unique positive solution of the
	equality system \eqref{equalitysystem} in Theorem~\ref{th32}, and put
	$C_k:=\sum_{j=1}^k\lambda_j$, $C_0:=0$; as in Lemma~\ref{lemma:poisson-boundaries},
	$p^{[m]}$ denotes the limiting Poisson winning probability.  
Then, for every fixed $m$,
	\begin{align}
		 P_{N,\mathrm M}^{[m]}(\mathrm{win})&\longrightarrow p^{[m]}=\sum_{k=1}^m e^{-C_k},
		 \label{markov-poisson-value-limit}\\
		 \frac{i_{*,N}^{[k]}}{N}&\longrightarrow e^{-C_k}, \qquad k=1,\ldots,m.
	 	\label{markov-poisson-threshold-limit}
	\end{align}
In particular,
	\begin{align*}
 		C_1&=1,&p^{[1]}&=e^{-1},\\
 		C_2&=\frac32,&p^{[2]}&=e^{-1}+e^{-3/2},\\
 		C_3&=\frac{47}{24},&p^{[3]}&=e^{-1}+e^{-3/2}+e^{-47/24}.
	\end{align*}
\end{theorem}
Conditions \eqref{MP1} and \eqref{MP3} are imposed only on proportional
	windows $\{\delta N,\ldots,N\}$, and no condition on the initial law is
	needed (in particular, $\alpha_{0,N}=0$ is allowed): since no selection is made before
	$j_N=\lfloor Ne^{-T}\rfloor$ and $j_N\to\infty$,
	Lemma~\ref{lem:markov-success-occupancy} gives $P(I_{j_N,N}=1)\to0$
	irrespective of $P(I_{1,N}=1)$.  
Independent trials are the special case $\alpha_i=1-p_{i+1}$,
	$\beta_i=p_{i+1}$; in particular, the classical secretary problem
	$p_i=1/i$ satisfies \eqref{MP1}, \eqref{MP2} and \eqref{MP3},
	although $\max_{1\le i<N}\beta_{i,N}=1/2$ does not tend to zero.

The proof is divided into lemmas so that the roles of rare entries, vanishing
	clusters, point-process convergence and dynamic programming are explicit.
\begin{lemma}
\label{lem:markov-success-occupancy}
Under \eqref{MP1} and \eqref{MP3}, for every $\delta\in(0,1)$,
		\begin{align*}
		\max_{\delta N\le i\le N}P(I_{i,N}=1)\longrightarrow0.
	\end{align*}
Moreover, for every compact interval $[a,b]\subset(0,1]$, the probability
	that two consecutive successes occur at indices in
	$\{\lfloor Na\rfloor,\ldots,\lfloor Nb\rfloor\}$ tends to zero.
\end{lemma}
\begin{proof}
Fix $\delta\in(0,1)$, put $i_0:=\lfloor\delta N/2\rfloor$,
	$\pi_{i,N}:=P(I_{i,N}=1)$, and
	$\eta_N:=\max_{i_0\le i<N}(1-\alpha_{i,N})$, $\bar\beta_N:=\max_{i_0\le i<N}\beta_{i,N}$, 
	so that $\eta_N\to0$ and $\bar\beta_N\to0$ by \eqref{MP1} and
	\eqref{MP3}, applied with $\delta/4$ in place of $\delta$.  
The forward recursion gives, for $i_0\le i<N$,
	\begin{align}\label{markov_forward_inequality}
		 \pi_{i+1,N} =\pi_{i,N}(1-\alpha_{i,N})+(1-\pi_{i,N})\beta_{i,N}
		 \le \eta_N\pi_{i,N}+\beta_{i,N}\le \eta_N\pi_{i,N}+\bar\beta_N.
	\end{align}
For all sufficiently large $N$, $\eta_N<1/2$, and iteration from $i_0$ yields
	\begin{align*}
		 \pi_{i,N}\le \eta_N^{\,i-i_0}+\frac{\bar\beta_N}{1-\eta_N},\qquad i_0\le i\le N.
	\end{align*}
For $i\ge\delta N$ we have $i-i_0\ge\delta N/2\to\infty$, so the right-hand
	side tends to zero uniformly in $\delta N\le i\le N$, proving the first assertion.

For the second assertion, a union bound gives
	\begin{align*}
		 P\bigl(I_{i,N}=I_{i+1,N}=1	\text{ for some }\lfloor Na\rfloor\le i<\lfloor Nb\rfloor\bigr)
		 \le \sum_{i=\lfloor Na\rfloor}^{\lfloor Nb\rfloor-1}P(I_{i,N}=1)(1-\alpha_{i,N}).
	\end{align*}
Apply the preceding bounds with $\delta=a$.  
Summing the first inequality in \eqref{markov_forward_inequality} over the window gives
	$(1-\eta_N)\sum_{i=\lfloor Na\rfloor}^{\lfloor Nb\rfloor-1}\pi_{i,N}
	\le1+\sum_{i=\lfloor Na\rfloor}^{\lfloor Nb\rfloor-1}\beta_{i,N}$,
	which is bounded by \eqref{MP2}.  
Hence the displayed sum is at most a constant multiple of $\eta_N$, which tends to zero.
\end{proof}

Define the indicator of a new success run by
	$J_{i,N}:=\mathbf 1_{\{I_{i-1,N}=0,\ I_{i,N}=1\}}$, $i=2,\ldots,N$,
	and define the logarithmic-time point process
	\begin{align*}
		 \mathcal N_N:=\sum_{i=2}^N J_{i,N}\,\delta_{-\log(i/N)}.
	\end{align*}
\begin{lemma}
\label{lem:markov-point-process}
Under \eqref{MP1}, \eqref{MP2} and \eqref{MP3},
	\begin{align*}
 		\mathcal N_N\Longrightarrow\mathcal N,
	\end{align*}
	where $\mathcal N$ is a rate-one Poisson point process on $(0,\infty)$.
The same limit holds when $J_{i,N}$ is replaced by the success indicator
	$I_{i,N}$.
\end{lemma}
\begin{proof}
Fix disjoint bounded logarithmic intervals
	$B_r=(s_r,t_r]$, $r=1,\ldots,q$.  
The corresponding index sets are
	$A_{r,N}:=\{i:Ne^{-t_r}<i\le Ne^{-s_r}\}$. 
Since
	$E[J_{i,N}] =\beta_{i-1,N}P(I_{i-1,N}=0)$, 
Lemma~\ref{lem:markov-success-occupancy} and \eqref{MP2} imply
	\begin{align*}
		\sum_{i\in A_{r,N}}E[J_{i,N}] \longrightarrow t_r-s_r.
	\end{align*}

We verify convergence of joint factorial moments.  
For ordered, pairwise nonadjacent indices $i_1<\cdots<i_k$ in a fixed compact proportional range,
	the Markov property and Lemma~\ref{lem:markov-success-occupancy} give,
	uniformly in the indices,
	\begin{align*}
 		P(J_{i_1,N}=\cdots=J_{i_k,N}=1)=\prod_{r=1}^k\beta_{i_r-1,N}\{1+o(1)\}.
	\end{align*}
Indeed, after each entry into state $1$, condition \eqref{MP3} returns the
	chain to state $0$ at the next step with probability $1-o(1)$, while before
	each later selected index the probability of being in state $1$ is $o(1)$.
Tuples containing adjacent indices contribute zero, and their product-sum
	contribution is bounded by
	$k\max_i\beta_{i,N}$ times a bounded power of the total intensity; it
	therefore vanishes by \eqref{MP1}.  
Consequently, for nonnegative integers $k_1,\ldots,k_q$,
	\begin{align*}
		 E\prod_{r=1}^q\bigl(\mathcal N_N(B_r)\bigr)_{k_r}
 		\longrightarrow \prod_{r=1}^q(t_r-s_r)^{k_r},
	\end{align*}
	where $(x)_k=x(x-1)\cdots(x-k+1)$.  
These are the joint factorial moments of
	independent Poisson random variables with means $t_r-s_r$.  
Tightness follows from the local boundedness of the mean measures, and hence
	$\mathcal N_N\Rightarrow\mathcal N$.

Finally, Lemma~\ref{lem:markov-success-occupancy} shows that, on every compact
	logarithmic interval, the probability of a success run of length greater than
	one tends to zero.  
Thus the point processes of run beginnings and of all
	successes differ with probability tending to zero on compact sets, proving
	the final assertion.
\end{proof}

We now recall the continuous limiting problem.  
Let $v^{[k]}(t)$ be the optimal probability of selecting the last point of a rate-one Poisson process 
	on an interval of remaining length $t$, with at most $k$ selections, and set
	$v^{[0]}\equiv0$.  
The first-point decomposition gives
	\begin{align*}
 		v^{[k]\prime}(t)=\bigl[e^{-t}+v^{[k-1]}(t)-v^{[k]}(t)\bigr]^+,\qquad v^{[k]}(0)=0.
	\end{align*}
As proved in Section~\ref{sec:poisson-proof}, this equation has a unique
	free boundary $C_k$, the boundaries satisfy $0<C_1<\cdots<C_m$, and
	\begin{align}
 		v^{[m]}(t)=p^{[m]}=\sum_{k=1}^m e^{-C_k}, \qquad t\ge C_m.
		\label{markov-poisson-flat-value}
	\end{align}
\begin{lemma}
\label{lem:markov-conditional-kernel}
For $t\ge0$, let $i_N(t)=\lfloor Ne^{-t}\rfloor$.  Under
\eqref{MP1}, \eqref{MP3} and the locally uniform version of
\eqref{MP2}, for every fixed $T<\infty$,
	\begin{align}
		 \sup_{0\le t\le T}\left|x_{i_N(t),N}-e^{-t}\right|&\longrightarrow0,
 		\label{markov_defect_uniform}\\
 		\sup_{0\le t\le T}\left|\sum_{j=i_N(t)+1}^N P_{i_N(t),j}^{(N)}
		 f\left(-\log\frac jN\right) -\int_0^t e^{-(t-s)}f(s)\,ds \right|&\longrightarrow0
 		\label{markov_kernel_uniform}
	\end{align}
	for every continuous $f$ on $[0,T]$.  
Here $x_{i,N}$ is the probability of no later success conditional on $I_{i,N}=1$, 
	and $P_{ij}^{(N)}$ is the conditional transition probability 
	from the current success at $i$ to the next success at $j$.
\end{lemma}
\begin{proof}
Write
	\begin{align*}
		 		 \bar\beta_N:=\max_{i_N(T)-1\le i<N}\beta_{i,N},\qquad
 		\eta_N:=\max_{i_N(T)-1\le i<N}(1-\alpha_{i,N}),
	\end{align*}
	so that $\bar\beta_N\to0$ and $\eta_N\to0$ by \eqref{MP1} and
	\eqref{MP3}.  
If $i<j-1$, let
	\begin{align*}
 		Q_{ij}^{(N)}:=\alpha_{i,N}\prod_{r=i+1}^{j-2}(1-\beta_{r,N}).
	\end{align*}
This is the conditional probability, given $I_{i,N}=1$, that there is no
	success at the indices $i+1,\ldots,j-1$.  
Uniformly for $i/N$ and $j/N$ in a fixed compact subset of $(0,1]$ with $i<j$,
	\begin{align*}
		 \log Q_{ij}^{(N)} &=\log\alpha_{i,N} -\sum_{r=i+1}^{j-2}\beta_{r,N}
		 +O\!\left(\bar\beta_N\sum_{r=i+1}^{j-2}\beta_{r,N}\right).
	\end{align*}
The sums of the $\beta_{r,N}$ are uniformly bounded on such proportional
	windows by \eqref{MP2}. 
Hence the locally uniform form of
	\eqref{MP2}, together with \eqref{MP1} and \eqref{MP3}, gives
	\begin{align}
 		Q_{i_N(t),i_N(s)}^{(N)}\longrightarrow e^{-(t-s)}
 	\label{markov_survival_uniform}
	\end{align}
	uniformly for $0\le s\le t\le T$, with the evident harmless adjustment
	of one or two endpoint indices caused by the floor functions.
Taking the terminal endpoint gives
	\begin{align*}
		x_{i,N}=\alpha_{i,N}\prod_{r=i+1}^{N-1}(1-\beta_{r,N}),
	\end{align*}
	and therefore \eqref{markov_defect_uniform}.

The adjacent-success mass satisfies
	$P_{i,i+1}^{(N)}=1-\alpha_{i,N}\le\eta_N=o(1)$ uniformly.  
For $j\ge i+2$,
	\begin{align*}
 		P_{ij}^{(N)}=Q_{ij}^{(N)}\beta_{j-1,N}=Q_{ij}^{(N)}-Q_{i,j+1}^{(N)}.
	\end{align*}
Thus the conditional law of the next success is the Stieltjes increment
	of the survival function $Q_{ij}^{(N)}$, apart from the uniformly
	vanishing adjacent atom.  
The uniform survival convergence \eqref{markov_survival_uniform}, 
	followed by summation by parts and the uniform continuity of $f$ on $[0,T]$, yields
	\eqref{markov_kernel_uniform}.  
Near $t=0$ both sides have total mass tending uniformly to zero, 
	since that mass equals $1-x_{i_N(t),N}$ on
	the discrete side and $1-e^{-t}$ on the limit side.  
This also verifies the endpoint $t=0$.
\end{proof}

\begin{lemma}
\label{lem:markov-bellman-convergence}
Under the assumptions of Lemma~\ref{lem:markov-conditional-kernel}, for each
fixed $k$ and every compact interval $[0,T]$, the finite waiting
values $W_{i,N}^{[k]}$ of \eqref{markovmultipleDP} satisfy
	\begin{align}
		 \sup_{0\le t\le T} \left|W_{i_N(t),N}^{[k]}-v^{[k]}(t)\right|\longrightarrow0,
		 \qquad i_N(t):=\lfloor Ne^{-t}\rfloor.
 		\label{markov_waiting_value_convergence}
	\end{align}
Consequently, the stopping-minus-continuation differences at a current
	success converge uniformly on $[0,T]$ to
	\begin{align*}
		 e^{-t}+v^{[k-1]}(t)-v^{[k]}(t).
	\end{align*}
\end{lemma}
\begin{proof}
The correct finite quantity corresponding to the Poisson value
	$v^{[k]}(t)$ is the waiting value $W_i^{[k]}$, not the value
	$V_i^{[k]}$ conditional on a success already being present.  
Indeed, from \eqref{markovmultipleDP},
	\begin{align}
 		W_{i,N}^{[k]} =\sum_{j=i+1}^N P_{ij}^{(N)}
 		\max\left\{x_{j,N}+W_{j,N}^{[k-1]},\,W_{j,N}^{[k]}\right\}.
 		\label{markov_waiting_recursion}
	\end{align}
This is the finite first-success decomposition corresponding exactly to
	\eqref{poisson-first-arrival}.  
Notice also that $W_{N,N}^{[k]}=0=v^{[k]}(0)$.

We argue by induction on $k$, with $W^{[0]}\equiv v^{[0]}\equiv0$.
Fix $T<\infty$ and put
	\begin{align*}
		 e_N^{[k]}(t) :=\left|W_{i_N(t),N}^{[k]}-v^{[k]}(t)\right|, \qquad
 		E_N^{[k]}(T):=\sup_{0\le t\le T}e_N^{[k]}(t).
	\end{align*}
Lemma~\ref{lem:markov-conditional-kernel}, applied to the continuous function
	\begin{align*}
		 F_k(s):=\max\{e^{-s}+v^{[k-1]}(s),\,v^{[k]}(s)\},
	\end{align*}
	and the $1$-Lipschitz property of the maximum give, from
	\eqref{markov_waiting_recursion} and \eqref{poisson-first-arrival},
	\begin{align}
 		e_N^{[k]}(t)\le \rho_{N,k}(T)+\sum_{j=i_N(t)+1}^N P_{i_N(t),j}^{(N)}
 		e_N^{[k]}\!\left(-\log\frac jN\right),
		 \label{markov_error_inequality}
	\end{align}
	where $\rho_{N,k}(T)\to0$.  
Here $\rho_{N,k}(T)$ contains the kernel
	approximation error from Lemma~\ref{lem:markov-conditional-kernel}, the
	uniform error in $x_{j,N}\to e^{-s}$, and the level-$(k-1)$ error
	$E_N^{[k-1]}(T)$, which tends to zero by the induction hypothesis.

By \eqref{markov_defect_uniform}, for all sufficiently large $N$,
	\begin{align*}
		 x_{i_N(t),N}\ge c_T:=\frac12e^{-T},\qquad 0\le t\le T.
	\end{align*}
Since
	\begin{align*}
 		x_{i,N}+\sum_{j=i+1}^N P_{ij}^{(N)}=1,
	\end{align*}
	the embedded transition kernel has mass at most $1-c_T$ uniformly on $[0,T]$.  
Taking the supremum in \eqref{markov_error_inequality} gives
	\begin{align*}
 		E_N^{[k]}(T)\le \rho_{N,k}(T)+(1-c_T)E_N^{[k]}(T),
	\end{align*}
	and therefore
	\begin{align*}
 		E_N^{[k]}(T)\le c_T^{-1}\rho_{N,k}(T)\longrightarrow0.
	\end{align*}
This proves \eqref{markov_waiting_value_convergence}.

Finally, at a current success state $i$ the stopping-minus-continuation
	difference is, by \eqref{markovmultipleDP},
	\begin{align*}
 		S_{i,N}^{[k]}-W_{i,N}^{[k]}=x_{i,N}+W_{i,N}^{[k-1]}-W_{i,N}^{[k]}.
	\end{align*}
The uniform convergence just proved and
	\eqref{markov_defect_uniform} show that this converges uniformly on
	$[0,T]$ to $e^{-t}+v^{[k-1]}(t)-v^{[k]}(t)$.
\end{proof}

\begin{proof}[Proof of Theorem~\ref{th:markov-poisson-odds}]
By Lemma~\ref{lem:markov-bellman-convergence}, the level-$k$ finite
	stopping-minus-continuation difference converges uniformly on compact
	logarithmic intervals to
	\begin{align*}
 		g_k(t):=e^{-t}+v^{[k-1]}(t)-v^{[k]}(t).
	\end{align*}
The Poisson analysis in Section~\ref{sec:poisson-proof} shows that $g_k$ has
	a unique crossing at $C_k$, with a strict sign change.  
The finite single-crossing assumption ensures that each finite stopping region is an upper
	interval in calendar time: by Theorem~\ref{th_markov_multiple} and \eqref{Vlevel},
	$S_{i,N}^{[k]}\ge W_{i,N}^{[k]}$ for $i\ge i_{*,N}^{[k]}$, whereas for
	$1\le i<i_{*,N}^{[k]}$ one has $G_i^{[k]}<0$, that is,
	$S_{i,N}^{[k]}<\sum_{j>i}P_{ij}^{(N)}S_{j,N}^{[k]}\le W_{i,N}^{[k]}$.  
Uniform convergence and the strict limiting crossing therefore imply
	\begin{align*}
 		-\log\frac{i_{*,N}^{[k]}}N\longrightarrow C_k,
	\end{align*}
	which is equivalent to \eqref{markov-poisson-threshold-limit}.

Fix $T>C_m$ and put $j_N=\lfloor Ne^{-T}\rfloor$.  
By the finite single-crossing assumption, the optimal policy is the nested
	threshold vector of Theorem~\ref{th_markov_multiple}, and by
	\eqref{markov-poisson-threshold-limit} its smallest threshold satisfies
	$i_{*,N}^{[m]}>j_N$ for all large $N$.  
The optimal policy therefore makes no selection at or before $j_N$ and
	still holds $m$ rights at time $j_N$.  
Let $W_{j,N}^{0,[m]}$ denote the optimal value at time $j$ when
	$I_{j,N}=0$ and $m$ rights remain.  Then
	\begin{align*}
 		P_{N,\mathrm M}^{[m]}(\mathrm{win})
 		=P(I_{j_N,N}=0)\,W_{j_N,N}^{0,[m]}+O\bigl(P(I_{j_N,N}=1)\bigr),
	\end{align*}
	and the error term is $o(1)$ by Lemma~\ref{lem:markov-success-occupancy}.
The first-success decomposition from a failure state reads
	\begin{align*}
 		W_{j,N}^{0,[m]}=\sum_{n=j+1}^N
 		\Bigl(\prod_{r=j}^{n-2}(1-\beta_{r,N})\Bigr)\beta_{n-1,N}\,
		\max\bigl\{x_{n,N}+W_{n,N}^{[m-1]},\,W_{n,N}^{[m]}\bigr\}.
	\end{align*}
Exactly as in the proof of Lemma~\ref{lem:markov-conditional-kernel}, the
	survival factor converges to $e^{-(T-s)}$ uniformly in
	$0\le s\le T$, where $s=-\log(n/N)$, and by
	Lemma~\ref{lem:markov-bellman-convergence} and
	\eqref{markov_defect_uniform} the maximum converges uniformly to
	$F_m(s)=\max\{e^{-s}+v^{[m-1]}(s),v^{[m]}(s)\}$.  
Summation by parts therefore gives
	\begin{align*}
 		W_{j_N,N}^{0,[m]}\longrightarrow\int_0^Te^{-(T-s)}F_m(s)\,ds
 			=v^{[m]}(T)=p^{[m]},
	\end{align*}
	by \eqref{poisson-first-arrival} and \eqref{markov-poisson-flat-value}.
Thus \eqref{markov-poisson-value-limit} follows.  
Lemma~\ref{lem:markov-point-process} is not needed for the value limit; it
	explains why the limiting problem is the Poisson last-arrival problem.
The formulas for $C_1,C_2,C_3$ and $p^{[1]},p^{[2]},p^{[3]}$ are those established in
Theorem~\ref{th32} and its Poisson--Bellman proof.  This completes the proof.
\end{proof}

\subsection{Bayesian last-success problem with an unknown success probability}
\label{subsec:odds-unknown-p}
Bruss and Louchard~\cite{brusslouchard} studied versions of the odds
	algorithm in which the success probabilities are unknown and are updated
	sequentially.  
We consider here the homogeneous Bayesian model, which also
	provides a simple example showing that the monotonicity condition in
	Section~\ref{sec:optimal-multiple-stopping} need not hold.

Let $p$ have a Beta prior distribution with parameters $\alpha>0$ and
	$\beta>0$, and, conditionally on $p$, let
	$I_1,\ldots,I_N$ be independent Bernoulli random variables with common
	success probability $p$.  
Thus
	\begin{align*}
 		p\sim\operatorname{Beta}(\alpha,\beta),\qquad
 		P(I_i=1\mid p)=p,\quad i=1,\ldots,N.
	\end{align*}
After integrating out $p$, the indicators are exchangeable but are no
	longer independent.  
The objective is to stop on the last success.
Since stopping at a failure cannot result in a win, there is no loss of
	optimality in restricting attention to success epochs.

Write $S_i:=\sum_{k=1}^i I_k$. 
At a success epoch $i$, the embedded state is denoted by $(i,s)$, 
	where $I_i=1$ and $S_i=s$.  
Given this state, the posterior distribution of $p$ is
	\begin{align*}
		 p\mid (i,s)\sim \operatorname{Beta}(\alpha+s,\,\beta+i-s)
	\end{align*}
	(cf.\ DeGroot~\cite{degroot}, Section~9.2).  
We use the notation
	\begin{align*}
 		B(a,b):=\int_0^1x^{a-1}(1-x)^{b-1}\,dx
	\end{align*}
	for the Beta function. 
Let $g_i(s)$ denote the conditional winning probability when one stops at
	the success state $(i,s)$.  
Then
	\begin{align}
 		g_i(s)&=P(I_{i+1}=\cdots=I_N=0\mid i,s)\notag\\
 			&=\int_0^1(1-p)^{N-i}\frac{p^{\alpha+s-1}(1-p)^{\beta+i-s-1}}
      			  {B(\alpha+s,\beta+i-s)}\,dp\notag\\
 			&=\frac{B(\alpha+s,\beta+N-s)}{B(\alpha+s,\beta+i-s)}
			=\frac{\displaystyle\prod_{k=i}^{N-1}(\beta-s+k)}
         			{\displaystyle\prod_{k=i}^{N-1}(\alpha+\beta+k)}.
 			\label{betareward}
	\end{align}
Empty products are interpreted as one, so that $g_N(s)=1$. 
For $j=i+1,\ldots,N$, let $P_{ij}(s)$ be the transition probability from
	the current success state $(i,s)$ to the next success state $(j,s+1)$:
	\begin{align}
		 P_{ij}(s) &:=P(I_{i+1}=\cdots=I_{j-1}=0,\ I_j=1\mid i,s)\notag\\
 			&=\int_0^1(1-p)^{j-i-1}p  \frac{p^{\alpha+s-1}(1-p)^{\beta+i-s-1}}
      			  {B(\alpha+s,\beta+i-s)}\,dp\notag\\
			 &=\frac{B(\alpha+s+1,\beta+j-s-1)}   {B(\alpha+s,\beta+i-s)}.
 		\label{betatransition}
	\end{align}
The alternatives ``no further success'' and ``the next success occurs at
	$j$'' form a partition, and hence $g_i(s)+\sum_{j=i+1}^NP_{ij}(s)=1$. 

Let $V_i^{[1]}(s)$ denote the optimal winning probability when the
	current embedded state is $(i,s)$.  
Backward induction gives the exact dynamic programming equation
	\begin{align}
 		V_N^{[1]}(s)&=1,\notag\\
 		V_i^{[1]}(s)&=\max\left\{  g_i(s),\  \sum_{j=i+1}^NP_{ij}(s)V_j^{[1]}(s+1) \right\},
 		\qquad 1\le s\le i<N.
		 \label{dpvis}
	\end{align}
Although the optimal stopping rule can be characterized exactly by
	backward induction, it does not, in general, admit a closed-form
	representation as a deterministic threshold rule.  
More precisely, the optimal stopping time is the first success epoch at which the immediate
	stopping reward dominates the state-dependent continuation value.  
Since this comparison depends on both the current time and the number of
	successes observed so far, the stopping region is generally a
	two-dimensional subset of the state space rather than a one-dimensional
	time-threshold region.

Before any observation, the optimal winning probability is
	\begin{align*}
		 		 V_0^{[1]}:=\sum_{j=1}^N \frac{B(\alpha+1,\beta+j-1)}{B(\alpha,\beta)}
   		V_j^{[1]}(1),
	\end{align*}
	where the coefficient is the probability that the first success occurs at time $j$. 
To examine monotonicity, define the one-step look-ahead difference on the
	embedded chain by
	\begin{align*}
		 G_i^{[1]}(s)&:=g_i(s)-\sum_{j=i+1}^NP_{ij}(s)g_j(s+1).
	\end{align*}
Using \eqref{betareward} and \eqref{betatransition}, each term in the sum
	satisfies
	\begin{align*}
 		P_{ij}(s)g_j(s+1)=\frac{B(\alpha+s+1,\beta+N-s-1)}{B(\alpha+s,\beta+i-s)},
	\end{align*}
	which is independent of $j$.  
Therefore
	\begin{align*}
 		G_i^{[1]}(s) &=g_i(s)\bigl(1-H_i(s)\bigr), \\
 		H_i(s)&:=\frac{(\alpha+s)(N-i)}{\beta+N-s-1}.
	\end{align*}
The one-step look-ahead rule stops precisely when $H_i(s)\le1$. 
For the embedded stopping problem to be monotone, whenever
	$H_i(s)\le1$, the same inequality must hold at every possible subsequent
	success state $(j,s+1)$, $j>i$.  
Since $H_j(s+1)$ is decreasing in $j$, this requirement is equivalent to
	\begin{align}
 		H_i(s)\le1 \quad\Longrightarrow\quad H_{i+1}(s+1)\le1, \qquad 1\le s\le i<N.
		 \label{betamonotonecondition}
	\end{align}
Condition \eqref{betamonotonecondition} does not hold in general.  
For example, if $N=10$ and $\alpha=\beta=1$, then
 	$H_6(1)=8/9<1$ and $H_7(2)=9/8>1$.
Both successive states occur with positive probability, so the stopping
	set of the embedded one-step problem is not absorbing.  
Hence the model is not a monotone stopping problem in the sense of 
	Section~\ref{subsec:monotone-multiple}.

Moreover, failure of monotonicity can make the one-step look-ahead rule
	strictly suboptimal.  
Take $N=4$, $\alpha=1$ and $\beta=4$.  
At state $(1,1)$,
	$H_1(1)=1$ and $g_1(1)=5/14$, 
	so the one-step rule, with the stopping convention at equality, stops.
The exact backward recursion, however, yields
	$V_2^{[1]}(2)=3/7$, $V_3^{[1]}(2)=5/8$ and $V_4^{[1]}(2)=1$, 
	and hence the continuation value at $(1,1)$ is
	\begin{align*}
		 \sum_{j=2}^4P_{1j}(1)V_j^{[1]}(2) &=\frac13\frac37+\frac4{21}\frac58+\frac5{42}\cdot1
 		=\frac8{21} >\frac5{14}.
	\end{align*}
Thus the optimal action is to continue, whereas the one-step look-ahead rule stops.
In this example the one-step comparison is an exact tie
	($H_1(1)=1$), so the suboptimality is tied to the stopping convention
	at equality.
Strict examples also exist: for $N=11$, $\alpha=1$ and $\beta=6$, 
	the state $(4,1)$ has
 	$H_4(1)=14/15<1$ and $g_4(1)=45/136\approx 0.3309$,
	so the one-step rule stops with a strict inequality, while the exact
	backward recursion \eqref{dpvis} gives the continuation value
	$11559/34034\approx0.3396>g_4(1)$.

For completeness, the Bayesian model with $\ell$ selection opportunities
	can still be solved by general dynamic programming, although the monotone
	multiple-stopping theorem is not applicable.  Set
	$V_i^{[0]}(s)=0$.  For $\ell\ge1$, define the value of selecting the
	current success and the value of rejecting it, respectively, by
	\begin{align}
		a_i^{[\ell]}(s)&:=g_i(s)+\sum_{j=i+1}^NP_{ij}(s)V_j^{[\ell-1]}(s+1),\notag\\
 		c_i^{[\ell]}(s)&:=\sum_{j=i+1}^NP_{ij}(s)V_j^{[\ell]}(s+1).
 	\label{betamultipleactions}
	\end{align}
Then
	\begin{align}
	 	V_N^{[\ell]}(s)&=1,\qquad \ell\ge1,\notag\\
		V_i^{[\ell]}(s)&=\max\{a_i^{[\ell]}(s),c_i^{[\ell]}(s)\},\qquad 1\le s\le i<N.
	 		\label{betamultipleDP}
	\end{align}
The optimal multiple selection times are obtained recursively by selecting
	a success whenever $a_i^{[\ell]}(S_i)\ge c_i^{[\ell]}(S_i)$, where
	$\ell$ is the number of remaining selection opportunities.  Equations
	\eqref{betamultipleactions}--\eqref{betamultipleDP} provide an exact finite
	algorithm, but no deterministic calendar-time threshold or monotone
	one-step characterization follows from the theory of 
	Section~\ref{subsec:monotone-multiple}.

\section*{Acknowledgments}
The author thanks Professors Tomomi Matsui, Naoto Miyoshi, and Aiko Kurushima 
	for fruitful collaborations and many valuable discussions over the years. 


\bigskip
\noindent Katsunori Ano\\
Department of Human-centered Data Science\\
Bunkyo Gakuin University\\
Tokyo, Japan\\
E-mail: \texttt{k-ano@bgu.ac.jp}

\end{document}